\documentclass[11pt, leqno]{amsart}  

\usepackage{graphicx}
\usepackage{amssymb,amsmath,amsthm,amscd,accents}
\usepackage{mathrsfs}
\usepackage{lscape}
\usepackage{enumerate}
\usepackage{upgreek}
\usepackage[usenames,dvipsnames]{color}
\usepackage[colorlinks=true, pdfstartview=FitV,
 linkcolor=blue,citecolor=blue,urlcolor=blue]{hyperref}
\usepackage{tikz}
\usetikzlibrary{arrows.meta,arrows}
\usepackage{bm}
\usepackage{marginnote}
\usepackage{verbatim}
\usepackage[normalem]{ulem}  

\usepackage[all]{xy}

\allowdisplaybreaks[3]

\newcommand{\arxiv}[1]{\href{http://arxiv.org/abs/#1}{\texttt{arXiv:#1}}}

\newcommand{\nc}{\newcommand}
\numberwithin{equation}{section}

\nc{\hs}{\hspace*}
\nc{\ms}{\mspace}

\nc{\qR}[1]{\ttq_{\mspace{-2mu}\raisebox{-.8ex}{${\scriptstyle{#1}}$}}}

\theoremstyle{plain}
\newtheorem{lemma}{Lemma}[section]
\newtheorem{proposition}[lemma]{Proposition}
\newtheorem{theorem}[lemma]{Theorem}
\newtheorem{algorithm}[lemma]{Algorithm}

\newtheorem{corollary}[lemma]{Corollary}
\newtheorem{conjecture}{Conjecture}
\newtheorem{convention}{Convention}

\newtheorem{Conj}{Conjecture}

\theoremstyle{definition}
\newtheorem{remark}[lemma]{Remark}
\newtheorem{example}[lemma]{Example}

\newtheorem{definition}[lemma]{Definition}

\newcommand{\leN}{ \preceq_{\mspace{-2mu}\raisebox{-.5ex}{\scalebox{.6}{${\rm N}$}}} }
\newcommand{\lN}{ \prec_{\mspace{-2mu}\raisebox{-.5ex}{\scalebox{.6}{${\rm N}$}}} }

\renewcommand{\le}{\leqslant}
\renewcommand{\ge}{\geqslant}
\renewcommand{\preceq}{\preccurlyeq}

\newcommand{\st}{\mathop{\mbox{\normalsize$*$}}\limits}

\newcommand{\hconv}{\mathbin{\scalebox{.9}{$\nabla$}}}
\newcommand{\sconv}{\mathbin{\scalebox{.9}{$\Delta$}}}

\newcommand{\seteq}{\mathbin{:=}}
\newcommand{\conv}{\mathop{\mathbin{\mbox{\large $\circ$}}}}
\newcommand{\soplus}{\mathop{\mbox{\normalsize$\bigoplus$}}\limits}
\newcommand{\ev}{{\operatorname{ev}}}

\newcommand{\lxi}{{  {}^\xi  } \hspace{-.2ex}}
\newcommand{\lxip}{{  {}^{\xi^\prime}  } \hspace{-.2ex}}

\newcommand{\lhii}{{  {}^\hii  } \hspace{-.2ex}}

\newcommand{\gmod}{\text{-}\mathrm{gmod}}

\newcommand{\K}{\sfK}
\newcommand{\ex}{\mathrm{ex}}
\newcommand{\fr}{\mathrm{fr}}
\newcommand{\Kex}{\K_\ex}
\newcommand{\Kfr}{\K_\fr}

\newcommand{\g}{\mathfrak{g}}

\newcommand{\n}{\mathfrak{n}}

\newcommand{\C}{\mathbb{C}}
\newcommand{\Q}{\mathbb{Q}}
\newcommand{\Z}{\mathbb{Z}\ms{1mu}}
\newcommand{\N}{\mathbb{N}\ms{1mu}}

\newcommand{\al}{{\ms{1mu}\alpha}}
\newcommand{\ep}{\epsilon}
\newcommand{\la}{\lambda}
\newcommand{\be}{{\ms{1mu}\beta}}
\newcommand{\ga}{\gamma}

\newcommand{\La}{\Lambda}

\newcommand{\wt}{{\rm wt}}
\newcommand{\fin}{{{\rm fin}}}
\newcommand{\up}{{\rm up}}

\newcommand{\de}{\mathfrak{d}}
\newcommand{\Dynkin}{\triangle}

\newcommand{\im}{\imath}

\newcommand{\ii}{ \textbf{\textit{i}}}
\newcommand{\TT}{ \textbf{\textit{T}}}

\newcommand{\jj}{ \textbf{\textit{j}}}

\newcommand*\reced[1]{ \fontsize{8}{8}\selectfont \tikz[baseline=(char.base)]{
 \node[shape=rectangle,draw,inner sep=1.4pt] (char) {#1};} \fontsize{12}{12}\selectfont }

\newcommand*\circled[1]{ \fontsize{6}{6}\selectfont \tikz[baseline=(char.base)]{
  \node[shape=circle,draw,inner sep=0.4pt] (char) {#1};} \fontsize{12}{12}\selectfont }

\newcounter{myc}

\newcommand{\Hom}{\operatorname{Hom}}

\newcommand{\HOM}{\mathrm{H{\scriptstyle OM}}}

\newcommand{\tSigma}{\widetilde{\Sigma}}
\newcommand{\tUpsigma}{\widetilde{\Upsigma}}

\newcommand{\tka}{\widetilde{\kappa}}

\newcommand{\tscrB}{\widetilde{\scrB}}

\newcommand{\tPsi}{\widetilde{\Psi}}
\newcommand{\tbrho}{\widetilde{\brho}}

\newcommand{\tii}{{\widetilde{\ii}}}

\newcommand{\tx}{\widetilde{x}}

\newcommand{\tB}{\widetilde{B}}

\newcommand{\tbfB}{\widetilde{\bfB}}
\newcommand{\tusfB}{\ms{1mu}\widetilde{\usfB}}
\newcommand{\tusfb}{\ms{1mu}\widetilde{\usfb}}

\newcommand{\tX}{\widetilde{X}}
\newcommand{\tF}{\widetilde{F}}

\newcommand{\tG}{\widetilde{G}}

\newcommand{\tm}{\widetilde{m}}

\newcommand{\tLa}{\widetilde{\Lambda}}
\newcommand{\tGamma}{\widetilde{\Gamma}}

\newcommand{\tD}{\widetilde{D}}

\newcommand{\hga}{\widehat{\ga}}
\newcommand{\heta}{\widehat{\eta}}
\newcommand{\htau}{\widehat{\tau}}
\newcommand{\hzeta}{\widehat{\zeta}}
\newcommand{\hii}{{\widehat{\ii}}}
\newcommand{\hjj}{{\widehat{\jj}}}
\newcommand{\hbe}{\widehat{\be}}

\newcommand{\usfB}{\underline{\mathsf{B}}}

\newcommand{\usfC}{\underline{\mathsf{C}}}

\newcommand{\uscrB}{{\underline{\scrB}}}

\newcommand{\ucalN}{\ms{1mu}\underline{\calN}}

\newcommand{\usfb}{{\underline{\sfb}}}
\newcommand{\usfc}{\underline{\sfc}}

\newcommand{\um}{\underline{m}}

\newcommand{\frakD}{\mathfrak{D}}
\newcommand{\frakQ}{\mathfrak{Q}}

\newcommand{\frakK}{\mathfrak{K}}
\newcommand{\frakX}{\mathfrak{X}}

\newcommand{\frakS}{\mathfrak{S}}

\newcommand{\sfC}{\mathsf{C}}

\newcommand{\sfc}{\mathsf{c}}

\newcommand{\sfE}{\mathsf{E}}
\newcommand{\sfF}{\mathsf{F}}

\newcommand{\sfL}{\mathsf{L}}

\newcommand{\sfD}{\mathsf{D}}
\newcommand{\sfP}{\mathsf{P}}
\newcommand{\sfQ}{\mathsf{Q}}
\newcommand{\sfW}{\mathsf{W}}
\newcommand{\sfp}{\mathsf{p}}
\newcommand{\sfh}{\mathsf{h}}

\newcommand{\sfT}{\mathsf{T}}

\newcommand{\sfX}{\mathsf{X}}

\newcommand{\sfz}{\mathsf{z}}

\newcommand{\sfb}{\mathsf{b}}

\newcommand{\sfd}{\mathsf{d}}

\newcommand{\sfK}{\mathsf{K}}

\newcommand{\bbA}{\mathbb{A}}

\newcommand{\bbF}{\mathbb{F}}

\newcommand{\bsg}{{\boldsymbol{g}}}
\newcommand{\bse}{{\boldsymbol{e}}}
\newcommand{\bsb}{{\boldsymbol{b}}}
\newcommand{\bsn}{{\boldsymbol{n}}}
\newcommand{\bsa}{{\boldsymbol{a}}}
\newcommand{\bsc}{{\boldsymbol{c}}}

\newcommand{\brho}{{\boldsymbol{\rho}}}
\newcommand{\bstau}{{\boldsymbol{\tau}}}

\newcommand{\bfk}{\mathbf{k}}
\newcommand{\bfB}{\mathbf{B}}
\newcommand{\bfF}{\mathbf{F}}
\newcommand{\bfE}{\mathbf{E}}
\newcommand{\bfL}{\mathbf{L}}

\newcommand{\bfg}{\mathbf{g}}
\newcommand{\bfc}{\mathbf{c}}

\newcommand{\bfa}{\mathbf{a}}
\newcommand{\bfn}{\mathbf{n}}

\newcommand{\bfdeg}{\mathbf{deg}}

\newcommand{\calQ}{\mathcal{Q}}

\newcommand{\calU}{\mathcal{U}}
\newcommand{\calA}{\mathcal{A}}
\newcommand{\calS}{\mathcal{S}}
\newcommand{\calF}{\mathcal{F}}

\newcommand{\calK}{\mathcal{K}}
\newcommand{\calR}{\mathcal{R}}

\newcommand{\calX}{\mathcal{X}}

\newcommand{\calM}{\mathcal{M}}
\newcommand{\calN}{\mathcal{N}}

\newcommand{\calT}{\mathcal{T}}
\newcommand{\calW}{\mathcal{W}}
\newcommand{\calL}{\mathcal{L}}

\newcommand{\calZ}{\mathcal{Z}}

\newcommand{\scrC}{\mathscr{C}}
\newcommand{\scrA}{\mathscr{A}}
\newcommand{\scrB}{\mathscr{B}}

\newcommand{\ttq}{{\ms{1mu}\mathtt{q}\ms{1mu}}}

\newlength{\mylength}
\newcommand*{\para}{%
  \rlap{\rotatebox{-30}{\rule[.05ex]{.4pt}{.77em}}}%
  \kern.04em%
  \rlap{\kern.36em\raisebox{0.649519052835em}{\rule{.6em}{.4pt}}}%
  \rule{.6em}{.4pt}\kern-.04em%
  \rotatebox{-30}{\rule[.05ex]{.4pt}{.77em}}}

\newcommand{\Rr}{\mathbf{r}}

\newcommand{\hDynkin}{{\widehat{\Dynkin}}}

\newcommand{\tDynkin}{{\widetilde{\Dynkin}}}

\newcommand{\diag}{\mathrm{diag}}

\newcommand{\wl}{\sfP}
\newcommand{\rl}{\sfQ}

\newcommand{\weyl}{\sfW}
\newcommand{\lan}{\langle}
\newcommand{\ran}{\rangle}

\newcommand{\isoto}[1][]{\mathop{\xrightarrow%
[{\raisebox{.3ex}[0ex][.3ex]{$\scriptstyle{#1}$}}]%
{{\raisebox{-.6ex}[0ex][-.6ex]{$\mspace{2mu}\sim\mspace{2mu}$}}}}}

\newcommand{\ee}{\end{enumerate}}
\newcommand{\bitem}{\begin{itemize}}
\newcommand{\eitem}{\end{itemize}}
\newcommand{\ben}{\begin{enumerate}[{\rm (1)}]}
\newcommand{\bnum}{\begin{enumerate}[{\rm (i)}]}
\newcommand{\bnump}{\begin{enumerate}[{\rm (i)$'$}]}
\newcommand{\bna}{\begin{enumerate}[{\rm (a)}]}
\newcommand{\bnA}{\begin{enumerate}[{\rm (A)}]}
\newcommand{\bc}{\begin{cases}}
\newcommand{\ec}{\end{cases}}

\newenvironment{myequation}
{\relax\setlength{\arraycolsep}{1pt}\begin{eqnarray}}
{\end{eqnarray}}
\newenvironment{myequationn}
{\relax\setlength{\arraycolsep}{1pt}\begin{eqnarray*}}
{\end{eqnarray*}}

\nc{\eq}{\begin{myequation}}
\nc{\eneq}{\end{myequation}}
\nc{\eqn}{\begin{myequationn}}
\nc{\eneqn}{\end{myequationn}}
\nc{\cl}{\colon}
\nc{\ake}[1][1ex]{\rule[-#1]{0ex}{1ex}}
\nc{\akew}[1][1ex]{\rule[-1ex]{#1}{0ex}}
\nc{\akeu}[1][1ex]{\rule[#1]{0ex}{1ex}}
\nc{\id}{\mathrm{id}}
\nc{\bl}{\bigl(}
\nc{\br}{\bigr)}
\nc{\qt}[1]{\quad\text{#1}}
\nc{\qtq}[1][{and}]{\quad\text{{#1}}\quad}
\nc{\eqs}[1]{\underset{\raisebox{.4ex}[.7ex][0ex]{$\scriptstyle{#1}$}}{=}}
\nc{\snoi}{\smallskip \noindent}
\nc{\mnoi}{\medskip \noindent}
\nc{\Mat}{\mathrm{Mat}}
\nc{\ol}{\overline}
\nc{\ul}{\underline}
\nc{\ang}[1]{\boldsymbol{\langle}{#1}\boldsymbol{\rangle}}
\nc{\rang}[1]{\boldsymbol{\langle}{#1}\boldsymbol{\rangle} \hspace{-.6ex} \boldsymbol{\rangle}}
\nc{\ba}{\begin{array}}
\nc{\ea}{\end{array}}
\nc{\noi}{\noindent}
\nc{\evq}{ \ev_{q=1}}

\usepackage[colorinlistoftodos]{todonotes}

\title[Quantum cluster algebra and quantum virtual Grothendieck ring]{Quantum cluster algebra, braid moves and \\ quantum virtual Grothendieck ring}

\author[K.-H. Lee]{Kyu-Hwan Lee$^{\star}$}
\thanks{$^{\star}$ K.-H. Lee was partially supported by a grant from the Simons Foundation (\#712100).}
\address[K.-H. Lee]{Department of Mathematics, University of Connecticut, Storrs, CT 06269, U.S.A.}
\email{khlee@math.uconn.edu}
\urladdr{https://www.math.uconn.edu/~khlee/}

\author[S.-j.~Oh]{Se-jin Oh$^{\dagger}$}
\address[S.-j.~Oh]{Department of Mathematics, Sungkyunkwan University, Suwon, South Korea}
\email{sejin092@gmail.com}
\urladdr{https://sites.google.com/site/mathsejinoh/}
\thanks{$^{\dagger}$ S.-j.\ Oh was supported by the Ministry of Education of the Republic of Korea and the National Research Foundation of Korea (NRF-2022R1A2C1004045).}

\date{\today}

\begin{document}

\begin{abstract}
In this paper, we study the quantum virtual Grothendieck ring, denoted by $\frakK_q(\g)$, which was introduced in \cite{KO22}, and further investigated in \cite{JLO1, JLO2}. Our 
approach involves examining this ring from two perspectives: first, by considering its connection to quantum cluster algebras of non-skew-symmetric types; and second, by exploring 
its relevance to categorification theory. We specifically focus on (i) the homomorphisms that arise from braid moves, particularly 4-moves and 6-moves, in the braid group; and 
(ii) the quantum Laurent positivity phenomena, which have not yet been proven for non-skew-symmetric types.

As applications of our results, we derive the substitution formulas for non-skew-symmetric types discussed in \cite{FHOO2} for skew-symmetric types, and  demonstrate that any truncated element in a heart subring, denoted by $\frakK_{q,Q}(\g)$, which corresponds to a simple module over the quiver 
Hecke algebra $R^\g$, 
possesses coefficients in $\Z_{\ge 0}[q^{\pm 1/2}]$. This result is particularly interesting because it implies that each truncated Kirillov--Reshetikhin 
polynomial in $\frakK_{q,Q}(\g)$ and each element in the standard basis $\sfE_q(\g)$ of the entire ring $\frakK_q(\g)$ have coefficients also in $\Z_{\ge 0}[q^{\pm 1/2}]$.
 Since (truncated) Kirillov--Reshetikhin polynomials can be obtained using a quantum cluster algebra algorithm and appear as quantum cluster variables, 
they provide compelling evidence in support of the quantum Laurent positivity conjecture in non-skew-symmetric types.
\end{abstract}

\setcounter{tocdepth}{1}

\maketitle
\tableofcontents

\section{Introduction}

Let $\g$ be a complex finite-dimensional simple Lie algebra, which has an index set $I$ and a Cartan matrix $\sfC$. The quantum loop algebra $U_q(\calL\g)$ associated with 
$\g$ has been extensively studied since its emergence. This is mainly due to the fact that its representation theory exhibits a rich structure characterized by monoidality 
and rigidity. Although the category $\scrC_\g$ of type $\mathbf{1}$-representations over $U_q(\calL\g)$ is not braided, it is well-known that the Grothendieck ring $K(\scrC_\g)$ 
is commutative.

Motivated by $\calW$-algebras, Frenkel--Reshetikhin~\cite{FR99} constructed the \emph{$q$-character} monomorphism 
$\chi_q: K(\scrC_\g) \hookrightarrow \Z[Y_{i,w}^{\pm 1}]_{i\in I,w \in \C(q)^\times}$ which also explains the commutativity of 
$K(\scrC_\g) \simeq \chi_q(K(\scrC_\g)) =:\calK(\scrC_\g)$.  
Nakajima \cite{Nak04} and Varagnolo--Vasserot \cite{VV03} then employed the geometry of quiver varieties to construct the non-commutative $t$-quantization, denoted as $\calK_t(\scrC_\bfg)$, of $\calK(\scrC_\bfg)$ for a simply-laced type $\bfg$. This construction is known as the \emph{quantum Grothendieck ring} and highlights the non-braided property of $\scrC_\bfg$. Moreover, $\calK_t(\scrC_\bfg)$ recovers $\calK(\scrC_\bfg)$ when $t=1$.
 
The quantum Grothendieck ring admits a remarkable Kazhdan--Lusztig algorithm for $\scrC_\bfg$, offering insights into the Jordan--H\"{o}lder multiplicity $P_{m,m'}$ of a simple module $L(m')$ in a standard module $E(m)$ in $K(\scrC_\bfg)$. Nakajima established this connection through the formulation of the Kazhdan--Lusztig type polynomial $P_{m,m'}(t)$ in $\calK_t(\scrC_\bfg)$, with its specialization at $t=1$ yielding the desired $P_{m,m'}$.

To be more specific, he constructed bases $\bfE_t(\bfg)=\{ E_t(m) \}$ and $\bfL_t(\bfg)=\{ L_t(m) \}$ for $\calK_t(\scrC_\bfg)$, consisting of \emph{$(q,t)$-characters} of standard modules and simple modules, respectively. He proved the uni-triangular property that exists between these two bases:
\begin{align} \label{eq: Et and Lt}
E_t(m) = L_t(m) + \sum_{m'; m' <_N m} P_{m,m'}(t) L_t(m') \quad \text{ for some } P_{m,m'}(t) \in t\Z[t] 
\end{align}
in $\calK_t(\scrC_\bfg)$. We call $\bfE_t(\bfg)$ the \emph{standard basis} and $\bfL_t(\bfg)$ the \emph{canonical basis}. 

Taking a significant stride forward, Nakajima went on to establish that every coefficient polynomial of $L_t(m)$, $E_t(m)$, and the polynomial $P_{m,m'}(t)$ resides in $\Z_{\ge 0}[t^{\pm 1/2}]$. Moreover, he proved that the specializations of $L_t(m)$ and $E_t(m)$ at $t=1$ yield the $q$-character $\chi_q(L(m))$ for the simple module $L(m)$ and the $q$-character $\chi_q(E(m))$ for the standard module $E(m)$ in $\scrC_\bfg$, respectively. These findings collectively underscore the profound nature of the algorithm he unveiled.  

\smallskip
 
In cases where $\g$ is non-simply-laced, a fully-developed geometry is lacking, rendering the direct application of the framework for $\bfg$ impractical. 
To address this, Hernandez 
\cite{Her04} introduced a conjectural Kazhdan--Lusztig algorithm for all $\g$, using the \emph{quantum Cartan matrix} $\sfC(q)$ associated with $\g$. He devised 
a non-commutative $t$-quantization, also denoted as $\calK_t(\g)$, of $\calK(\g)$ and constructed a basis $\bfF_q(\g)$ for $\calK_t(\g)$ through a special
algorithm known as the \emph{$t$-algorithm}. It is noteworthy that the $t$-algorithm can be viewed as a $t$-quantization of the Frenkel--Mukhin algorithm 
in \cite{FM01}, utilizing $\sfC(q)$.

Subsequently, Hernandez constructed bases $\bfE_t(\g)$ and $\bfL_t(\g)$ for $\calK_t(\g)$, adhering to~\eqref{eq: Et and Lt}, and expected them to also exhibit 
the same positivity as $L_t(m)$, $E_t(m)$, and $P_{m,m'}(t)$. Particularly, in~\cite{Her05}, he showed that the $(q,t)$-character $L_t(Y_{i,w})$ of a 
\emph{fundamental representation} $L(Y_{i,w})$  manifests the desired \emph{quantum positivity}. This implies that every coefficient polynomial of   $L_t(Y_{i,w})$   
resides in $\Z_{\ge 0}[t^{\pm 1/2}]$, consequently establishing the quantum positivity of elements in $\bfE_t(\g)$ as well. 

\smallskip

 The quantum cluster algebra $\scrA_q$, introduced by Berenstein--Fomin--Zelevinsky in~\cite{FZ02,BZ05}, serves as a crucial framework for the investigation 
 of the dual-canonical/upper-global basis $\bfB^\up$ introduced by Lusztig/Kashiwara, through a combinatorial point of view. The algebra $\scrA_q$ is a 
 non-commutative $\Z[q^{\pm 1/2}]$-subalgebra of the skew field of $\Q(q^{1/2})(Z_i)_{i \in \sfK}$, generated by cluster variables. These variables are derived 
 from  the initial cluster $\{ Z_i \}_{i \in \sfK}$ through a sequence of quantum exchange relations known as \emph{mutations}. Despite the non-trivial nature 
 of these mutations, the \emph{quantum Laurent phenomenon}, proven 
 in~\cite{FZ02,BZ05}, assures that $\scrA_q$ remains contained in the (non-commutative) Laurent polynomial ring $\Z[q^{\pm 1/2}][Z_{i}^{\pm 1}]_{i \in \sfK}$.

Another interesting feature of $\mathscr{A}_q$, which has been confirmed for skew-symmetric cases (see~\cite{Dav18}), is the 
\emph{quantum Laurent positivity conjecture}. This conjecture asserts that every cluster variable indeed belongs to $\Z_{\ge 0}[q^{\pm 1/2}][Z_{i}^{\pm 1}]_{i \in \sfK}$.

\smallskip

Building upon the earlier discussion, it is worthwhile to emphasize that the algebras $\calK_t(\g)$ and $\scrA_q$ exhibit similar quantum Laurent positivity phenomena. 
The connection between these two structures is further underscored by the notable works of Hernandez--Leclerc in \cite{HL10,HL15,HL16}. 
This relationship gains additional support from the results 
in~\cite{B21,KKK2,KO19,OS19,KKOP20,KKOP2,FHOO}, which collectively establish that $\calK_t(\g)$ possesses a quantum cluster algebra structure.

\smallskip

In a recent development, Fujita, Hernandez, Oya, and the second named author of this paper have successfully resolved the quantum positivity conjectures for $L_t(m)$ and 
$P_{m,m'}(t)$ in $\calK_t(\g)$ for non-simply-laced Lie algebras $\g$. This achievement was made possible through the utilization of the quantum cluster algebra 
structure of $\calK_t(\g)$ and the application of a technique known as the \emph{propagation of positivity} \cite{FHOO,FHOO2}.  Let us explain their results more precisely.

For each specific $\g$, we define $\bfg_\fin$ of simply-laced type as follows: 
$$
(\g,\bfg_\fin) =(B_n,A_{2n-1}), \ (C_n,D_{n+1}), \  (F_4,E_6), \  (G_2,D_{4}) \ \text { and } \ \g=\bfg_\fin \ \ \text{otherwise}.
$$
In their work~\cite{FHOO}, it was established that the algebras $\calK_t(\scrC^0_{\g})$ and $\calK_t(\scrC^0_{\bfg_\fin})$ are isomorphic, with their 
presentations characterized as a \emph{boson-extension} of the negative half of the quantum group $U^-_q(\bfg_{\fin})$ (refer also to~\cite[(1.7)]{JLO2}). 
Here, $\scrC^0_{\g}$ denotes the \emph{skeleton category}, also known as the \emph{Hernandez--Leclerc} subcategory. Taking a further step in \cite{FHOO2}, they proved that the isomorphisms between $\calK_t(\scrC^0_{\g})$ and 
$\calK_t(\scrC^0_{\bfg_\fin})$ can be interpreted as sequences of mutations-permutations between \emph{adapted} sequences of indices. Moreover, they extended  these isomorphisms to the skew fields of their quantum tori, introducing the \emph{substitution formula}.

The implications of these isomorphisms are far-reaching, leading to several applications:
\bna
\item \textbf{Propagating positivity}: Using the isomorphisms to propagate positivity from type $\bfg_\fin$ to type $\g$, 
they proved the quantum positivity conjectures for $L_t(m)$ and $P_{m,m'}(t)$ in non-simply-laced type $\g$.
\item \textbf{Recovery of $q$-character}: Using the monoidal categorification result presented in~\cite{KKOP2} and the isomorphisms, they proved that 
$L_t(m)$ recovers the $q$-character of a simple module $L(m)$ at $q=1$, provided it is \emph{reachable}.
\item \textbf{Interplay between categories}: The substitution formula revealed intriguing behaviors among elements in $\bfL_t(\g)$ and $\bfL_t(\bfg_\fin)$.
\ee

In summary, the algebraic structure of $\calK_t(\scrC^0_{\g})$ exhibits (skew-)symmetry, irrespective of the inherent asymmetry in $\g$. Moreover, the observed quantum positivity 
phenomena of $\calK_t(\scrC^0_{\g})$ are connected to those witnessed in the context of the \emph{skew-symmetric} quantum cluster algebra $\scrA_q$. 
 
\smallskip

In light of the aforementioned observations, Kashiwara and the second named author of this paper introduced what is now called the \emph{quantum virtual Grothendieck ring} $\frakK_q(\g)$: a distinguished subring of a non-commutative quantum torus over $\bbA \seteq \Z[q^{\pm 1/2}]$, generated by the variables $\tX_{i,p}$ associated with $\g$,
and characterized as the common kernel of $t$-deformed screening operators $S_{i,t}$ (see Definition~\ref{def: qvgr} for more details).
Note that the construction of $\frakK_q(\g)$ can be seen as a $\sfC(t)$-analogue of Hernandez's construction of $\calK_t(\scrC^0_\g)$, where $\sfC(t)$ is another quantization of $\sfC$ known as the \emph{$t$-quantized Cartan matrix}.

The main results established in~\cite{KO22,JLO1} include:
\bnum
\item \textbf{Quantum cluster algebra structures}: The algebra $\frakK_q(\g)$, along with various subrings, possesses quantum cluster algebra structures. Importantly, the (a)symmetries of these structures coincide with those of the underlying Lie algebra $\g$.
\item \textbf{Bases and relations}: Bases $\sfF_q(\g)$, $\sfE_q(\g)$, and $\sfL_q(\g)$ are identified in $\frakK_q(\g)$, and they exhibit similar relations to those in $\calK_t(\g)$, including~\eqref{eq: Et and Lt}. Notably, $\sfE_q(\g)$ is termed the \emph{standard basis}, and $\sfL_q(\g)$ the \emph{canonical basis}, mirroring the naming in the $\calK_t(\g)$ cases.
\ee

\begin{Conj}[\cite{JLO1}]
\bna
\item \label{it: conj1} 
In the case where $L_q(m)$ is real, meaning that $L_q(m)^2 = q^{a/2} L_q(m^2)$ for some $a \in \Z$, every coefficient polynomial of $L_q(m)$ resides in $\Z_{\ge 0}[q^{\pm 1/2}]$. 
\item \label{it: conj2} For a \emph{Kirillov--Reshetikhin (KR)} monomial $m \in \frakK_q(\g)$, the element $L_q(m)$ in $\sfL_q(\g)$ coincides with $F_q(m)$ in $\sfF_q(\g)$ consisting of real elements, where $F_q(m)$ is a KR-polynomial $($see {\rm \S~\ref{subsec: KR}} for the definition$)$.
\ee
In particular, combining~\eqref{it: conj1} and~\eqref{it: conj2}, we expect that each coefficient polynomial of a KR-polynomial $F_q(m)$ is 
in $\Z_{\ge0}[q^{\pm 1/2}]$.
\end{Conj}

We remark here that each basis element $F_q(m)$ in $\sfF_q(\g)$ can be obtained from the $q$-algorithm,
which is a $\sfC(t)$-analogue of the $t$-algorithm and is defined inductively. Thus determining whether the coefficient polynomials of $F_q(m)$ are in $\Z_{\ge0}[q^{\pm 1/2}]$ or not
is quite intricate.  

\bigskip

In this paper, we further study  the quantum virtual Grothendieck ring $\frakK_q(\g)$, utilizing its quantum cluster algebra structure. It is important to note 
that $\frakK_q(\g)$ is generated by a countably infinite set of \emph{fundamental} polynomials, denoted as $\{ F_q(X_{i,p}) \}_{(i,p) \in \hDynkin_0}$ and parameterized by a specific 
index set $\hDynkin_0$ (see~\eqref{eq: hDynkin0}).

For each Dynkin quiver $Q=(\Dynkin,\xi)$ associated with the Lie algebra $\g$ (see~\S~\ref{subsec: Dynkin quiver}),  the \emph{heart} subring $\frakK_{q,Q}(\g)$ of 
$\frakK_q(\g)$ is generated by the fundamental polynomials $\{ F_q(X_{i,p}) \}_{(i,p) \in \Gamma^Q_0 \subset \hDynkin_0}$
(see~\cite{FO21,JLO2}).  The utilization of the heart subrings yields the following homomorphisms~\cite{HL15,FHOO,JLO1}:
\bna
\item \textbf{Isomorphism $\Psi_Q$}:
An $\bbA$-algebra isomorphism, denoted by $\Psi_Q$, is established between the integral form $\calA_{\bbA}(\n)$ of the quantum unipotent coordinate ring $\calA_{q}(\n)$ and the heart subring $\frakK_{q,Q}(\g)$. Importantly, this isomorphism maps the \emph{normalized dual-canonical/upper-global basis} $\tbfB^\up$ to the canonical basis $\sfL_{q,Q}(\g) \seteq \sfL_q(\g) \cap \frakK_{q,Q}(\g)$ of $\frakK_{q,Q}(\g)$.
\item \textbf{Automorphism $\Psi_{\frakQ,\frakQ'}$}:
For any pair of Dynkin quivers $Q=(\Dynkin,\xi)$ and $Q'=(\Dynkin,\xi)$, there exists an automorphism denoted by $\Psi_{\frakQ,\frakQ'}$ of $\frakK_q(\g)$. This automorphism sends the heart subring $\frakK_{q,Q'}(\g)$ to $\frakK_{q,Q}(\g)$ and respects the canonical basis $\sfL_q$.
\ee
Note that the integral form $\calA_\bbA(\n)$ of the quantum unipotent coordinate ring $\calA_{q}(\n)$ is isomorphic to the quantum cluster algebra $\calA_\ii$ (\cite{GLS13,GY17}). The initial seed of $\calA_\ii$ can be constructed through the combinatorics of a reduced sequence $\ii$ associated with the longest element $w_\circ$ of the Weyl group $\weyl$. 

\smallskip
 
As a main result, Theorem~\ref{thm: comm square} establishes, through the use of isomorphisms and the degrees of quantum cluster monomials, that the isomorphism $\Psi_{\xi,\xi'}:\frakK_{q,\xi'}(\g) \isoto \frakK_{q,\xi}(\g)$ obtained by restricting $\Psi_{\frakQ,\frakQ'}$ to the \emph{$\xi'$-negative subring} $\frakK_{q,\xi'}(\g)$ of $\frakK_{q}(\g)$ (refer to Definition~\ref{def: negative subring}) is compatible with sequences of mutations-permutations between the seeds associated with infinite sequences $\hii'$ and $\hii$ in $I^{(\infty)}$ respectively adapted to $Q'$ and $Q$ (see~\eqref{eq: hii}).

The proof of Theorem~\ref{thm: comm square} involves the following steps:
\bna
\item \textbf{Braid moves as mutations-permutations}: The interpretation of braid moves as mutations-permutations is established.
\item \textbf{Analysis of degrees of quantum cluster monomials}: The behaviors of degrees of quantum cluster monomials concerning mutations-permutations corresponding to braid moves are analyzed. This analysis includes the study of $4$-moves for types $B_n,C_n,F_4$ and $6$-moves for type $G_2$, involving more complicated computations compared to the investigations of $2$-moves and $3$-moves presented in~\cite{FHOO2}.
\ee
Through this detailed analysis, it is proved that the isomorphism $\Psi_{\xi,\xi'}$ induces a corresponding bijection between their quantum cluster monomials, establishing the compatibility highlighted in Theorem~\ref{thm: comm square}.

\smallskip

As an application, we extend the automorphisms $\Psi_{\frakQ,\frakQ'}$ to substitution formulas $\tPsi_{\frakQ,\frakQ'}$ for the variables in the fractional skew field $\bbF(\calX_q)$ of $\frakK_q(\g)$ (Theorem~\ref{thm: substituion}). These substitution formulas, derived from mutations-permutations and consequently being automorphisms, involve non-trivial fractions among Laurent polynomials. Independent of the quantum Laurent phenomena, these substitution formulas yield Laurent polynomials in $\frakK_q(\g)$ and respect the canonical basis $\sfL_q$.

Exploring these substitution formulas reveals highly non-trivial relations among elements in $\sfL_q$. For concrete examples of these formulas, particularly for type $B_2$, we refer the reader to Appendix \ref{sec: substitution formula}. The investigation of the substitution formulas offers valuable insights into relationships among elements in the canonical basis $\sfL_q$.

\smallskip

Another main result (Theorem~\ref{thm: quantum positive}) of this paper is a proof of quantum positivity for:
{\rm (a)} fundamental polynomials $F_q(X_{i,p})=L_q(X_{i,p})$,
{\rm (b)} elements in the standard basis $\sfE_q$,
{\rm (c)} truncated KR-polynomials in $(\frakK_{q,Q}(\g))_{\le \xi}$.
The quantum cluster algebra algorithm generates fundamental polynomials $L_q(X_{i,p})$ and truncated KR-polynomials $F_q(m)_{\le \xi}$, starting from their highest dominant monomials $\underline{X_{i,p}}$ and $\underline{m}$ (\cite{HL16,B21,JLO1}). This result serves as robust evidence supporting the quantum positivity conjecture for non-skew-symmetric types.

\smallskip 
 
The main argument of the proof involves the categorification theory of $\calA_\bbA(\n)$ through the category $R\gmod$ of finite-dimensional graded modules over the quiver Hecke algebra $R$ (\cite{KL1,R08}).
Recently, Kashiwara, Kim, Park, and the second named author introduced the concept of a \emph{Laurent family} of simple modules over $R$ and proved that each Laurent family exhibits the quantum Laurent positivity phenomena. This positivity is observed with respect to the basis of $\calA_\bbA(\n)$ that corresponds to the isomorphism classes of simple modules over $R$ through the process of categorification (\cite{KKOP23}).

\smallskip 

The proof of Theorem~\ref{thm: quantum positive} employs the following key ingredients:
\bnum
\item The isomorphisms among $\calK_\bbA(R\gmod)$, $\frakK_{q,Q}(\g)$, and $\calA_\bbA(\n)$ (see~\eqref{eq: isos}).
\item The truncation isomorphism $(\cdot)_{\le \xi}:\frakK_{q,Q}(\g) \isoto (\frakK_{q,Q}(\g))_{\le \xi}$ (see~\eqref{eq: truncation}).
\item The structure of KR-polynomials (see Theorem~\ref{thm: KR-poly}).
\item The fact that KR-polynomials in $\frakK_{q,Q}(\g)$ \emph{can be understood as certain characters} of determinantal modules over $R$ (\cite{KKOP18,JLO2}).
\ee
With these ingredients, we show in Theorem~\ref{thm: quantum positive} that:
\bna
\item The set of KR-polynomials $\{ F_q(m^{(i)}[p,\xi_i]) \ | \ (i,p) \in \Gamma^Q_0 \}$ is the set of \emph{such characters} of the Laurent family 
$\calL\seteq\{ M(w_k^{\ii}\varpi_{i_k},\varpi_{i_k}) \}$ of \emph{determinantal modules} for any reduced sequence $\ii$ of $w_\circ$ \emph{adapted to} $Q=(\Dynkin,\xi)$.
\item Any element $\sfX_{\le \xi} \in (\frakK_{q,Q}(\g))_{\le \xi}$ corresponding to a simple module $X$ over $R$ is quantum positive (see Example~\ref{ex: quantum positive} for instances).
\ee
As special cases, we obtain the desired result (Corollary~\ref{cor: pos cor}) that fundamental polynomials, elements in the standard basis $\sfE_q(\g)$, and truncated KR-polynomials in $(\frakK_{q,Q})(\g)_{\le \xi}$ are quantum positive.

\medskip

This paper is structured as follows:
Section~\ref{sec: Quantum cluster} provides a review of the fundamental concepts in the theory of quantum cluster algebras, covering essential topics such as the degrees of quantum cluster monomials and the valued quiver associated with an exchange matrix.
Section~\ref{sec: Braid and degree} explores the mutations-permutations corresponding to moves in the braid group of $\g$. It contains intricate computations and makes crucial observations concerning the behaviors of degrees of quantum cluster monomials with respect to these mutations-permutations.
 We allocate considerable space to this section because the computations are quite important for later parts.  
Section~\ref{sec: basis and moves} investigates the connection between mutations-permutations corresponding to moves and the basis $\tbfB^{\up}$ in $\calA_{\bbA}(\g)$ through the isomorphism $\calA_\ii \isoto \calA_{\bbA}(\n)$.
In Section~\ref{sec: substitution}, we establish the compatibility between the isomorphism $\Psi_{\xi,\xi'}$ and mutations-permutations corresponding to moves. This compatibility is then applied to the analysis of substitution formulas.
Section~\ref{sec: quiver Hecke and Laurent} reviews the categorification of $\calA_\bbA(\n)$ via $R\gmod$ and discusses the results from~\cite{KKOP23} related to the concept of a Laurent family.
Section~\ref{sec: positivity} presents the proof of the quantum positivity of fundamental polynomials, the standard basis $\sfE_q(\g)$, and truncated KR-polynomials in $(\frakK_{q,Q})(\g)_{\le \xi}$.

\subsection*{Convention} Throughout this paper, we use the following convention.
\begin{enumerate}[\,\,$\bullet$]
\item For a statement $\mathtt{P}$, we set $\delta(\mathtt{P})$ to be $1$ or $0$
depending on whether $\mathtt{P}$ is true or not. In particular, we 
set $\delta_{i,j} \seteq \delta(i=j)$ (the Kronecker's delta). 
\item For a totally ordered set $J = \{ \cdots < j_{-1} < j_{0} < j_{1} < j_{2} \cdots \}$, write
$$\prod_{j \in J}^\to A_j \seteq \cdots A_{j_{-1}}A_{j_{0}}A_{j_{1}}A_{j_{2}} \cdots.$$
\item For $a \in \Z$, we set $[a]_+ \seteq \max(a,0)$.  Note that
$a = [a]_+ - [-a]_+$. 
\item We set $\N \seteq \Z_{\ge1}$ and $\N_0 \seteq \Z_{\ge 0}$.
\item For a set $A$, we denote by $\frakS_A$ the group of permutations of $A$. In particular, if $A = \N$ and $k \in \N$, we write $\sigma_k$ for the simple transposition of $k$ and $k
+1$ in $\frakS_\N$.
\item We set $q$ to be an indeterminate and 
$\bbA \seteq \Z[q^{\pm 1/2}]$, where $q^{1/2}$ denotes the formal square root of
$q$. For an $\bbA$-algebra $\calR$ and elements $X$ and $Y$ in $\calR$, we write $X \cong Y$ if there exists $a \in \Z$ such that $X =q^{a/2}Y$. 
\end{enumerate}

\section{Quantum cluster algebra structure} \label{sec: Quantum cluster}

\subsection{Quantum cluster algebra}
Let $\sfK$ be a (possibly infinite) countable index set with a decomposition 
$\sfK=\sfK_\ex \sqcup \sfK_\fr$. We call $\sfK_\ex$ the set of exchangeable indices and $\sfK_\fr$ the set of frozen indices. Let $\La=(\La_{i,j})_{i,j \in \sfK}$
be a skew-symmetric $\Z$-valued matrix.

\begin{definition}
We denote by $\calT(\La)$ the $\bbA$-algebra 
generated by $\{ Z_{k}^{\pm1} \}_{k \in \sfK}$ subject to the following defining
relations:
\begin{align*}
Z_jZ_j^{-1} = Z_j^{-1}Z_j =1   \quad \text{and} \quad  Z_iZ_j = q^{\La_{i,j}}Z_jZ_i \quad \text{for } i,j \in \sfK. 
\end{align*}
\end{definition}
We define the $\Z$-algebra anti-involution $\overline{(\cdot)}$ of $\calT(\La)$, called the \emph{bar-involution}, 
by 
\begin{align} \label{eq: bar involution in quantum cluster}
\overline{q^{1/2}} \seteq q^{-1/2}, \qquad \overline{Z_k} \seteq Z_k,
\end{align}
for all $k \in \sfK$. 
For $\bfa =(\bfa_i)_{i \in \sfK} \in \Z^{\oplus \sfK}$, we define the element 
$Z^\bfa$ of $\calT(\La)$ as  
$$
Z^\bfa \seteq q^{ \frac 1 2  \sum_{i > j} \bfa_i\bfa_j \La_{i,j} }
\prod^{\to}_{i \in \sfK} Z_i^{\bfa_i},
$$ 
where $>$ denotes a total order on $\sfK$. Note that $Z^\bfa$
does not depend on the choice of a total order. The set $\{ Z^\bfa \ | \ \bfa \in \Z^{\oplus \sfK} \}$ forms an $\bbA$-basis of $\calT(\La)$. Since $\calT(\La)$
is an Ore domain, it is embedded into the skew field of fraction $\bbF(\calT(\La))$. 

\begin{definition} \label{def: Ex matrix}
We call a $\Z$-valued matrix $\tB = (b_{i,j})_{i \in \sfK,j \in \sfK_\ex}$
an \emph{exchange} matrix if it satisfies the following properties:
\bna
\item \label{it: finite connectedness}
For each $j \in \sfK$, there exist finitely many $i \in \sfK$
such that $b_{i,j} \ne 0$.
\item Its principal part $B=(b_{i,j})_{i,j\in \sfK_\ex}$ is \emph{skew-symmetrizable} with
an $\N_0$-valued diagonal matrix $D$; i.e., $DB$ is skew-symmetric. 
\ee
\end{definition}
For an exchange matrix $\tB$, we associate a \emph{valued quiver} $\tGamma^{\tB}$ whose set of vertices is $\sfK$, and whose arrows between vertices are assigned by the following rules:
\begin{align}     \label{eq: bivalued}
\bc
\xymatrix@R=0.5ex@C=6ex{ *{\bullet}<3pt> \ar@{->}[r]^{{\ulcorner}a,b {\lrcorner} }_<{k}  &*{\bullet}<3pt> \ar@{-}[l]^<{l}}
& \text{ if  $l , k \in \Kex$, $l \ne k$, $b_{kl}=a \ge 0$ and $b_{lk}=b \le 0$}, \\
\xymatrix@R=0.5ex@C=6ex{ *{\circ}<3pt> \ar@{->}[r]^{{\ulcorner}a,0 {\lrcorner} }_<{k}  &*{\bullet}<3pt> \ar@{-}[l]^<{l}} \ \ \text{(resp.}  \  \xymatrix@R=0.5ex@C=6ex{ *{\circ}<3pt> \ar@{<-}[r]^{{\ulcorner}0,b {\lrcorner} }_<{k}  &*{\bullet}<3pt> \ar@{-}[l]^<{l}} \text{)}
& \text{ if  $l \in \Kex$, $k \in \Kfr$ and $b_{kl}=a \ge 0$ (resp. $b_{kl}=b \le 0$)}.
\ec
\end{align}
Here we do not draw an arrow between $k$ and $l$ if $b_{kl} = 0$ (and $b_{lk} = 0$ when $l, k \in \Kex$).
Note that $\circ$ denotes a vertex in $\Kfr$, and
we call $\ulcorner a,b  \lrcorner$ the \emph{value} of the arrow.

\begin{convention} \label{conv: valued quiver} In this paper, we adopt the following scheme for depiction of a valued quiver: 
\bnum
\item if $b_{kl}=1$ and $b_{lk}=-b<0$, use $\xymatrix@R=0.5ex@C=6ex{ k \ar@{->}[r]|{<b \;}   & l} $,
\item if $b_{kl}=2$ and $b_{lk}=-b<0$, use $\xymatrix@R=0.5ex@C=6ex{ k \ar@{=>}[r]|{<b \;}   & l} $,
\item if $b_{kl}=3$ and $b_{lk}=-b<0$, use $\xymatrix@R=0.5ex@C=6ex{ k \ar@{=>}[r]|{<b \;} & l \ar@{-}[l]|{ \; \; \; \; \; } } $,
\item we usually skip $< \hspace{-.7ex} 1$ in an arrow $($when $\ulcorner a ,-1 \lrcorner$ and $1 \le a\le 3)$ for notational simplicity,
\item if $b_{kl} = -b_{lk}=\alpha>1$, use $\xymatrix@R=0.5ex@C=6ex{ k \ar[r]^{\alpha}   & l}$.
\ee
\end{convention}
 
We say that a pair $(\La,\tB)$  is \emph{compatible} if 
\begin{align*}
\sum_{k \in \sfK} b_{ki}\La_{kj} = d_i \delta_{i,j} \ \text{for any } i \in \sfK_\ex, \  j \in \sfK ,
\end{align*} where $\{ d_i \}_{i \in \sfK_\ex}$ are positive integers.
 In this case,
$B$ is skew-symmetrizable with the diagonal matrix $D=\diag(d_i \ | \ i \in \sfK_\ex )$. 
Also, for $k \in \sfK_\ex$, the \emph{mutation} of $(\La,\tB)$
\emph{in direction $k$} is a compatible pair 
$\mu_k(\La,\tB)\seteq (\mu_k(\La),\mu_k(\tB))$ defined as follows (\cite{BZ05}):
\begin{align*}
\mu_k(\La) \seteq E^T \La E \quad \text{ and } \quad 
\mu_k(\tB) \seteq E \tB F, 
\end{align*}
where the matrices $E=(e_{i,j})_{i,j\in \sfK}$ and $F=(f_{i,j})_{i,j\in \sfK_\ex}$
are given by 
\begin{align*}
e_{i,j} \seteq \bc
\; \delta_{i,j} & \text{ if } j \ne k,\\
-1 & \text{ if } i =j=k, \\
[-b_{ik}]_+ & \text{ if } i \ne j=k,
\ec
\qquad
f_{i,j} \seteq \bc
\; \delta_{i,j} & \text{ if } i \ne k,\\
-1 & \text{ if } i =j=k, \\
\; [b_{kj}]_+ & \text{ if } i = k \ne j. 
\ec
\end{align*}

Note that the mutation $\mu_k$ is involutive; i.e., $\mu_k(\mu_k(\La,\tB))=(\La,\tB)$.
The mutation $\mu_k$ on $\tB$ can be described as the operation $\mu_k$ on its corresponding valued quiver $\tGamma^{\tB}$ in the following algorithm:
\begin{algorithm} \label{Alg. mutation}
For $k \in \sfK_\ex$, the \emph{valued quiver mutation $\mu_k$} transforms $\tGamma^{\tB}$ into a new valued quiver $\mu_k(\tGamma^{\tB})=\tGamma^{\mu_k(\tB)}$ via the following rules, where we assume {\rm (i)} $ac>0$ or $bd>0$, and {\rm (ii)} we do not perform {\rm ($\mathcal{NC}$)} and {\rm ($\mathcal{C}$)} below, if $i$ and $j$ are frozen at the same time.
\ben
\item[]\hspace{-0.69cm}{\rm ($\mathcal{NC}$)} For each full-subquiver   $\xymatrix@!C=7mm@R=1mm{ i \ar[r]_{\ulcorner  a,b \lrcorner }  \ar@/^0.7pc/[rr]^{\ulcorner  e,f \lrcorner} & k  \ar[r]_{\ulcorner  c,d \lrcorner} & j }$ in $\tGamma^{\tB}$,
we change the value of the arrow from $i$ to $j$ into $\ulcorner e+ac,f-bd \lrcorner:$
\begin{align*}
\xymatrix@!C=20mm@R=1mm{ i  \ar[r]_{\ulcorner e+ac,f-bd \lrcorner} & j}.
\end{align*}
\item[{\rm ($\mathcal{C}$)}]  For each  full-subquiver   $\xymatrix@!C=7mm@R=1mm{ i \ar[r]_{\ulcorner  a,b \lrcorner }  & k  \ar[r]_{\ulcorner  c,d \lrcorner} & j  \ar@/_0.7pc/[ll]_{\ulcorner  e,f \lrcorner}  }$ with $(e,f) \ne (0,0)$ in $\tGamma^{\tB}$, we change the  valued arrow between $i$ and $j$ as follows:
\begin{align*}
\bc
\xymatrix@!C=20mm@R=7mm{ i  \ar@{<-}[r]_{\ulcorner e-bd,f+ac \lrcorner} & j}   & \text{ if }  f+ac \le  0 \le e-bd,\\
\xymatrix@!C=20mm@R=7mm{ i  \ar[r]_{\ulcorner f+ac,e-bd \lrcorner} & j}  & \text{ if }   f+ac \ge  0 \ge  e-bd.
\ec
\end{align*}

\item[{\rm ($\mathcal{R}$)}] Reverse the direction of each arrow incident to vertex $k$ and change the value $\ulcorner  a,b \lrcorner $  of each arrow into $\ulcorner  -b,-a \lrcorner$.
\ee
Here if there is no arrow between $i$ and $j$ in {\rm ($\mathcal{NC}$)} and {\rm ($\mathcal{C}$)}, then put $e = f = 0$ and follow the same rule.
\end{algorithm}

We define an isomorphism of $\Q(q^{1/2})$-algebras 
$\mu_k^* : \bbF(\calT(\mu_k \La)) \simeq \bbF(\calT(\La))$ by 
$$
\mu_k^*(Z_j) \seteq \bc
Z^{\bfa'} +Z^{\bfa''}  & \text{ if }j =k, \\
Z_j & \text{ if } j \ne k, 
\ec
$$
where
$$\bfa_i' =   -\delta_{i,k} + \delta(i\ne k)[b_{ik}]_+ \quad  \text{ and }   \quad  
\bfa_i'' =   -\delta_{i,k} + \delta(i\ne k)[-b_{ik}]_+ \quad \text{ for all } i \in \sfK.$$

\begin{definition}
For a compatible pair $(\La,\tB)$, an element $\sfz$ in $\bbF(\calT(\La))$ is  called
a \emph{quantum cluster variable} (resp. \emph{quantum cluster monomial})
if    
$$\sfz = \mu^*_{k_l}\cdots \mu^*_{k_2} \mu^*_{k_1}(Z_j) \qquad
\text{(resp. } \mu^*_{k_l}\cdots \mu^*_{k_2} \mu^*_{k_1}(Z^\bfa)  \text{)},$$
for some finite sequence $(k_1,k_2,\ldots,k_l)$ in $\sfK_\ex$
and $j \in \sfK$ (resp. $\bfa \in \Z_{\ge 0}^{\oplus \sfK}$). The \emph{quantum cluster algebra} $\scrA_q(\La,\tB)$ is the $\bbA$-subalgebra of $\bbF(\calT(\La))$
generated by all the quantum cluster variables. We call the triple $\calS \seteq (\{ Z_k \}_{k \in \sfK},\La,\tB)$ the \emph{initial seed} of $\scrA_q(\La,\tB)$.
\end{definition}

\begin{theorem} [The quantum Laurent phenomenon \cite{BZ05}] 
The quantum cluster algebra 
$\scrA_q(\La,\tB)$ is contained in $\calT(\La)$.
\end{theorem}

\begin{remark}
For a permutation $\pi \in \frakS_\sfK$ of $\sfK$ preserving $\sfK_\fr$, we
denote by $\pi(\La,\tB) = (\pi \La,\pi \tB)$ where 
$(\pi\La)_{i,j} \seteq \La_{\pi^{-1}(i)\pi^{-1}(j)}$ for $i,j \in \sfK$
and 
$(\pi\tB)_{i,j} \seteq b_{\pi^{-1}(i)\pi^{-1}(j)}$ for $(i,j) \in \sfK \times \sfK_\ex$. The operation $\pi$ on a compatible pair $(\La,\tB)$ obviously induces
the $\bbA$-algebra isomorphism 
$$
\pi^*: \scrA_q(\pi(\La,\tB)) \simeq  \scrA_q(\La,\tB) \quad \text{ given by }
\pi^*(Z_j) = Z_{\pi^{-1}(j)} \text{ for all }j\in \sfK
$$
and bijections between the sets of quantum cluster monomials. 
\end{remark}

Let $(\La,\tB)$ be a compatible pair.  An element $\sfz$ in $\calT(\La)$ is \emph{pointed} 
if it is of the following form: 
\begin{align} \label{eq: def degree and codegree}
\sfz =  q^{a} Z^{\bsg^R} + \sum_{ \bsn \in \Z_{\ge 0}^{\oplus \Kex} \setminus \{ 0 \} } p_\bsn Z^{\bsg^R+\tB \bsn}  
\end{align}
for some $a \in \dfrac{1}{2}\Z$, $\bsg^R \in \Z^{\oplus \K}$ and $p_\bsn \in \bbA$. In this case, we call $\bsg^R$ the \emph{degree} of the pointed 
element $\sfz$ and denote it by $\mathbf{deg}(\sfz)$. Note that the notion of degree \emph{does depend} on $\tB$.

\begin{theorem} \cite{FZ07,DWZ10,GHKK,Tran}  \hfill \label{thm: same degree same element}
\bna
\item Every quantum cluster monomial $\sfz$ in $\scrA_q(\La,\tB)$ is pointed.
\item If two quantum cluster monomials  $\sfz,\sfz' \in \scrA_q(\La,\tB)$ have the same degree, then $\sfz=\sfz'$. 
\ee    
\end{theorem}

Let $\sfz \in \scrA_q(\La,\tB)$ and 
$\sfz' \in \scrA_q(\mu_k(\La,\tB))$ be pointed elements such that
$\mathbf{deg}(\sfz) = (\bsg_i) \in \Z^{\oplus \sfK}$ 
and 
$\mathbf{deg}(\sfz') = (\bsg'_i) \in \Z^{\oplus \sfK}$. 
If $\mu^*_k(\sfz')=\sfz$, we have (\cite{FZ07,DWZ10,GHKK})
\begin{align} \label{eq: g-vector mutation}
\bsg_j = \bc
-\bsg'_k & \text{ if } j=k, \\
\bsg'_j + [b'_{jk}]_+ \bsg'_k & \text{ if $j \ne k$ and $\bsg_k' \ge 0$}, \\
\bsg'_j + [-b'_{jk}]_+ \bsg'_k & \text{ if $j \ne k$ and $\bsg_k' \le 0$}.
\ec
\end{align}

\section{Behaviours of degrees with respect to braid moves and forward shifts} \label{sec: Braid and degree}

In this section, we study the behaviours of degrees
with respect to several moves among sequences of simple reflections. For $2$-moves,  $3$-moves and forward shifts, they are investigated in \cite{FHOO2}. We mainly focus on the $4$-moves and $6$-moves, which have not been  
investigated before as far as the authors know, and quite complicated. We mainly follow the notations and frameworks in \cite{FHOO2}. 

\subsection{Setup} Let $\g$ be a complex finite-dimensional simple Lie algebra.
We denote by $I$ the index set of simple (co)roots $\Pi=\{ \al_i \}_{i \in I}$ (resp. $\Pi^\vee = \{ h_i \}_{i \in I}$), by $\Dynkin$
the Dynkin diagram of $\g$, by $\sfh$ the Coxeter number of $\g$, 
by $\sfC=(\sfc_{i,j})_{i,j\in I}$ the Cartan matrix of $\g$,
by $\sfD={\rm diag}(\sfd_i \in \Z_{\ge1} \ | \  i \in I)$ the minimal (left) symmetrizer of $\sfC$, by $\Phi^\pm$ the set of positive (resp. negative) roots of $\g$, 
and by $\wl$ the weight lattice of $\g$. 
We write $i \sim j$ for $i,j \in I$ if $\sfc_{i,j} < 0$, 
and let 
$\varpi_i$ be the fundamental weight for each $i \in I$. 
We set
$\rl \seteq \bigoplus_{i\in I}\Z \al_i$
and
$\rl^\pm  \seteq \pm \sum_{i\in I}\N_0 \al_i$. For $\be = \sum_{i \in I}n_i\al_i \in \rl^+$, we set $|\be|\seteq\sum_{i \in I} n_i$.
We denote by
$d(i,j)$ the number of edges between $i$ and $j$ in $\Dynkin$ (see \cite{JLO1} for examples of $d(i,j)$).

We take the symmetric bilinear form $( -, - )$ on $\wl$ such that
$$
(\al_i,\al_j) =\sfd_i \sfc_{i,j}=\sfd_j \sfc_{j,i} \quad \text{ and } \quad \lan h_i,\al_j \ran = \dfrac{2(\al_i,\al_j)}{(\al_i,\al_i)}=\sfc_{i,j}.
$$
Let $\weyl$ be the Weyl group of $\g$ generated by simple reflections 
$\{s_i\}_{i \in I}$ such that $s_i(\la)= \la - \lan h_i,\la \ran \al_i$ for $\la \in \wl$. 
For a finite sequence $\ii=(i_1,\ldots,i_r)$ in the  Weyl group $\weyl$,
we call $\ii$ \emph{reduced}, if $s_{i_1}\ldots s_{i_r}$ is a reduced expression of some $w$ in $\weyl$. Note that there exists a unique longest element 
$w_\circ$ in $\weyl$ and 
$w_\circ$ induces the automorphism $*$ on $I$ by $w_\circ(\al_i)= -\al_{i^*}$. 

We say that two sequences $\ii=(i_1,\ldots,i_r)$ and $\ii'=(i'_1,\ldots,i'_r)$ of indices are \emph{commutation equivalent},  denoted by $\ii \overset{c}{\sim}\ii'$,  if $\ii'$ can be obtained from $\ii$
by applying 2-moves (or commutation moves) $(i_k,i_{k+1}) \to (i_{k+1}=i_k',i_k=i'_{k+1})$ for $1 \le k <r$ and $d(i_k,i_{k+1})>1$. 
We denote by $[\ii]$ the equivalence class of $\ii$ under $\overset{c}{\sim}$.

\smallskip 

Let $\ii=(i_1,\ldots,i_\ell)$ be a reduced sequence of $w_\circ$. Then it is well-known 
that
\begin{align} \label{eq: param positive roots}
\Phi^+ = \{ \be^\ii_k \seteq s_{i_1} \ldots s_{i_{k-1}}(\al_{i_k}) \ | \  1 \le k \le \ell \} \quad \text{ and }  \quad   |\Phi^+ |=\ell.
\end{align} 

\begin{example}  \label{ex: B2 longest}
For $\g$ of type $B_2$, there are only two reduced sequences $\ii$ and $\ii'$ of 
$w_\circ$ given by $\ii=(1,2,1,2)$ and $\ii'=(2,1,2,1)$, respectively. 
Then 
\begin{align*}
\Phi^+ &= \{ \be^\ii_1=\al_1, \be^\ii_2=\al_1+\al_2, \be^\ii_3=\al_1+2\al_2, \be^\ii_4=\al_2 \}, \\
 &= \{ \be^{\ii'}_1=\al_2, \be^{\ii'}_2=\al_1+2\al_2, \be^{\ii'}_3=\al_1+\al_2, \be^{\ii'}_4=\al_1 \}.
\end{align*}
\end{example}

\subsection{Sequences of indices} Let $\ii=(i_u)_{u \in \N} \in I^\N$ be an arbitrary sequence satisfying 
\begin{align} \label{eq: the infinity condition}
|\{ u \in \N \ | \ i_u =i \}| =\infty \quad \text{ for any } i \in I.     
\end{align}
We denote by $I^{(\infty)}$ the set of sequences in $I^{\N}$
satisfying~\eqref{eq: the infinity condition}.  

\smallskip 

For $\ii \in I^{r}$ ($r \in \N \sqcup \{ \infty\}$), $u \in [1,r]$ and $j \in I$, we set
\begin{align*}
u_\ii^+(j) &\seteq \min(\{ k \in \N \ |  \ k >u, \ i_k=j\} \sqcup \{ r+1\}),  &&  u_\ii^+ \seteq u_\ii^+(i_u), \allowdisplaybreaks\\
u_\ii^-(j) &\seteq \max( \{k \in \N \ |  \ k <u, \ i_k=j\} \sqcup \{ 0 \} ),  &&  u_\ii^- \seteq u_\ii^-(i_u), \allowdisplaybreaks\\
u_\ii^{\min}(j) &\seteq \min(k \in \N \ |  \   \ i_k=j ), && u_\ii^{\min} \seteq u_\ii^{\min}(i_u). 
\end{align*}
We drop $_\ii$ if there is no danger of confusion. 

For $\ii \in I^{(\infty)}$, we define an $\N\times \N$ exchange matrix $\tB^{\ii}=(b^\ii_{u,v})_{u,v \in \N}$ as follows:
$$
b^\ii_{u,v} \seteq \bc
\pm 1  & \text{ if } v =u^{\pm}, \\
\pm \sfc_{i_u,i_v} & \text{ if } u < v<u^+<v^+ \text{ (resp. }v<u<v^+<u^+\text{)},\\
0 & \text{ otherwise.}
\ec
$$
Note that $\tB^{\ii}$ is skew-symmetrizable with $\diag(\sfd_{i_u} \ | \  u \in \N)$ and satisfies the conditions in Definition~\ref{def: Ex matrix}.

Let $\La^{\ii}=(\La^{\ii}_{u,v})$ be a skew-symmetric matrix given by
$$
\La^{\ii}_{u,v}= - \La^{\ii}_{v,u} \seteq (\varpi_{i_u}-w^\ii_u\varpi_{i_u},\varpi_{i_v}+w^\ii_v\varpi_{i_v}) \quad \text{ for } u \le v,
$$
where 
$$w^\ii_k \seteq s_{i_1}\ldots s_{i_k} \in \weyl \quad \text{ for } k \in \N.$$ 
 For $s \in \N$, we set 
$$
\sfK \seteq [1,s], \qquad \sfK_\fr \seteq \{ k \in \sfK \ | \ k^+ > s \}, \qquad \sfK_\ex \seteq \sfK \setminus \sfK_\fr.
$$

\begin{lemma} \cite[Lemma 1.1]{FHOO2}\label{Lem:Lambda}
For $u,v \in \N$ with $u < v^+$, we have $\La^{\ii}_{u,v} = (\varpi_{i_u}-w^\ii_u\varpi_{i_u}, \varpi_{i_v}+w^\ii_v\varpi_{i_v}).$
\end{lemma}

\begin{proposition} \cite[Proposition 1.2]{FHOO2}
Let $\ii \in I^{(\infty)}$. 
\bna 
\item The pair $(\La^{\ii},\tB^{\ii})$ is \emph{compatible with $\diag(\sfd_{i_u})_{u \in \N}$}; i.e., for any $u,v \in \N$, we have
$$
\sum_{k \in \N} b^\ii_{k,u}\La^\ii_{k,v} = 2\sfd_{i_u} \delta_{u,v}.
$$
\item For each $s \in \N$, we set 
$$\La^{\ii,s} \seteq (\La^{\ii}_{u,v})_{u,v \in \sfK} \quad \text{ and } \quad \tB^{\ii,s} \seteq (b^{\ii}_{u,v})_{u\in \sfK, v \in \sfK_\ex}.$$ 
Then the pair  $(\La^{\ii,s},\tB^{\ii,s})$ is compatible with $\diag(\sfd_{i_k})_{k \in \sfK_\ex}$. 
\ee
\end{proposition}

We define $\tGamma^\ii = (\tGamma^\ii_0, \tGamma^\ii_1)$ as the valued quiver associated with 
$\tB^{\ii}$, where $\tGamma^\ii_0$ denotes the set of vertices and $\tGamma^\ii_1$ denotes the
set of arrows.  
For a subset $\calK$ of $\tGamma^\ii_0$, we denote by $\tGamma^\ii|_\calK$
the full subquiver of $\tGamma^\ii$ whose set of vertices is $\calK$. 
 
\smallskip 

We denote by $\calA_\ii$ the quantum cluster algebra associated with the compatible pair $(\La^{\ii},\tB^{\ii})$, which is an $\bbA$-subalgebra
of the quantum torus $\calT(\La^{\ii})$. For each $s \in \N$, we also have the quantum cluster algebra  $\calA^s_\ii=\scrA_q(\La^{\ii,s},\tB^{\ii,s})$
in the quantum torus $\calT(\La^{\ii,s})$. Then we have
$$
\raisebox{1.5em}{\xymatrix@R=3ex@C=4ex{   \calA^1_\ii \ar@{^{(}->}[r] \ar@{^{(}->}[d]  & \calA^2_\ii \ar@{^{(}->}[r] \ar@{^{(}->}[d]  & \cdots \ar@{^{(}->}[r]  & \calA_\ii  \ar@{^{(}->}[d]  \\ 
\calT(\La^{\ii,1}) \ar@{^{(}->}[r]  & \calT(\La^{\ii,2}) \ar@{^{(}->}[r]  & \cdots \ar@{^{(}->}[r]  & \calT(\La^{\ii})}
}
\quad \text{ and } \quad \bigcup_{s \in \N} \calA_\ii^s =\calA_\ii. 
$$

For $\ii \in I^\N$, let $C_\ii \subset \Z^{\oplus \N}$ denote the cone defined by
\begin{equation} \label{eq:cone}
C_\ii \seteq \left\{\bsg = (g_u)_{u \in \N} \in \Z^{\oplus \N} \; \middle | \; \sum_{v \ge u, i_v = i_u} g_v \ge 0, \forall u \in \N   \right\}
= \sum_{u \in \N} \N_0 (\bse_u - \bse_{u^-_{\ii}}),
\end{equation}
where $\{\bse_u\}_{u \in \N}$ is the natural basis of $\Z^{\oplus \N}$ and we understand $\bse_0 = 0$. 

For $\bsg = (g_u)_{u \in \N} \in \Z^{\oplus \N}$ and $\ii \in I^{\N}$, write
\begin{equation}
\sfp_{\ii}(\bsg; i)\seteq \sum_{u\in \N, i_u = i} g_u.\label{eq:p-sum}
\end{equation}

\subsection{$2$-moves} In this and next subsections, we recall the results on
$2$-moves  and $3$-moves investigated in~\cite{FHOO2}.

\smallskip 

For sequences $\ii=(i_u)_{u \in \N},\ii'=(i'_u)_{u \in \N} \in I^{(\infty)}$,
we say that $\ii'$ can be obtained from $\ii$ via a \emph{$2$-move} if there
exists a unique $k \in \N$ such that $$i_k=i'_{k+1}, \ \  i_{k+1}=i'_{k}, \ \ i_u=i'_u \ \ (u \not\in \{ k,k+1\})$$ and $$\sfc_{i_{k+1},i_{k}}=\sfc_{i_k,i_{k+1}}=0.$$ In this case, we write $\ii'=\gamma_k\ii$.

\begin{proposition} \cite[\S 2.1]{FHOO2} \label{prop: c-move}
Let $\ii,\ii' \in I^\N$ be related by $\ii'=\gamma_k\ii$ for some $k \in \N$. 
\bnum
\item $\tB^{\ii'}= \sigma_k\tB^{\ii}$ and $\La^{\ii'} = \sigma_k \La^{\ii}$.
\item We have an $\bbA$-algebra isomorphism 
$$
\gamma^*_k(=\sigma^*_k) : \calA_{\ii'} \isoto \calA_\ii \quad \text{ given by } Z_u \mapsto Z_{\sigma_k(u)} \text{ for all }u\in \N,
$$
which induces a bijection between the sets of quantum cluster monomials. 
\item If $\sfz \in \calA_{\ii'}$ is a quantum cluster monomial with $\bfdeg(\sfz)=\bsg'=(g'_u)_{u \in \N}$, then the quantum cluster monomial
$\gamma_k^*\sfz \in \calA_\ii$ has $\bfdeg(\gamma_k^*\sfz)=\bsg=(g_u)_{u \in \N}$ as follows:
\begin{align} \label{eq: com g}
g_u = \bc
g'_{\sigma_k(u)}  & \text{ if } u \in \{k,k+1\}, \\
g'_u & \text{ otherwise}. 
\ec    
\end{align}
\item The map $\bsg' \mapsto \bsg$ given by~\eqref{eq: com g} sends the cone $C_{\ii'} \subset \Z^{\oplus \N}$ into the cone $C_{\ii} \subset \Z^{\oplus \N}$.
\item \label{it: c pgi} The elements $\bsg,\bsg' \in \Z^{\oplus\N}$ related by~\eqref{eq: com g} satisfy
$$
\sfp_{\ii}(\bsg;i) =\sfp_{\ii'}(\bsg';i) \quad \text{ for all $i \in I$.}
$$
\ee    
\end{proposition}

\subsection{$3$-moves}
For sequences $\ii=(i_u)_{u \in \N},\ii'=(i'_u)_{u \in \N} \in I^{\infty}$,
we say that $\ii'$ can be obtained from $\ii$ via a \emph{$3$-move}, if there
exists a unique $k \in \N$ such that 
$$i_k=i_{k+2}=i'_{k+1}, \ \ i_{k+1}=i'_{k}=i'_{k+2}, \ \ i_u=i'_u \ \ (u \not\in [ k,k+2] )$$ and 
$$\sfc_{i_{k+1},i_{k}}\sfc_{i_k,i_{k+1}}=1.$$ In this case, we write $\ii'=\beta_k\ii$.

 \begin{proposition} \cite[\S 2.2]{FHOO2} \label{prop: 3 move}
 Let $\ii,\ii' \in I^{(\infty)}$ be related by $\ii'=\beta_k\ii$ for some $k \in \N$. 
\bnum
\item $\tB^{\ii'}= \sigma_{k+1}\mu_k\tB^{\ii}$ and $\La^{\ii'} = \sigma_{k+1}\mu_k \La^{\ii}$.
\item We have an $\bbA$-algebra isomorphism 
$$
\beta^*_k(=\mu^*_k\sigma^*_{k+1}) : \calA_{\ii'} \isoto \calA_\ii \quad \text{ given by } Z_u \mapsto \mu_k(Z_{\sigma_{k+1}(u)}) \text{ for all }u\in \N,
$$
which induces a bijection between the sets of quantum cluster monomials. 
\item If $\sfz \in \calA_{\ii'}$ is a quantum cluster monomial with $\bfdeg(\sfz)=\bsg'=(g'_u)_{u \in \N}$, then the quantum cluster monomial
$\beta_k^*\sfz \in \calA_\ii$ has $\bfdeg(\beta_k^*\sfz)=\bsg=(g_u)_{u \in \N}$ as follows:
\begin{align}\label{eq: 3 g}
g_u = \bc
-g'_{k}  & \text{ if } u=k, \\
g'_{\sigma_{k+1}(u)} + [g_k']_+  & \text{ if } u \in \{ k+2,(k+1)_\ii^- \}, \\
g'_{\sigma_{k+1}(u)} - [-g_k']_+  & \text{ if } u \in \{ k+1,k_\ii^- \}, \\
g'_u & \text{ otherwise}. 
\ec
\end{align}
\item The map $\bsg' \mapsto \bsg$ given by~\eqref{eq: 3 g} sends the cone $C_{\ii'} \subset \Z^{\oplus \N}$ into the cone $C_{\ii} \subset \Z^{\oplus \N}$.
\item \label{it: 3 pgi} The elements $\bsg,\bsg' \in \Z^{\oplus\N}$ related by~\eqref{eq: 3 g} satisfy
$$
\sfp_{\ii}(\bsg;i) = \bc
\sfp_{\ii'}(\bsg';i) + [-g_k']_+ & \text{ if } i=i_k \text{ and } k^-_\ii =0, \\
\sfp_{\ii'}(\bsg';i) - [g_k']_+ & \text{ if } i=i_{k+1} \text{ and } (k+1)^-_\ii =0, \\
\sfp_{\ii'}(\bsg';i) & \text{ otherwise}.
\ec
$$
\ee     
\end{proposition}

\subsection{$4$-moves}
For sequences $\ii=(i_u)_{u \in \N},\ii'=(i'_u)_{u \in \N} \in I^{(\infty)}$,
we say that $\ii'$ can be obtained from $\ii$ via a \emph{$4$-move} if there
exists a unique $k \in \N$ such that $$i_k=i_{k+2}=i'_{k+1}=i'_{k+3}, \ \ i_{k+1}=i_{k+3}=i'_{k}=i'_{k+2}, \   \ i_u=i'_u \ \ (u \not\in [k,k+3])$$ and $$\sfc_{i_{k+1},i_{k}}\sfc_{i_k,i_{k+1}}=2.$$ In this case, we write $\ii'=\eta_k\ii$.
Throughout this subsection, we assume that $\ii,\ii' \in I^\N$ are related by $\ii'=\eta_k\ii$ for some $k \in \N$, and sometimes  we write $i \seteq i_k$ and $j \seteq i_{k+1}$.

\begin{lemma} \label{Lem:eta1}
We have 
\begin{align} \label{eq: 4-move}
\tB^{\ii'} = \sigma_{k+2}\sigma_{k} \mu_k\mu_{k+1}\mu_k \tB^{\ii}.
\end{align}
\end{lemma}  

\begin{proof}
By condition~\eqref{it: finite connectedness} in Definition~\ref{def: Ex matrix},
it is enough to consider the local behaviour under the operation described in~\eqref{eq: 4-move}.
Note that 
\ben
\item the set $\calK=\{ k,k+1,k+2,k+3, k^-_\ii = (k+1)^-_{\ii'}, (k+1)^-_\ii = k^-_{\ii'} \}$
consists of vertices adjacent to $\{k, k+1\}$ in $\tGamma^\ii$ (resp. $\tGamma^{\ii'}$) and themselves, 
\item the sequence of mutations $\mu_k\mu_{k+1}\mu_k$ on $\tGamma^\ii$ only affects the full subquiver $\tGamma^\ii|_\calK$ (resp. $\tGamma^{\ii'}|_\calK$). 
\ee
Thus it is enough to check that 
$$\mu_k\mu_{k+1}\mu_k\tGamma^\ii|_\calK = \sigma_{k}\sigma_{k+2}\tGamma^{\ii'}|_\calK .$$
In this proof, 
we shall only consider the case when both $k^-_\ii$ and $(k+1)^-_\ii$ exist, since the proofs for the other cases are similar. In that case, the valued quiver $\tGamma^\ii|_\calK $ can be depicted as follows:
\begin{align} \label{eq: general eta}
\tGamma^\ii |_\calK = \scalebox{0.87}{\raisebox{6mm}{
\xymatrix@R=3ex@C=3ex{
& k+2 \ar[dr]|{\ulcorner b,-a \lrcorner } && k \ar[ll] \ar[dr]|{\ulcorner b,-a \lrcorner } && k_\ii^- \ar[ll] \\
k+3 \ar@{.>}[ur]|{\ulcorner a,-b \lrcorner }  && k+1 \ar[ll] \ar[ur]|{\ulcorner a,-b \lrcorner } && (k+1)_\ii^- \ar[ll] \ar@{.>}[ur]|{\ulcorner a,-b \lrcorner }
}}}
\end{align}
where $a = -\sfc_{i_{k+1},i_k}$ and $b = -\sfc_{i_k, i_{k+1}}$.  Here
\begin{itemize}
\item the dotted arrow between $k_\ii^-$ and  $(k+1)_\ii^-$ exists if and only if $k_\ii^- < (k+1)_\ii^-$,
\item the dotted arrow between $k+3$ and  $k+2$ exists if and only if  $(k+2)_\ii^+ < (k+3)_\ii^+$.
\end{itemize}
Let us consider the case when $i_k =n-1$, $i_{k+1}=n$ and $\g=B_n$.  Then~\eqref{eq: general eta} can be drawn as follows with Convention~\ref{conv: valued quiver}:
\begin{align*}  
\tGamma^\ii |_\calK = \scalebox{0.87}{\raisebox{6mm}{
\xymatrix@R=3ex@C=3ex{
& k+2 \ar[dr]|{\ulcorner 2} && k \ar[ll] \ar[dr]|{\ulcorner 2} && k_\ii^- \ar[ll] \\
k+3 \ar@{:>}[ur]    && k+1 \ar[ll]  \ar@{=>}[ur] && (k+1)_\ii^- \ar[ll] \ar@{:>}[ur]
}}}
\end{align*}
Then, by following Algorithm~\ref{Alg. mutation}, one can see that
$\mu_k\mu_{k+1} \mu_k\Gamma^\ii|_\calK$ can be depicted as follows:
\begin{align*}  
\mu_k\mu_{k+1} \mu_k\Gamma^\ii |_\calK = \scalebox{0.87}{\raisebox{6mm}{
\xymatrix@R=3ex@C=3ex{
& k+2   \ar@{.>}[dl]|{2 \urcorner}   && k   \ar[ll] \ar[dl]|{2 \urcorner}  && k_\ii^-   \ar[ll]  \\
k+3 \ar@{=>}[urrr]   \ar@{<-}[rr] && k+1 \ar@{=>}[urrr]   && (k+1)_\ii^-   \ar@{<.}[ur]|{2 \urcorner}   \ar[ll]
}}}
\end{align*}
Here the dotted arrow between $k+3$ and  $k+2$ exists if and only if  $(k+2)_\ii^+ > (k+3)_\ii^+$.

On the other hand, the valued quiver $\Gamma^{\ii'}|_\calK$ is depicted as:
\begin{align*}  
\Gamma^{\ii'}|_\calK = \scalebox{0.87}{\raisebox{6mm}{
\xymatrix@R=3ex@C=3ex{
k+3 \ar@{.>}[dr]|{\ulcorner 2} && k+1 \ar[dr]|{\ulcorner 2} \ar[ll]  && (k+1)_{\ii'}^- \ar[ll] \ar@{.>}[dr]|{\ulcorner 2}  \\
& k+2 \ar@{=>}[ur] && k  \ar@{=>}[ur] \ar[ll]  && k_{\ii'}^- \ar[ll]
}}}
\end{align*}
Here
\begin{itemize}
\item the dotted arrow between $(k+1)_{\ii'}^-$ and  $k_{\ii'}^-$ exists if and only if $$(k+1)_{\ii'}^-=k_\ii^- > (k+1)_\ii^- = k_{\ii'}^-,$$
\item the dotted arrow between $k+3$ and  $k+2$ exists if and only if
$$ (k+3)_\ii^+=(k+2)_{\ii'}^+ < (k+3)_{\ii'}^+= (k+2)_\ii^+.$$
\end{itemize}
Thus our assertion follows. For $C_n$ and $F_4$ cases, one can prove in a similar way.
\end{proof}

\begin{lemma} \label{Lem:eta2}
We have 
$$\La^{\ii'}  = \sigma_{k+2}\sigma_{k}\mu_k\mu_{k+1}\mu_k \La^{\ii}.$$
\end{lemma}

\begin{proof}
Let us denote by
$$
\Lambda \seteq \La^{\ii}  \quad \text{ and } \quad \Lambda''' \seteq \mu_k\mu_{k+1}\mu_k \La^{\ii} 
\qquad
\text{for notational simplicity.}$$

Note that
$$ \La'''_{u,v} = \La_{u,v}  \quad \text{ if } u,v \not\in \{k,k+1\}.$$

\noindent
(A) For the cases when $|\{u \ne v\} \cap \{ k,k+1 \} |=1$,  we have
\begin{align*}
 \La'''_{k+1,v} & = \La''_{k+1,v}   = -\La_{k+1,v} +\La_{k+3,v}+\La_{(k+1)_\ii^-,v} \quad (\text{for } v>k+1), \\
 \La'''_{u,k+1} & =  \La''_{u,k+1} = -\La_{u,k+1}  +  \La_{u, k+3} + \La_{u,(k+1)_\ii^-} \quad (\text{for } u<k),\\
 \La'''_{k,v} & 
= \La_{k,v} + \sfc_{i_{(k+1)^-_\ii},i_k} \La_{(k+1)^-_\ii,v} - \La_{k+2,v} -\sfc_{i_{k+3},i_k} \La_{k+3,v}+\La_{k^-_\ii,v}     \quad (\text{for } v>k+1), \\
 \La'''_{u,k} &  
=  \La_{u,k} + \sfc_{i_{(k+1)^-_\ii},i_k} \La_{u,(k+1)^-_\ii} - \La_{u,k+2} -\sfc_{i_{k+3},i_k}  \La_{u,k+3}+ \La_{u,k_\ii^-}  \quad (\text{for } u<k),
\end{align*}
by direct calculations.

\smallskip
\noindent
(B) For the cases when $ \{u ,v\} = \{ k,k+1 \}$, we have
\begin{align*}
 \La'''_{k,k+1} & = - \La_{k,k+1} -\sfc_{i_{(k+1)^-_\ii},i_k} \La_{(k+1)^-_\ii,k+1} - \La_{k+1,k+2}  + \La_{k, k+3} +\sfc_{i_{(k+1)^-_\ii},i_k}\La_{(k+1)^-_\ii, k+3}  \\
&- \La_{k+2, k+3}   - \La_{(k+1)_\ii^-,k}  + \La_{(k+1)_\ii^-,k+2} -\sfc_{i_{k+3},i_k}\La_{ k+1,k+3}  + \sfc_{i_{k+3},i_k} \La_{ (k+1)_\ii^-,k+3}  \\
& -\La_{k_\ii^-,k+1}  +  \La_{k_\ii^-, k+3} + \La_{k_\ii^-, (k+1)_\ii^-}, \\
\La'''_{k+1,k} & = - \La'''_{k,k+1}.
\end{align*}

Recall that   $(i_k,i_{k+1},i_{k+2},i_{k+3}) = (i,j,i,j) \overset{\eta_k}{\leadsto} (i'_k,i'_{k+1},i'_{k+2},i'_{k+3}) =(j,i,j,i)$.
Thus it suffices to prove that the following equation holds:
$$
\Lambda'''_{\sigma_{k+2}\sigma_k(u),\sigma_{k+2}\sigma_k(v)} = \La^{\ii'}_{u,v}.
$$
Since only $\{k,k+1,k+2,k+3\}$ are changed by $\sigma_{k+2}\sigma_k$, it is enough to show that the following $15$ equations hold:
\bnum
\item  \label{it: eq(i)} $\La^{\ii'}_{k,v}= \Lambda'''_{k+1,v}= -\La_{k+1,v} +\La_{k+3,v}+\La_{(k+1)_{\ii}^-,v}$ for $v>k+3$,
\item  \label{it: eq(ii)} $\La^{\ii'}_{k+1,v}= \Lambda'''_{k,v}=\La_{k,v} + \sfc_{j,i} \La_{(k+1)^-_\ii,v} - \La_{k+2,v} -\sfc_{j,i} \La_{k+3,v}+\La_{k^-_\ii,v}$ for $v>k+3$,
\item  \label{it: eq(iii)} $\La^{\ii'}_{k+2,v}= \Lambda'''_{k+3,v}= \La_{k+3,v}$ for $v>k+3$,
\item  \label{it: eq(iv)} $\La^{\ii'}_{k+3,v}= \Lambda'''_{k+2,v}= \La_{k+2,v}$ for $v>k+3$,
\item \label{it: eq(v)} $\La^{\ii'}_{u,k}= \Lambda'''_{u,k+1}=  -\La_{u,k+1}  +  \La_{u, k+3} + \La_{u,(k+1)_\ii^-} $ for $u<k$,
\item  \label{it: eq(vi)} $\La^{\ii'}_{u,k+1}= \Lambda'''_{u,k}= \La_{u,k} + \sfc_{j,i} \La_{u,(k+1)^-_\ii} - \La_{u,k+2} -\sfc_{j,i}  \La_{u,k+3}+ \La_{u,k_\ii^-}$ for $u<k$,
\item  \label{it: eq(vii)}$\La^{\ii'}_{u,k+2}= \Lambda'''_{u,k+3}=\La_{u,k+3}$ for $u<k$,
\item  \label{it: eq(viii)} $\La^{\ii'}_{u,k+3}= \Lambda'''_{u,k+2}=\La_{u,k+2}$ for $u<k$,
\item \label{it: eq(ix)} $\La^{\ii'}_{k,k+1}= \Lambda'''_{k+1,k}$,
\item \label{it: eq(x)} $\La^{\ii'}_{k+2,k+3}= \Lambda'''_{k+3,k+2}=\La_{k+3,k+2}$,
\item  \label{it: eq(xi)} $\La^{\ii'}_{k,k+3}= \Lambda'''_{k+1,k+2}= -\La_{k+1,k+2} +\La_{k+3,k+2}+\La_{(k+1)_\ii^-,k+2} $,
\item  \label{it: eq(xii)} $\La^{\ii'}_{k,k+2}= \Lambda'''_{k+1,k+3}= -\La_{k+1,k+3} +\La_{(k+1)_\ii^-,k+3}$,
\item  \label{it: eq(xiii)} $\La^{\ii'}_{k+1,k+2}= \Lambda'''_{k,k+3}=  \La_{k,k+3} + \sfc_{j,i} \La_{(k+1)^-_\ii,k+3} - \La_{k+2,k+3} +\La_{k^-_\ii,k+3} $,
\item  \label{it: eq(xiv)} $\La^{\ii'}_{k+1,k+3}= \Lambda'''_{k,k+2}= \La_{k,k+2} + \sfc_{j,i} \La_{(k+1)^-_\ii,k+2}  -\sfc_{j,i} \La_{k+3,k+2}+\La_{k^-_\ii,k+2}$,
\item \label{it: eq(xv)} $\La^{\ii'}_{u,v} =\Lambda'''_{u,v}$ for $u,v \not\in \{k,k+1,k+2,k+3\}$.
\ee

\medskip

\noindent
\eqref{it: eq(iii)}, \eqref{it: eq(iv)}, \eqref{it: eq(vii)}, \eqref{it: eq(viii)} and \eqref{it: eq(x)}-cases:  Let us recall the following facts:
\begin{enumerate}
\item[{\rm (a)}] \label{it: 3to2} $w^{\ii}_{k+3}\varpi_j = w^{\ii'}_{k+3}\varpi_j= w^{\ii'}_{k+2}\varpi_j$,
\item[{\rm (b)}] \label{it: 3to2p} $w^{\ii'}_{k+3}\varpi_i= w^{\ii}_{k+3}\varpi_i =w^{\ii}_{k+2}\varpi_i.$
\end{enumerate}
The assertion for \eqref{it: eq(iii)}, \eqref{it: eq(iv)}, \eqref{it: eq(vii)} and \eqref{it: eq(viii)}-cases follows directly.  
Let us check \eqref{it: eq(x)}-cases. By Lemma~\ref{Lem:Lambda}, we have
\begin{align*}
\La_{k+3,k+2} & = (\varpi_j - w^\ii_{k+3}\varpi_j,\varpi_i+ w^\ii_{k+2}\varpi_i  )   \\
&= (\varpi_j - w^{\ii'}_{k+2}\varpi_j,\varpi_i+ w^{\ii'}_{k+3}\varpi_i  ) = \La^{\ii'}_{k+3,k+2},
\end{align*}
which completes our assertions for these cases.

\medskip

\noindent
\eqref{it: eq(i)}, \eqref{it: eq(v)}, \eqref{it: eq(xi)} and \eqref{it: eq(xii)}-cases: Since the proofs for these cases are all similar, we shall give a proof for \eqref{it: eq(i)}-case only.
Note that
$$  \La^{\ii'}_{k,v} = (\varpi_j - w_{k-1}^{\ii}s_j \varpi_j,  \varpi_{i_v}+w^\ii_v \varpi_{i_v})=(\varpi_j - w_{k-1}^{\ii}(\varpi_j-\al_j),  \varpi_{i_v}+w^\ii_v \varpi_{i_v}).$$
On the other hand
\begin{align*}
-\La_{k+1,v}  & = -(\varpi_j - w_{k-1}^{\ii}s_i s_j \varpi_j,  \varpi_{i_v}+w^\ii_v \varpi_{i_v}) \allowdisplaybreaks \\
&= ( - \varpi_j - w_{k-1}^{\ii}(-\varpi_j+\al_j-\sfc_{i,j}\al_i),  \varpi_{i_v}+w^\ii_v \varpi_{i_v}), \allowdisplaybreaks\\
\La_{k+3,v} & =(\varpi_j - w_{k-1}^{\ii} s_i s_js_i s_j \varpi_j,  \varpi_{i_v}+w^\ii_v \varpi_{i_v}) \allowdisplaybreaks\\
& =(\varpi_j - w_{k-1}^{\ii} (\varpi_j+\sfc_{i,j}\al_i-2\al_j),  \varpi_{i_v}+w^\ii_v \varpi_{i_v}), \allowdisplaybreaks\\
\La_{(k+1)_\ii^-,v} & =(\varpi_j - w_{k-1}^{\ii} \varpi_j,  \varpi_{i_v}+w^\ii_v \varpi_{i_v}).
\end{align*}
Thus the sum of the three given above is
$$
(\varpi_j - w_{k-1}^{\ii} (\varpi_j-\al_j),  \varpi_{i_v}+w^\ii_v \varpi_{i_v}) =   \La^{\ii'}_{k,v},
$$
which shows the assertion \eqref{it: eq(i)}.

\medskip
\noindent
\eqref{it: eq(ii)}, \eqref{it: eq(vi)}, \eqref{it: eq(xiii)} and \eqref{it: eq(xiv)}-cases:
Since the proofs for these cases are all similar, we shall give a proof for \eqref{it: eq(ii)}-case only.
Note that
$$
\La^{\ii'}_{k+1,v}  = (\varpi_i - w^\ii_{k-1}s_j s_i \varpi_i, \varpi_{i_v} +w^\ii_v \varpi_{i_v})=  (\varpi_i - w^\ii_{k-1} (\varpi_i-\al_i+\sfc_{j,i}\al_j), \varpi_{i_v} +w^\ii_v \varpi_{i_v}).
$$

First let us consider the following two terms:
\begin{align*}
\sfc_{j,i} \La_{(k+1)^-_\ii,v}  & =   (\sfc_{j,i}\varpi_j - w^\ii_{k-1} \sfc_{j,i}\varpi_j, \varpi_{i_v} +w^\ii_v \varpi_{i_v}) \allowdisplaybreaks\\
-\sfc_{j,i} \La_{k+3,v} & = (-\sfc_{j,i}\varpi_j - w^\ii_{k-1} (-\sfc_{j,i} \varpi_j-2\al_i+2\sfc_{j,i} \al_j) \varpi_j, \varpi_{i_v} +w^\ii_v \varpi_{i_v}),
\end{align*}
whose sum is
\begin{align} \label{eq: (ii)-step1}
( - w^\ii_{k-1}(-2\al_i+2\sfc_{j,i}\al_j), \varpi_{\ii_v} +w^\ii_v \varpi_{i_v}).
\end{align}
Next, let us consider the following three terms:
\begin{align*}
\La_{k,v} & =  (\varpi_i - w^\ii_{k-1}s_i \varpi_i, \varpi_{i_v} +w^\ii_v \varpi_{i_v})  = (\varpi_i - w^\ii_{k-1}(\varpi_i-\al_i),\varpi_{i_v} +w^\ii_v \varpi_{i_v}) \allowdisplaybreaks\\
- \La_{k+2,v}  &= -(\varpi_i - w^\ii_{k-1}s_is_js_i \varpi_i, \varpi_{i_v} +w^\ii_v \varpi_{i_v}) \allowdisplaybreaks\\
& = (-\varpi_i - w^\ii_{k-1}(-\varpi_i-\sfc_{j,i}\al_j+2\al_i),\varpi_{i_v} +w^\ii_v \varpi_{i_v}), \allowdisplaybreaks\\
\La_{k^-_\ii,v} & = (\varpi_i - w^\ii_{k-1} \varpi_i, \varpi_{i_v} +w^\ii_v \varpi_{i_v}),
\end{align*}
whose sum is
\begin{align} \label{eq: (ii)-step2}
(\varpi_i - w^\ii_{k-1}(\varpi_i -\sfc_{j,i}\al_j+\al_i),\varpi_{i_v} +w^\ii_v \varpi_{i_v}).
\end{align}

Hence,  RHS of \eqref{it: eq(ii)} becomes
\begin{align*} 
\eqref{eq: (ii)-step1} + \eqref{eq: (ii)-step2} &= (\varpi_i - w^\ii_{k-1}(\varpi_i-\al_i+\sfc_{j,i}\al_j), \varpi_{i_v} +w^\ii_v \varpi_{i_v})= \La^{\ii'}_{k+1,v},
\end{align*}
which implies assertion \eqref{it: eq(ii)}.

\medskip
\noindent
\eqref{it: eq(ix)}-case: This is the most complicated case.  We have to show that
\begin{equation} \label{eq: 13-terms}
\begin{aligned}
\La^{\ii'}_{k,k+1} & =   \La_{k,k+1} +\sfc_{j,i} \La_{(k+1)^-_\ii,k+1} + \La_{k+1,k+2}  - \La_{k, k+3} \allowdisplaybreaks\\
&-\sfc_{j,i}\La_{(k+1)^-_\ii, k+3} + \La_{k+2, k+3}   + \La_{(k+1)_\ii^-,k}  - \La_{(k+1)_\ii^-,k+2} +\sfc_{j,i}\La_{ k+1,k+3} \allowdisplaybreaks\\
& - \sfc_{j,i} \La_{ (k+1)_\ii^-,k+3}  +\La_{k_\ii^-,k+1}  -  \La_{k_\ii^-, k+3} - \La_{k_\ii^-, (k+1)_\ii^-},
\end{aligned}
\end{equation}
whose RHS consists of $13$ terms.

On the other hand, $\La^{\ii'}_{k,k+1}$ can be rewritten as follows:
\begin{align*}
\La^{\ii'}_{k,k+1} & =(\varpi_{j}- w^\ii_{k-1}s_j\varpi_{j}, \varpi_i+w^\ii_{k-1}s_js_i\varpi_i)   \allowdisplaybreaks\\
& =(\varpi_{j}- w^\ii_{k-1}(\varpi_{j}-\al_j), \varpi_i+w^\ii_{k-1}(\varpi_i-\al_i+\sfc_{j,i}\al_j)) \allowdisplaybreaks\\
& = (\varpi_{j},  w^\ii_{k-1}( \varpi_i)) -(\varpi_{j},  w^\ii_{k-1}( \al_i)) +\sfc_{j,i}(\varpi_{j}, w^\ii_{k-1}(\al_j)) \allowdisplaybreaks\\
&\hspace{5ex}-( w^\ii_{k-1}(\varpi_{j}), \varpi_i) +(  w^\ii_{k-1}(\al_j), \varpi_i).
\end{align*}

Let us analyze the $13$ terms in~\eqref{eq: 13-terms}:
\ben
\item\label{eq: (1)} $\La_{k,k+1} = - \La_{k+1,k} \overset{Lemma~\ref{Lem:Lambda}}{=}
- (\varpi_j -w^{\ii}_{k-1}s_is_j\varpi_j, \varpi_i+w^\ii_{k-1}s_i\varpi_i ),$ 
\item\label{eq: (7)} $\La_{(k+1)_\ii^-,k}  =(\varpi_j -w^{\ii}_{k-1}\varpi_j, \varpi_i+w^\ii_{k-1}s_i\varpi_i)$, 
\item\label{eq: (2)} $\sfc_{j,i}\La_{(k+1)^-_\ii,k+1} = \sfc_{j,i} (\varpi_j -w^{\ii}_{k-1}\varpi_j, \varpi_j+w^\ii_{k-1}s_is_j\varpi_j)$, 
\item\label{eq: (5)} $-\sfc_{j,i}\La_{(k+1)^-_\ii, k+3} =-\sfc_{j,i}(\varpi_j -w^{\ii}_{k-1}\varpi_j, \varpi_j+w^\ii_{k-1}s_is_js_is_j\varpi_j)$,
\item\label{eq: (3)} $\La_{k+1,k+2} = (\varpi_j -w^{\ii}_{k-1}s_is_j\varpi_j, \varpi_i+w^\ii_{k-1}s_is_js_i\varpi_i),$ 
\item\label{eq: (8)} $ - \La_{(k+1)_\ii^-,k+2}  = -(\varpi_j -w^{\ii}_{k-1}\varpi_j, \varpi_i+w^\ii_{k-1}s_is_js_i\varpi_i)$,
\item\label{eq: (4)} $- \La_{k, k+3} =  -(\varpi_i -w^{\ii}_{k-1}s_i\varpi_i, \varpi_j+w^\ii_{k-1}s_is_js_is_j\varpi_j),$  
\item\label{eq: (6)} $\La_{k+2, k+3} = (\varpi_i -w^{\ii}_{k-1}s_is_js_i\varpi_i, \varpi_j+w^\ii_{k-1}s_is_js_is_j\varpi_j)$, 
\item\label{eq: (9)} $\sfc_{j,i}\La_{ k+1,k+3} = \sfc_{j,i} (\varpi_j -w^{\ii}_{k-1}s_is_j\varpi_j, \varpi_j+w^\ii_{k-1}s_is_js_is_j\varpi_j)$,
\item\label{eq: (10)} $- \sfc_{j,i} \La_{ (k+1)_\ii^-,k+3} = -\sfc_{j,i} (\varpi_j -w^{\ii}_{k-1}\varpi_j, \varpi_j+w^\ii_{k-1}s_is_js_is_j\varpi_j)$,
\item\label{eq: (11)} $\La_{k_\ii^-,k+1} = (\varpi_i -w^{\ii}_{k-1}\varpi_i, \varpi_j+w^\ii_{k-1}s_is_j\varpi_j)$,
\item\label{eq: (12)} $-  \La_{k_\ii^-, k+3} =- (\varpi_i -w^{\ii}_{k-1}\varpi_i, \varpi_j+w^\ii_{k-1}s_is_js_is_j\varpi_j)$,
\item\label{eq: (13)} $- \La_{k_\ii^-, (k+1)_\ii^-}  \overset{Lemma~\ref{Lem:Lambda}}{=}
- (\varpi_i -w^{\ii}_{k-1}\varpi_i, \varpi_j+w^\ii_{k-1}\varpi_j)$,
\ee

Using the techniques applied to the previous cases, we have
\begin{align*}
\eqref{eq: (1)}+\eqref{eq: (7)} & =  -( w^{\ii}_{k-1}s_i(\al_j), \varpi_i ), \allowdisplaybreaks\\
\eqref{eq: (2)}+\eqref{eq: (5)} & 
=  \sfc_{j,i}(\varpi_j ,w^\ii_{k-1}(\al_j)) -(\al_i, \al_j), \allowdisplaybreaks\\
\eqref{eq: (3)}+\eqref{eq: (8)} & 
= (\al_j,\al_i) + (w^{\ii}_{k-1}s_i(\al_j),\varpi_i), \allowdisplaybreaks\\
\eqref{eq: (4)}+\eqref{eq: (6)} &
 = (w^{\ii}_{k-1}(\al_j),\varpi_j)-\sfc_{j,i}(w^{\ii}_{k-1}\al_j,\varpi_j)+(\al_i,\al_j),  \allowdisplaybreaks\\
\eqref{eq: (9)}+\eqref{eq: (10)} 
& = \sfc_{j,i}( w^{\ii}_{k-1}\al_j, \varpi_j) - 2( w^{\ii}_{k-1} (\al_i), \varpi_j)-\sfc_{i,j}(\al_i,\varpi_i),  \allowdisplaybreaks\\
\eqref{eq: (11)}+\eqref{eq: (12)}+\eqref{eq: (13)}  
&=  -(\varpi_i,w^\ii_{k-1}(\varpi_j))+ (\varpi_i,w^\ii_{k-1}(\al_j)) +(  w^{\ii}_{k-1}\varpi_i, \varpi_j).
\end{align*}
Hence the summation of the $13$ terms becomes
\begin{align*}
\sum_{s=1}^{13} (s)& = -(\varpi_i,w^\ii_{k-1}(\varpi_j))+ (\varpi_i,w^\ii_{k-1}(\al_j)) +(  w^{\ii}_{k-1}\varpi_i, \varpi_j) \\
& \hspace{15ex}   - (w^{\ii}_{k-1}(\al_j),\varpi_j)  + \sfc_{j,i}(\varpi_j,w^\ii_{k-1}(\al_j)) = \La^{\ii'}_{k,k+1},
\end{align*}
which completes the proof of assertion~\eqref{it: eq(ix)}.
\end{proof}

As a consequence of the previous two lemmas, we have the following proposition:

\begin{proposition}
We have an isomorphism of $\bbA$-algebras
\[ \eta_k^*(= \mu_k^*\mu_{k+1}^* \mu_k^*  \sigma_{k}^* \sigma_{k+2}^*)\colon \calA_{\ii'} \isoto \calA_\ii \]
given by
\[ Z_u \mapsto  \mu_k^*\mu_{k+1}^*\mu_{k}^*(Z_{\sigma_{k+2}\sigma_{k}(u)}) \quad  \text{for all $u \in \N$},\]
which induces a bijection between the sets of quantum cluster monomials.
\end{proposition}

\begin{lemma} \label{Lem:4g}
 If $\sfz \in \calA_{\ii'}$ is a quantum cluster monomial with $\bfdeg(\sfz)=\bsg'=(g'_u)_{u \in \N}$, then the quantum cluster monomial
$\eta_k^*\sfz \in \calA_\ii$ has $\bfdeg(\eta_k^*\sfz)=\bsg=(g_u)_{u \in \N}$ as follows:
\begin{equation}\label{eq:4g}
\begin{aligned}
g_u = \begin{cases}
   g_{k+2}'-\sfc_{j,i}[g_{k+1}']_++[A]_+ &\text{ if } u=k+3, \\
  g_{k+3}' -[-g_{k+1}']_+ + [B]_+ &\text{ if } u=k+2,\\
 -A  + \sfc_{j,i} [-B]_+&\text{ if } u=k+1,\\
 -B &\text{ if } u=k,\\
g_{(k+1)^-_{\ii'}}'  +[g _{k+1}']_+ -[-B]_+ &\text{ if } u=k^-_\ii,\\
  g_{k^-_{\ii'}}'+[A]_+ - \sfc_{j,i}[B]_+ &\text{ if } u=(k+1)^-_\ii,\\
 g_u' & \text{ otherwise},
\end{cases}
\end{aligned}
\end{equation}
where
\begin{align} \label{eq:4g AB}
A = g_{k}' + \sfc_{j,i} [-g_{k+1}']_+ \quad \text{ and }\quad  B = -g_{k+1}'+\sfc_{i,j}[- A]_+.
\end{align}
\end{lemma}

\begin{proof}
By applying ~\eqref{eq: g-vector mutation} repeatedly, 
we obtain the assertion. 
\end{proof}

\begin{lemma} \label{lem: cone to cone 4g}
The map $\bsg' \mapsto \bsg$ given by~\eqref{eq:4g} sends the cone $C_{\ii'} \subset \Z^{\oplus \N}$ into the cone $C_{\ii} \subset \Z^{\oplus \N}$.    
\end{lemma}

\begin{proof}
To simplify the notation, for $\bsg = (g_u)_{u \in \N}$ and $\bsg' = (g'_u)_{u \in \N}$, we set
\begin{align}\label{eq: Sigma_u}
\Sigma_u \seteq \sum_{v \ge u, i_v = i_u} g_v, \quad \Sigma'_u \seteq \sum_{v \ge u, i'_v = i'_u} g'_v    
\end{align}
for each $u \in \N$. Throughout this proof, we shall use the notations $A$ and $B$ in~\eqref{eq:4g AB}. 

We have to show that $\Sigma'_u \ge 0$ $(\forall u \in \N)$ implies $\Sigma_u \ge 0$ $(\forall u \in \N)$.
\bna
\item If $u > k+3$, we have $\Sigma_u = \Sigma'_u \ge 0$.
\item If $u=k+3$, we have $\Sigma_{k+3} =  \Sigma'_{k+2}  -\sfc_{j,i}[g_{k+1}']_++[A]_+ \ge 0 $.
\item If $u=k+2$, the assertion holds when $ g'_{k+1} \ge 0  $ obviously. Let us assume further $g'_{k+1}  < 0$. Then we have
\begin{align*}
&  \Sigma_{k+2} =  \Sigma'_{k+3} + g_{k+1}' + [B]_+  =  \Sigma'_{k+1} + [B]_+,
\end{align*}
which implies the assertion for $u=k+2$.
\item If $u=k+1$, we have
\begin{equation} \label{eq: Sigma k+1}
\begin{aligned}
\Sigma_{k+1} & =   \Sigma'_{k+2} -\sfc_{j,i}[g_{k+1}']_++[A]_+  -A  + \sfc_{j,i} [-B]_+ \\
& =   \Sigma'_{k+2} -\sfc_{j,i}[g_{k+1}']_++[-A]_+   + \sfc_{j,i} [-B]_+.
\end{aligned}
\end{equation}
Thus the assertion holds when $B \ge 0$. Assume that $B<0$. Then we have
\begin{align*}
\Sigma_{k+1} & =  \Sigma'_{ k+2 }  -\sfc_{j,i}([g_{k+1}']_+ -g'_{k+1}) + [-A]_+    -2[-A]_+   \\
& =  \Sigma'_{ k+2 }  -\sfc_{j,i}[-g_{k+1}']_+ -[-A]_+.
\end{align*}
Thus the assertion holds when  $B<0$ and $A>0$. Let us consider the case when  $B<0$ and $A<0$. Then we have
\begin{align*}
\Sigma_{k+1} & =  \Sigma'_{k+2} -\sfc_{j,i}[g_{k+1}']_+ -A    - \sfc_{j,i} B \allowdisplaybreaks\\
& =  \Sigma'_{k+2} -\sfc_{j,i}[g_{k+1}']_+ -A    + \sfc_{j,i} g_{k+1}' +2A \allowdisplaybreaks\\
& =  \Sigma'_{k+2} -\sfc_{j,i}[g_{k+1}']_+    + \sfc_{j,i} g_{k+1}' + g_{k}' + \sfc_{j,i} [-g_{k+1}']_+  = \Sigma'_{ k },
\end{align*}
which completes the assertion.
\item If $u=k$, we have
\begin{align}
\Sigma_{k} & =   \Sigma'_{k+1} - [g'_{k+1}]_+ +  [-B]_+  \label{eq: Sigma k}\\
                         &=  \delta(g'_{k+1} \ge 0)  \Sigma'_{k+3}+  \delta(g'_{k+1}<0)  \Sigma'_{k+1} +  [-B]_+\ge 0, \nonumber
\end{align}
which implies the assertion.
\item \label{it: last case 4g} If $u<k$,  we have
\begin{align*}
 g_{k+3} + g_{k+1} + g_{(k+1)^-_\ii} 
& =  g_{k+2}'-\sfc_{j,i}[g_{k+1}']_+  +2[A]_+-A    +   g_{k^-_{\ii'}}'  - \sfc_{j,i}B \\
& = g_{k+2}'   + g_{k}'      +   g_{k^-_{\ii'}}'
\end{align*}
and
\begin{align*}
 g_{k+2} + g_{k} + g_{k^-_\ii}  = g_{k+3}' + g_{k+1}'   + g_{(k+1)^-_{\ii'}}', 
\end{align*}
and thus $\Sigma_u = \Sigma'_u$, which implies the assertion. \qedhere
\ee  
\end{proof}

By~\eqref{it: last case 4g} in the proof of Lemma~\ref{lem: cone to cone 4g}, we have the following corollary:

\begin{corollary} \label{Cor: i-degree}
Assume that $k^- \in \N$. Then we have
$$
\sfp_{\ii'}(\bsg'; i) = \sfp_{\ii}(\bsg; i) \quad \text{ for all } i\in I, 
$$
where $\bsg'\mapsto \bsg$ under the map given by~\eqref{eq:4g}.  
\end{corollary}

\begin{lemma} \label{lem: 4g Pgi}
The elements $\bsg,\bsg' \in \Z^{\oplus\N}$ related by~\eqref{eq:4g} satisfy
the followings: 
\ben
\item  \label{it: 4move 1 non-trivial} When $i = i_{k+1}$ and $(k+1)^-_\ii =0$, we have
$$
\sfp_{\ii}(\bsg;i) = \bc
\sfp_{\ii'}(\bsg';i) -[g_k']_+ & \text{ if }  g_{k+1}' \ge 0, \\
\sfp_{\ii'}(\bsg';i) -g_k' & \text{ if }  g_{k+1}'<0, g_k' \ge 0 \text{ and } A \ge 0, \\
\sfp_{\ii'}(\bsg';i)  -2g_{k}' + \sfc_{j,i} g_{k+1}'  & \text{ if }  g_{k+1}'<0, g_k' \ge 0, A < 0 \text{ and } B \ge 0, \\
\sfp_{\ii'}(\bsg';i) & \text{ otherwise}.
\ec
$$ 
\item \label{it: 4move 2 non-trivial} When $i = i_{k}$ and $k^-_\ii =0$, we have
$$
\sfp_{\ii}(\bsg;i) = \bc
\sfp_{\ii'}(\bsg';i) - g'_{k+1} & \text{ if } g_{k+1}' \ge 0, B \ge 0, \\
\sfp_{\ii'}(\bsg';i) -\sfc_{i,j}[- g_{k}']_+      & \text{ if } g_{k+1}' \ge 0, B < 0, \\
\sfp_{\ii'}(\bsg';i) &  \text{ if } g_{k+1}' < 0, B \ge 0, \\
\sfp_{\ii'}(\bsg';i)  -g_{k+1}'+\sfc_{i,j}[- g_{k}']_+  & \text{ if } g_{k+1}' < 0, B < 0.
\ec
$$
\item Except above two cases, we have 
$$
\sfp_{\ii}(\bsg;i) = \sfp_{\ii'}(\bsg';i).
$$
\ee
Here
$$
A = g_{k}' + \sfc_{j,i} [-g_{k+1}']_+ \quad \text{ and }\quad  B = -g_{k+1}'+\sfc_{i,j}[- A]_+.
$$
\end{lemma}

\begin{proof}
Since the proofs for $\gamma_k$ and $\beta_k$ are given in~\cite[Lemma 2.18]{FHOO2}, let us focus on the case $\eta_k$.

\noindent
(1) Let us first consider when $i=i_{k+1}$ and $(k+1)^-_\ii =0$. 
In~\eqref{eq: Sigma k+1}, we see that 
\begin{align*}
\Sigma_{k+1} =   \Sigma'_{k+2} -\sfc_{j,i}[g_{k+1}']_+ +[-A]_+   + \sfc_{j,i} [-B]_+.
\end{align*}

\noindent
(i) If $g_{k+1}' \ge 0$, then $A=g_k'$,$[-B]_+= g'_{k+1}-\sfc_{i,j}[-g_k']_+$ and 
\begin{align*}
\Sigma_{k+1} &=   \Sigma'_{k+2} -\sfc_{j,i}g_{k+1}' +[-g_k']_+   + \sfc_{j,i}(g'_{k+1}-\sfc_{i,j}[-g_k']_+)  \\   
& = \Sigma'_{k+2} - [-g_k']_+ = \Sigma'_{k}-[g_k']_+ . 
\end{align*}

\noindent
(ii)  If $g_{k+1}' < 0$, then $A = g_{k}' - \sfc_{j,i} g_{k+1}'$ and $  B = -g_{k+1}'+\sfc_{i,j}[- A]_+.$

\smallskip

(ii-a) Let us assume that $g_{k}' < 0$ and $A \le 0$. Then $ B = g_{k+1}'-\sfc_{i,j}g_{k}' < 0$ and 
\begin{align*}
\Sigma_{k+1} & =   \Sigma'_{k+2} -\sfc_{j,i}[g_{k+1}']_+ +[-A]_+   + \sfc_{j,i} [-B]_+    
= \Sigma'_{k+2}  + g_{k}'= \Sigma'_{k} 
\end{align*}

(ii-b) Let us assume that $g_{k}' \ge 0$ and $A \ge 0$.  Then
$B = -g_{k+1}'$ and 
$$
\Sigma_{k+1} =   \Sigma'_{k+2} = \Sigma'_{k}- g_k'.    
$$

(ii-c) Finally, assume that $g_{k}' \ge 0$ and $A < 0$. Then  
$B =  g_{k+1}'-\sfc_{i,j}g_{k}'$. 
\begin{align*}
\Sigma_{k+1} &=   \Sigma'_{k+2} -g_{k}' + \sfc_{j,i} g_{k+1}' + \sfc_{j,i} [-B]_+ 
= \bc   
\Sigma'_{k}  -2g_{k}' + \sfc_{j,i} g_{k+1}'   & \text{ if } B \ge 0, \\
\Sigma'_{k}    & \text{ if } B < 0.
\ec
\end{align*}
\noindent
(2) Let us consider when $i=i_{k}$.
In~\eqref{eq: Sigma k}, we see that 
\begin{align*}
\Sigma_{k} =   \Sigma'_{k+1} - [g'_{k+1}]_+ +  [-B]_+.
\end{align*}
Then we have 
%
\begin{align*}
\Sigma_{k} = 
\bc
\Sigma'_{k+1} - g'_{k+1} & \text{ if } g_{k+1}' \ge 0, B \ge 0, \\
\Sigma'_{k+1} -\sfc_{i,j}[- g_{k}']_+      & \text{ if } g_{k+1}' \ge 0, B < 0, \\
\Sigma'_{k+1} &  \text{ if } g_{k+1}' < 0, B \ge 0, \\
\Sigma'_{k+1}  -g_{k+1}'+\sfc_{i,j}[- g_{k}']_+  & \text{ if } g_{k+1}' < 0, B < 0,
\ec  
\end{align*}
which implies our assertion.
\end{proof}

\subsection{$6$-moves} \label{subsec: G2}
For sequences $\ii=(i_u)_{u \in \N}$ and $\ii'=(i'_u)_{u \in \N} \in I^{(\infty)}$,
we say that $\ii'$ can be obtained from $\ii$ via a \emph{$6$-move}, if there
exists a unique $k \in \N$ such that
\begin{align*}
& i_k=i_{k+2}=i_{k+4}=i'_{k+1}=i'_{k+3}=i'_{k+5},  \\
& i_{k+1}=i_{k+3}=i_{k+5}=i'_{k}=i'_{k+2}=i'_{k+4}, \\
& i_u=i'_u \ (u \not\in [ k,k+5]) 
\end{align*}
and $$\sfc_{i_{k+1},i_{k}}\sfc_{i_k,i_{k+1}}=3.$$ 
In this case, we write $\ii'=\zeta_k\ii$.
Throughout this subsection, we assume that $\ii,\ii' \in I^{(\infty)}$ are related by $\ii'=\zeta_k\ii$ for some $k \in \N$
 and we sometimes write $i \seteq i_k$ and $j \seteq i_{k+1}$.

 \begin{lemma} \label{Lem:zeta1}
We have 
\begin{align} \label{eq: 6-move}
\tB^{\ii'}= \sigma_{k+4}\sigma_{k+2}\sigma_{k}
 \mu_{k}  \mu_{k+3}  \mu_{k+2} \mu_{k}  \mu_{k+1}  \mu_{k+3} \mu_{k}  \mu_{k+2}  \mu_{k+1} \mu_{k}  \tB^{\ii}.
\end{align}
\end{lemma}  

\begin{proof}
As in the $4$-move case, it is enough to show that
$$
\mu_{k}  \mu_{k+3}  \mu_{k+2} \mu_{k}  \mu_{k+1}  \mu_{k+3} \mu_{k}  \mu_{k+2}  \mu_{k+1} \mu_{k} \tGamma^\ii|_{\calK}
= \sigma_{k+4} \sigma_{k+2}\sigma_{k}\tGamma^{\ii'}|_\calK
$$
where $\calK = [k,k+5] \sqcup \{ k^-_\ii, (k+1)^-_\ii \}$.

We shall only consider the case when both $k^-_\ii$ and $(k+1)^-_\ii$ are nonzero since the proofs for the other cases are similar.
In that case, the valued quiver $\tGamma^\ii |_\calK$  is depicted as:
\begin{align} \label{eq: general zeta}
\tGamma^\ii |_\calK = \scalebox{0.87}{\raisebox{6mm}{
\xymatrix@R=3ex@C=3ex{
& k+4  \ar[dr]|{\ulcorner b,-a \lrcorner } && k+2  \ar[ll] \ar[dr]|{\ulcorner b,-a \lrcorner } && k \ar[ll] \ar[dr]|{\ulcorner b,-a \lrcorner } && k_\ii^- \ar[ll] \\
k+5  \ar@{.>}[ur]|{\ulcorner a,-b \lrcorner } && k+3  \ar[ll]  \ar[ur]|{\ulcorner a,-b \lrcorner }  && k+1 \ar[ll] \ar[ur]|{\ulcorner a,-b \lrcorner } && (k+1)_\ii^- \ar[ll] \ar@{.>}[ur]|{\ulcorner a,-b \lrcorner }
}}}
\end{align}
where $a = -\sfc_{i_{k+1},i_k}$ and $b = -\sfc_{i_k,i_{k+1}}$.  Here
\begin{itemize}
\item the dotted arrow between $k_\ii^-$ and  $(k+1)_\ii^-$ exists if and only if $k_\ii^- < (k+1)_\ii^-$,
\item the dotted arrow between $k+5$ and  $k+4$ exists if and only if  $(k+4)_\ii^+ < (k+5)_\ii^+$.
\end{itemize}

Let us consider the case when $i_k =1$, $i_{k+1}=2$, $\g=G_2$ and dotted arrows exist.  Then~\eqref{eq: general zeta} can be drawn as follows with Convention~\ref{conv: valued quiver}:
\begin{align*}
\tGamma^\ii |_\calK =  \scalebox{0.87}{\raisebox{6mm}{
\xymatrix@R=3ex@C=3ex{
& k+4  \ar@{=>}[dr]  && k+2  \ar[ll] \ar@{=>}[dr] && k \ar[ll] \ar@{=>}[dr]  && k_\ii^- \ar[ll] \\
k+5  \ar[ur]|{\llcorner 3 }  && k+3  \ar[ll]  \ar@{-}[ul] \ar[ur]|{\llcorner 3 }  && k+1 \ar[ll]   \ar@{-}[ul] \ar[ur]|{\llcorner 3 }  && (k+1)_\ii^- \ar[ll] \ar[ur]|{\llcorner 3 } \ar@{-}[ul]
}}}
\end{align*}
From now on, we shall depict the sequence of mutations on $\tGamma^\ii |_\calK$ in turn:
\begin{align*}
\overset{\mu_k}{\to} &
 \scalebox{0.87}{\raisebox{6mm}{
\xymatrix@R=3ex@C=3ex{
& k+4  \ar@{=>}[dr]  && k+2  \ar[ll] \ar[rr]   && k  \ar[rr] \ar@{=>}[dl]  && k_\ii^-   \ar@/_1pc/[llll] \\
k+5  \ar[ur]|{\llcorner 3 }  && k+3  \ar[ll]  \ar@{-}[ul] \ar[ur]|{\llcorner 3 }  && k+1 \ar[ll]  \ar@{->}[rr]_2    \ar@{-}[ur] && (k+1)_\ii^-   \ar[ul]|{3 \lrcorner }
}}} \allowdisplaybreaks \\
\overset{\mu_{k+1}}{\to} &
 \scalebox{0.87}{\raisebox{6mm}{
\xymatrix@R=3ex@C=3ex{
& k+4  \ar@{=>}[dr]  && k+2  \ar[ll] \ar[rr]   && k  \ar[rr] \ar@{=>}[dr]  \ar@{=>}[dlll]  && k_\ii^-   \ar@/_1pc/[llll] \\
k+5  \ar[ur]|{\llcorner 3 }  && k+3  \ar@{-}[urrr] \ar[ll]  \ar@{-}[ul] \ar[ur]|{\llcorner 3 } \ar[rr]  && k+1  \ar[ur]|{\llcorner 3 }  \ar@{<-}[rr]_2     && (k+1)_\ii^-   \ar@{-}[ul]
}}} \allowdisplaybreaks \\
\overset{\mu_{k+2}}{\to} &
 \scalebox{0.87}{\raisebox{6mm}{
\xymatrix@R=3ex@C=3ex{
& k+4    \ar[rr]  && k+2  \ar@/_1pc/[rrrr]  \ar@{=>}[dl] && k  \ar[ll] \ar@{=>}[dr]   && k_\ii^-   \ar@/_1pc/[llllll] \\
k+5  \ar[ur]|{\llcorner 3 }  && k+3   \ar[ll]   \ar@{-}[ur]  \ar[rr]  && k+1  \ar[ur]|{\llcorner 3 }  \ar@{<-}[rr]_2     && (k+1)_\ii^-   \ar@{-}[ul]
}}} \allowdisplaybreaks \\
\overset{\mu_{k}}{\to} &
 \scalebox{0.87}{\raisebox{6mm}{
\xymatrix@R=3ex@C=3ex{
& k+4    \ar[rr]  && k+2  \ar@/_1pc/[rrrr]  \ar@{=>}[dl]  \ar[rr]   && k \ar@{=>}[dl]   && k_\ii^-   \ar@/_1pc/[llllll] \\
k+5  \ar[ur]|{\llcorner 3 }  && k+3   \ar[ll]   \ar@{-}[ur]  \ar[rr]  && k+1 \ar[ul]|{3 \lrcorner }    \ar@{-}[ur]  \ar[rr]     && (k+1)_\ii^-    \ar[ul]|{ 3 \lrcorner }
}}} \allowdisplaybreaks \\
\overset{\mu_{k+3}}{\to} &
 \scalebox{0.87}{\raisebox{6mm}{
\xymatrix@R=3ex@C=3ex{
& k+4    \ar[rr]  && k+2  \ar@/_1pc/[rrrr]  \ar@{=>}[dlll]  \ar[rr]   && k \ar@{=>}[dl]   && k_\ii^-   \ar@/_1pc/[llllll] \\
k+5  \ar[ur]|{\llcorner 3 } \ar@{-}[urrr]  \ar[rr] && k+3     \ar[ur]|{\llcorner 3}  && k+1  \ar[ll]      \ar@{-}[ur]  \ar[rr]     && (k+1)_\ii^-    \ar[ul]|{ 3 \lrcorner }
}}} \allowdisplaybreaks \\
\overset{\mu_{k+1}}{\to} &
 \scalebox{0.87}{\raisebox{6mm}{
\xymatrix@R=3ex@C=3ex{
& k+4    \ar[rr]  && k+2  \ar@/_1pc/[rrrr]  \ar@{=>}[dlll]  \ar[rr]   && k \ar@{=>}[dlll]   && k_\ii^-   \ar@/_1pc/[llllll] \\
k+5  \ar[ur]|{\llcorner 3 } \ar@{-}[urrr]  \ar[rr] && k+3  \ar@{-}[urrr]   \ar[ur]|{\llcorner 3} \ar[rr]  && k+1      \ar[ur]|{\llcorner 3}       && (k+1)_\ii^-    \ar[ll]
}}} \allowdisplaybreaks \\
\overset{\mu_{k}}{\to} &
 \scalebox{0.87}{\raisebox{6mm}{
\xymatrix@R=3ex@C=3ex{
& k+4    \ar[rr]  && k+2  \ar@/_1pc/[rrrr]  \ar@{=>}[dlll]      && k \ar@{=>}[dl] \ar[ll]  && k_\ii^-   \ar@/_1pc/[llllll] \\
k+5  \ar[ur]|{\llcorner 3 } \ar@{-}[urrr]   \ar[rr] && k+3  \ar[urrr]|{\llcorner 3}      && k+1  \ar@{->}[ll]^2    \ar@{-}[ur]        && (k+1)_\ii^-    \ar[ll]
}}} \allowdisplaybreaks \\
\overset{\mu_{k+2}}{\to} &
 \scalebox{0.87}{\raisebox{6mm}{
\xymatrix@R=3ex@C=3ex{
& k+4     && k+2  \ar@{<-}@/^1pc/[rrrr]    \ar[ll]\ar[rr]     && k \ar@{=>}[dl] \ar@{=>}[dlllll]    \ar[rr] && k_\ii^-     \\
k+5  \ar@{-}[urrrrr] \ar[urrr]|{\llcorner 3}   \ar[rr] && k+3  \ar[urrr]|{\llcorner 3}      && k+1  \ar@{->}[ll]^2    \ar@{-}[ur]        && (k+1)_\ii^-    \ar[ll]
}}} \allowdisplaybreaks \\
\overset{\mu_{k+3}}{\to} &
 \scalebox{0.87}{\raisebox{6mm}{
\xymatrix@R=3ex@C=3ex{
& k+4     && k+2  \ar@{<-}@/^1pc/[rrrr]    \ar[ll]\ar[rr]     && k   \ar@{=>}[dlll]    \ar[rr] && k_\ii^-     \\
k+5    \ar[urrr]|{\llcorner 3}   && k+3  \ar@{-}[urrr]    \ar[ll]     && k+1  \ar@{<-}[ll]^2    \ar[ur]|{\llcorner 3}        && (k+1)_\ii^-    \ar[ll]
}}} \allowdisplaybreaks \\
\overset{\mu_{k}}{\to} &
 \scalebox{0.87}{\raisebox{6mm}{
\xymatrix@R=3ex@C=3ex{
& k+4     && k+2      \ar[ll]  \ar@{=>}[dl]     && k   \ar@{=>}[dl]    \ar[ll] && k_\ii^-   \ar[ll]  \\
k+5    \ar[urrr]|{\llcorner 3}   && k+3  \ar[urrr]|{\llcorner 3}     \ar@{-}[ur] \ar[ll]     && k+1  \ar[ll]  \ar@{-}[ur]  \ar[urrr]|{\llcorner 3}        && (k+1)_\ii^-    \ar[ll]
}}}.
\end{align*}

On the other hand, the valued quiver $\tGamma^{\ii'}$ is depicted as:
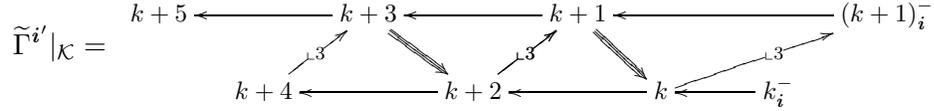
\begin{align*}  
\tGamma^{\ii'} |_\calK =
\scalebox{0.87}{\raisebox{6mm}{
\xymatrix@R=3ex@C=3ex{
 k+5     && k+3  \ar@{=>}[dr]  \ar[ll]       && k+1 \ar@{=>}[dr]     \ar[ll] &&&& (k+1)_\ii^-   \ar[llll]  \\
&k+4 \ar[ur]|{\llcorner 3}     && k+2 \ar@{-}[ul] \ar[ur]|{\llcorner 3}  \ar[ll] \ar[ur]|{\llcorner 3}    && k \ar@{-}[ul]  \ar[ll]     \ar[urrr]|{\llcorner 3}    && k_\ii^-    \ar[ll]
}}}
\end{align*}
Thus our assertion follows. The remaining cases can be proved in a similar way.
\end{proof}

\begin{lemma} \label{Lem:zeta2}
We have
\begin{align} \label{eq: general gamma B}
\La^{\ii'}  = \sigma_{k+4}\sigma_{k+2}\sigma_{k}   \mu_{k}  \mu_{k+3}  \mu_{k+2} \mu_{k}  \mu_{k+1}  \mu_{k+3} \mu_{k}  \mu_{k+2}  \mu_{k+1} \mu_{k}\La^{\ii}.    
\end{align} 
\end{lemma}

\begin{proof}
To simplify notation, let us denote $\Lambda^{\ii}$ as $\Lambda$. Unless $k \leq u \leq k+5$ or $k \leq v \leq k+5$, the expressions $\Lambda_{u,v} = \Lambda^{\ii'}_{u,v}$ hold true, and the values of $\Lambda_{u,v}$ remain unchanged under the operations $\mu_{k}$, $\mu_{k+1}$, $\mu_{k+2}$, $\mu_{k+3}$, $\sigma_{k}$, $\sigma_{k+2}$, and $\sigma_{k+4}$. Therefore, for simplicity, we will only consider $\Lambda_{u,v}$ when $k \leq u \leq k+5$ or $k \leq v \leq k+5$, which we will assume from now on. Furthermore, due to the skew-symmetry of $\Lambda$ and $\Lambda^{\ii'}$, we can assume further $u < v$.

Recall the definition
\begin{equation*}
\La_{u,v} =  (\varpi_{i_u}-w_u\varpi_{i_u}, \varpi_{i_v}+w_v\varpi_{i_v}) \quad \text{for $u< v$}.
\end{equation*}
Since the number of elements in the set $\weyl$ is 12 and the sequence $\ii$ contains only $1$ and $2$, we will demonstrate that there exist only a limited number of finite subsequences of $\ii$ that need to be examined to directly verify the assertion.

Consider a fixed set $\underline{\weyl}$ consisting of reduced sequences of elements in $\weyl$. Let us examine how $\Lambda_{u,v}$ is determined. To avoid minor exceptions, assume $k \ge 4$. If $u < k$ and $k \le v \le k+5$, then $\Lambda_{u,v}$ is determined by the subsequence
\[ w_1 \, a \, w_2 \, i_k \, i_{k+1} \cdots i_{k+5} , \]
where $w_1, w_2 \in \underline{\weyl}$, $i_u = a$, and $i_v = i_\ell$ for some $\ell = k, k+1, \dots, k+5$. Note that it is sufficient to consider $w_1, w_2 \in \underline{\weyl}$, as $\Lambda_{u,v}$ does not change with other expressions of the Weyl group elements represented by $w_1$ and $w_2$.

If $k \le u, v \le k+5$ then $\La_{u,v}$ is determined by \[ i_k \, i_{k+1} \cdots i_{k+5} .\]
If $k \le u \le k+5$ and $v>k+5$ then $\La_{u,v}$ is determined by the subsequence
\[ i_k \, i_{k+1} \cdots i_{k+5}w_3 d , \]
where  $w_3 \in \underline{\weyl}$, $i_v=d$, $i_u=i_\ell$ for some $\ell=k,k+1, \dots , k+5$.

All three cases mentioned above are encompassed in the sequences given by
\begin{equation} \label{eqn-redun}  {w_1} \, a  \, {w_2} \, i_k  \, i_{k+1} \cdots i_{k+5}  \, {w_3} \, d,  \end{equation}
where $w_1, w_2, w_3 \in \underline{\weyl}$,  $a,d \in \{1,2\}$ and \[ (i_k, i_{k+1}, \dots ,i_{k+5}) \in \{ (1,2,1,2,1,2), (2,1,2,1,2,1) \}.\]

However, to apply $\mu_k, \mu_{k+1}, \mu_{k+2}, \mu_{k+3}$, it is crucial to keep track of the corresponding mutations of $\tB^\ii$. As demonstrated in the proof of Lemma \ref{Lem:zeta1}, we only need to consider $k_\ii^-$ and $(k+1)_\ii^-$ in addition to $k, k+1, \dots , k+5$ when dealing with mutations (note that $(k+4)_\ii^+$ and $(k+5)_\ii^+$ are not relevant). In the matrices obtained from the mutations of $\tB^\ii$ associated with the mutations of $\Lambda^{\ii}$ in the statement of this lemma, we observe that
\begin{equation} \label{eqn-ree}
\text{the $(k,j)$- and $(k+1, j)$-entries are zero for $j < k$ unless $j=k_\ii^-$ or $j=(k+1)\ii^-$.}
\end{equation}
Therefore, these entries do not contribute to the mutations of $\Lambda_{u,v}$ by definition.

Now consider a sequence \[ \jj_1 a \jj_2  \, i_k  \, i_{k+1} \cdots i_{k+5}  \, \jj_3 d,\]
where $\jj_1=i_1\, i_2 \cdots i_{u-1}$, $i_u=a$, $i_v=d$ and $\jj_2$ and $\jj_3$ are finite sequences.
Let $\{   b,   c \} = \{ i_{k_\ii^-}, i_{(k+1)_\ii^-} \}$. Without loss of generality, we may consider the following three cases:
\begin{eqnarray*}
&&\text{ (i) $  \jj_1 a \cdots   b  c \cdots   c  \, i_k  \, i_{k+1} \cdots i_{k+5}  \, \jj_3 d$}, \\
&&\text{(ii) $  \jj_1 a   c \cdots   c  \, i_k  \, i_{k+1} \cdots i_{k+5}  \, \jj_3 d$ \qquad with $a =   b$}, \\
&&\text{(iii) $  \cdots   ba \cdots   a  \, i_k  \, i_{k+1} \cdots i_{k+5}  \, \jj_3 d$ \qquad  with $a=  c$ }.
\end{eqnarray*}
Applying the observations  \eqref{eqn-redun} and \eqref{eqn-ree} in determining $\La_{u,v}$ and their mutations, we only need to consider
\begin{eqnarray}
&&\text{ (i) $ \leadsto \  w_1 a w_2    b   c  \, i_k  \, i_{k+1} \cdots i_{k+5}  \, w_3 d$}, \label{eqn-sto-1} \\
&&\text{(ii) $ \leadsto \   w_1 a   c  \, i_k  \, i_{k+1} \cdots i_{k+5}  \, w_3 d \quad \text{ or } \quad  w_1 a   c    c  \, i_k  \, i_{k+1} \cdots i_{k+5}  \, w_3 d$}, \label{eqn-sto-2} \\
&&\text{(iii) $ \leadsto \  w_1  a  \, i_k  \, i_{k+1} \cdots i_{k+5}  \, w_3 d$ }, \label{eqn-sto-3}
\end{eqnarray}
where $w_1, w_2, w_3 \in \underline{\weyl}$,  $a,d \in \{1,2\}$, $(b,c) \in \{(1,2),(2,1) \}$ and \[ (i_k, i_{k+1}, \dots ,i_{k+5}) \in \{ (1,2,1,2,1,2), (2,1,2,1,2,1) \}.\]

We consider the submatrices of $\La^{\ii}$ given by the finite sequences in \eqref{eqn-sto-1}-\eqref{eqn-sto-3}.
These sequences are not mutually exclusive, and we deliberately maintain this redundancy for convenience, as it is unnecessary to remove it when directly checking using a computer. 
The total number of sequences in \eqref{eqn-sto-1}-\eqref{eqn-sto-3} is 62,208. If $k\le 3$, it is easy to list all possible sequences. We have used a computer to verify the 
assertion of the lemma for all these finite subsequences, which took less than 12 minutes on a regular laptop (single CPU).
\end{proof}

\begin{remark}
The above lemma assumes a general situation. For our purpose in later sections, we need only special $\ii$
and $\ii'$ in this remark (see Remark~\ref{rmk: our purpose} below also), 
and hence we \emph{do not} need to use a computer. 
Consider $\ii=(1,2,1,2,1,2,1,2,\ldots)$ and $\ii'=(1,2,2,1,2,1,2,1,\ldots)$ with $k=3$. Then we obtain 
\begin{eqnarray*} 
\La^{\ii,8}={\tiny \left(\begin{array}{rrrrrrrr}
0 & -3 & -1 & -3 & 0 & 0 & 2 & 3 \\
3 & 0 & 0 & -3 & 0 & 0 & 3 & 6 \\
1 & 0 & 0 & -3 & 0 & 0 & 3 & 6 \\
3 & 3 & 3 & 0 & 0 & 0 & 3 & 9 \\
0 & 0 & 0 & 0 & 0 & 0 & 2 & 6 \\
0 & 0 & 0 & 0 & 0 & 0 & 0 & 6 \\
-2 & -3 & -3 & -3 & -2 & 0 & 0 & 3 \\
-3 & -6 & -6 & -9 & -6 & -6 & -3 & 0
\end{array}\right)} \xrightarrow[]{\mu_3} {\tiny \left(\begin{array}{rrrrrrrr}
0 & -3 & -2 & -3 & 0 & 0 & 2 & 3 \\
3 & 0 & 0 & -3 & 0 & 0 & 3 & 6 \\
2 & 0 & 0 & 0 & 0 & 0 & 2 & 6 \\
3 & 3 & 0 & 0 & 0 & 0 & 3 & 9 \\
0 & 0 & 0 & 0 & 0 & 0 & 2 & 6 \\
0 & 0 & 0 & 0 & 0 & 0 & 0 & 6 \\
-2 & -3 & -2 & -3 & -2 & 0 & 0 & 3 \\
-3 & -6 & -6 & -9 & -6 & -6 & -3 & 0
\end{array}\right)} \xrightarrow[]{\mu_3\mu_5\mu_4} \allowdisplaybreaks\\  {\tiny
 \left(\begin{array}{rrrrrrrr}
0 & -3 & -1 & -3 & 0 & 0 & 2 & 3 \\
3 & 0 & 3 & 3 & 3 & 0 & 3 & 6 \\
1 & -3 & 0 & 0 & 1 & 0 & 1 & 3 \\
3 & -3 & 0 & 0 & 3 & 0 & 3 & 9 \\
0 & -3 & -1 & -3 & 0 & 0 & 0 & 3 \\
0 & 0 & 0 & 0 & 0 & 0 & 0 & 6 \\
-2 & -3 & -1 & -3 & 0 & 0 & 0 & 3 \\
-3 & -6 & -3 & -9 & -3 & -6 & -3 & 0
\end{array}\right)} \xrightarrow[]{\mu_3\mu_4\mu_6} {\tiny
 \left(\begin{array}{rrrrrrrr}
0 & -3 & 1 & 0 & 0 & 0 & 2 & 3 \\
3 & 0 & 6 & 6 & 3 & 9 & 3 & 6 \\
-1 & -6 & 0 & 0 & -1 & -3 & -1 & 0 \\
0 & -6 & 0 & 0 & 0 & 0 & 0 & 0 \\
0 & -3 & 1 & 0 & 0 & 0 & 0 & 3 \\
0 & -9 & 3 & 0 & 0 & 0 & 0 & 3 \\
-2 & -3 & 1 & 0 & 0 & 0 & 0 & 3 \\
-3 & -6 & 0 & 0 & -3 & -3 & -3 & 0
\end{array}\right)} \xrightarrow[]{\mu_3\mu_6\mu_5} \allowdisplaybreaks\\ {\tiny
 \left(\begin{array}{rrrrrrrr}
0 & -3 & 2 & 0 & 3 & 3 & 2 & 3 \\
3 & 0 & 6 & 6 & 6 & 9 & 3 & 6 \\
-2 & -6 & 0 & 0 & 0 & 0 & 0 & 0 \\
0 & -6 & 0 & 0 & 0 & 0 & 0 & 0 \\
-3 & -6 & 0 & 0 & 0 & 3 & -1 & 0 \\
-3 & -9 & 0 & 0 & -3 & 0 & -3 & -3 \\
-2 & -3 & 0 & 0 & 1 & 3 & 0 & 3 \\
-3 & -6 & 0 & 0 & 0 & 3 & -3 & 0
\end{array}\right)}  \xrightarrow[]{\sigma_7\sigma_5\sigma_3}  {\tiny
\left(\begin{array}{rrrrrrrr}
0 & -3 & 0 & 2 & 3 & 3 & 3 & 2 \\
3 & 0 & 6 & 6 & 9 & 6 & 6 & 3 \\
0 & -6 & 0 & 0 & 0 & 0 & 0 & 0 \\
-2 & -6 & 0 & 0 & 0 & 0 & 0 & 0 \\
-3 & -9 & 0 & 0 & 0 & -3 & -3 & -3 \\
-3 & -6 & 0 & 0 & 3 & 0 & 0 & -1 \\
-3 & -6 & 0 & 0 & 3 & 0 & 0 & -3 \\
-2 & -3 & 0 & 0 & 3 & 1 & 3 & 0
\end{array}\right)} = \La^{\ii',8}. 
\end{eqnarray*}
\end{remark}

As a consequence of the previous two lemmas, we have the following proposition:

\begin{proposition}
We have an isomorphism of $\bbA$-algebras
\begin{align}\label{eq: the G2 mutation}
 \zeta_k^*(=
\mu_{k}^* \mu_{k+1}^* \mu_{k+2}^* \mu_{k}^* \mu_{k+3}^* \mu_{k+1}^* \mu_{k}^* \mu_{k+2}^* \mu_{k+3}^* \mu_{k}^*
 \sigma_{k}^* \sigma_{k+2}^*\sigma_{k+4}^*)\colon \calA_{\ii'} \isoto \calA_\ii
\end{align}
given by
\[ Z_u \mapsto
\mu_{k}^* \mu_{k+1}^* \mu_{k+2}^* \mu_{k}^* \mu_{k+3}^* \mu_{k+1}^* \mu_{k}^* \mu_{k+2}^* \mu_{k+3}^* \mu_{k}^*
(Z_{\sigma_{k+4}\sigma_{k+2}\sigma_{k}(u)}) \quad  \text{for all $u \in \N$},\]
which induces a bijection between the sets of quantum cluster monomials.
\end{proposition}

\begin{remark}
There are another $31$ sequences of mutations of length $10$, which give rise to mutation equivalence from $\calA_{\ii'} \to \calA_{\ii}$. For instance, we have
$$
\mu_{k+3}^*\mu_{k+2}^*\mu_{k+1}^*\mu_{k+3}^*\mu_{k}^*\mu_{k+2}^*\mu_{k+3}^*\mu_{k+1}^* \mu_{k}^*\mu_{k+3}^*     \sigma_{k}^* \sigma_{k+2}^*\sigma_{k+4}^* \colon \calA_{\ii'} \isoto \calA_\ii.
$$
In Appendix~\ref{sec: G2 equi}, we record all the $32$ sequences of mutations of length $10$ that produce the equivalence with index permutation $\sigma_{k}^* \sigma_{k+2}^*\sigma_{k+4}^*$. Recall that $\bfdeg$ does not depend on
the choice of sequences of mutations \cite{Pla13,GHKK} .
\end{remark}

\begin{lemma}\label{Lem:g-vector 6move}
 If $\sfz \in \calA_{\ii'}$ is a quantum cluster monomial with $\bfdeg(\sfz)=\bsg'=(g'_u)_{u \in \N}$, then the quantum cluster monomial
$\zeta_k^*\sfz \in \calA_\ii$ has $\bfdeg(\zeta_k^*\sfz)=\bsg=(g_u)_{u \in \N}$ as follows:
\begin{equation}\label{eq:6g}
\begin{aligned}
g_j = \begin{cases}
g'_{(k+1)_{\ii'}^-} + [g'_{k+1}]_+  + [B]_+  - [-G]_+ - [-I]_+ & \text{ if } j =k_\ii^-, \\
g'_{k_{\ii'}^-} +[D ]_+ -\sfc_{j,i} [F]_+  + 2[H]_+  -\sfc_{j,i} [I]_+  & \text{ if } j =(k+1)_\ii^-, \\
-I    & \text{ if } j =k, \\
-H +\sfc_{j,i} [-I]_+  & \text{ if } j =k+1, \\
-G  +[I]_+   & \text{ if } j =k+2, \\
 -E +\sfc_{j,i} [-G]_+   +[H]_+      & \text{ if } j =k+3, \\
g'_{k+5} -[-B ]_+ +[G]_+     & \text{ if } j =k+4, \\
g'_{k+4}  -[-A]_+  -\sfc_{j,i}[B]_+  +[E]_+     & \text{ if } j =k+5, \\
g'_j   & \text{ otherwise},
\end{cases}
\end{aligned}
\end{equation}
where
\begin{equation} \label{eq: 6moves form}
\begin{aligned}
A & \seteq g'_{k+2} -\sfc_{j,i} [g'_{k+1}]_+, \qquad
B  \seteq g'_{k+3} -[-g'_{k+1}]_+ , \qquad
C  \seteq  -g'_{k+1}  -\sfc_{i,j} [A]_+  -[-B]_+, \\
D  &\seteq g'_{k}  +\sfc_{j,i}[-g'_{k+1}]_+  -2 [-A]_+  + \sfc_{j,i}[-C]_+, \ \qquad\qquad
E  \seteq  -A -\sfc_{j,i}[C]_+   +[D ]_+, \\
 F &\seteq  -C +\sfc_{i,j} [-D ]_+,    \qquad\qquad\qquad \qquad\qquad\quad \ 
G  \seteq  -B   -[-C]_+  +\sfc_{i,j}[-E]_+ +[F]_+, \\
 H  &\seteq  -D   +[E]_+    +\sfc_{j,i}[-F]_+,  \quad \qquad \qquad \qquad \qquad\qquad
\  I  \seteq  -F +[G]_+   +\sfc_{i,j}[-H]_+.
\end{aligned}
\end{equation}
\end{lemma}

\begin{proof}
By applying ~\eqref{eq: g-vector mutation} repeatedly, 
we can obtain the assertion. 
\end{proof}

\begin{lemma} \label{lem: cone to cone 6g}
The map $\bsg' \mapsto \bsg$ given by~\eqref{eq:6g} sends the cone $C_{\ii'} \subset \Z^{\oplus \N}$ into the cone $C_{\ii} \subset \Z^{\oplus \N}$.    
\end{lemma}

\begin{proof}
Let us use similar arguments to those in Lemma~\ref{lem: cone to cone 4g} and notations in~\eqref{eq: 6moves form}. 
As in the $4$–move case, we check positivity by verifying that all partial sums $\Sigma_u=\sum_{v\le u} g_v$ remain nonnegative.
Because the $6$–move only modifies entries between $k$ and $k+5$,
we only need to check those indices.
 
\medskip
 \noindent
(a) If $u > k+5$, we have $\Sigma_u = \Sigma'_u \ge 0$.

\medskip
\noindent
(b) If $u=k+5$, we have  
\begin{align*}
\Sigma_{k+5} & = \Sigma_{k+4}' -[-A]_+ -\sfc_{j,i}[B]_+ + [E]_+.
\end{align*}
In the above equation, the only term that can contribute a negative amount is $-[-A]_+$. Therefore, to check the non-negativity of $\Sigma_{k+5}$ it is enough to distinguish whether $A$ is nonnegative or negative.

Note that if $A < 0$, then $E > 0$ and $A+E \ge 0$ since 
\begin{align*}
A +E  &=   A+ (-A -\sfc_{j,i}[C]_+   +[D ]_+) = -\sfc_{j,i}[C]_+   +[D ]_+ \ge 0.  
\end{align*}
 Thus $\Sigma_{k+5} \ge 0$.
Assume that $A \ge 0$. In this case, we have
 \begin{align*}
\Sigma_{k+5} & = \Sigma_{k+4}'  -\sfc_{j,i}[B]_+   + [E]_+,
\end{align*}
which proves the assertion for $k+5$. 

\medskip
\noindent
(c)  If $u=k+4$, we have 
\begin{align*}
\Sigma_{k+4} & = \Sigma_{k+5}' -[-B]_+ +[G]_+.
\end{align*}
As in the case (b), it suffices to consider when $B <0$. If $g'_{k+1} <0$, then 
$\Sigma'_{k+5}+B =\Sigma'_{k+1} \ge 0$,  and  if $g'_{k+1} \ge 0$, then 
$\Sigma'_{k+5}+B =\Sigma'_{k+3} \ge 0$. Thus the assertion follows.

\medskip
\noindent
(d)  If $u=k+3$, we have 
\begin{equation}  \label{eq: k+3 start} 
\begin{aligned}
\Sigma_{k+3} =  & \Sigma_{k+4}'  -[-A]_+ -\sfc_{j,i}[B]_+ + [E]_+  -E +\sfc_{j,i} [-G]_+   +[H]_+ .
\end{aligned}
\end{equation}
 Applying $[E]_+-E = [-E]_+$, ~\eqref{eq: k+3 start}  becomes 
 \begin{equation} \label{eq: k+3 2nd} 
\begin{aligned}  
\Sigma_{k+3} =  & \Sigma_{k+4}'  -[-A]_+ -\sfc_{j,i}[B]_+ + [-E]_+ +\sfc_{j,i}[-G]_+   +[H]_+ .
\end{aligned} 
\end{equation}
Now we divide into cases according to the signs of 
$A$ and $E$.

\smallskip
\noindent
\underline{($A<0,E> 0$)}: In this case, we have 
$$ [-E]_+=0 , \quad -[-A]_+=A=g'_{k+2} -\sfc_{j,i} [g'_{k+1}]_+$$ and hence ~\eqref{eq: k+3 2nd}  becomes
\begin{align*}
&\underline{\Sigma_{k+4}'  +g'_{k+2}} -\sfc_{j,i} [g'_{k+1}]_+ -\sfc_{j,i}[B]_+   +\sfc_{j,i}[-G]_+   +[H]_+ \\
& = \underline{\Sigma_{k+2}'}   -\sfc_{j,i} [g'_{k+1}]_+ -\sfc_{j,i}[B]_+  +\sfc_{j,i} [-G]_+   +[H]_+. 
\end{align*}
Since the only term that can contribute a negative amount is $[-G]_+$, it is enough to consider when 
$$0>G =-B   -[-C]_+  +\sfc_{i,j}[-E]_+ +[F]_+.$$ 
Note that it suffices to show that
\begin{align} \label{eq: k+3 critical part}
    \Sigma_{k+2}'   -\sfc_{j,i} [g'_{k+1}]_+ -\sfc_{j,i}[B]_+  +\sfc_{j,i}  [-G]_+ \ge 0.
\end{align}
Replacing $G$ with the above formula, ~\eqref{eq: k+3 2nd} turns into 
 \begin{align*}
 &\Sigma_{k+2}'   -\sfc_{j,i} [g'_{k+1}]_+ -\sfc_{j,i}[B]_+    -\sfc_{j,i} (-B   -[-C]_+  +\sfc_{i,j}[-E]_+ +[F]_+)  \\
 & = \Sigma_{k+2}'   -\sfc_{j,i} [g'_{k+1}]_+ -\sfc_{j,i}([B]_+-B)   +\sfc_{j,i} [-C]_+ -3[-E]_+ - \sfc_{j,i}[F]_+ \\
 & \overset{*}{=} \Sigma_{k+2}'   -\sfc_{j,i} [g'_{k+1}]_+ -\sfc_{j,i}[-B]_+   +\sfc_{j,i} [-C]_+   - \sfc_{j,i}[F]_+.
  \end{align*}
Here $\overset{*}{=}$ follows from  
$$\text { (1) $E \ge 0$ by the assumption} \quad  \text{ and } \quad \text{ (2) $[B]_+-B=[-B]_+$.}$$ Finally assume 
$$0>C=-g'_{k+1}     -[-B]_+.$$ Then we have
\begin{align*}
 &\Sigma_{k+2}'   -\sfc_{j,i} [g'_{k+1}]_+ -\sfc_{j,i}[-B]_+ - \sfc_{j,i}[F]_+    +\sfc_{j,i}g'_{k+1}    +\sfc_{j,i} [-B]_+ \\
  &\hspace{10ex} \overset{\dagger}{=} \Sigma_{k+2}'  -\sfc_{j,i} [-g'_{k+1}]_+  - \sfc_{j,i}[F]_+.
\end{align*}
Here $\overset{\dagger}{=}$ follows from 
 (3) $[g'_{k+1}]_+-g'_{k+1}=[-g'_{k+1}]_+$ and
(4) $[-B]_+-[-B]_+=0$. 
Then our assertion for this case follows. 

\smallskip 
\noindent
\underline{($A\ge 0,E<  0$)}: In this case, we have 
$$ [-E]_+=-E , \quad -[-A]_+=0$$
and hence ~\eqref{eq: k+3 2nd}  becomes
\begin{align*}
& \Sigma_{k+4}'  -\sfc_{j,i}[B]_+  -E  +\sfc_{j,i}[-G]_+   +[H]_+ 
\end{align*}

Since $-E>0$, it suffices to consider when 
$$0>G =-B   -[-C]_+  -\sfc_{i,j}E +[F]_+,$$
as in the previous case. 
Replacing $G$ with the above formula, ~\eqref{eq: k+3 2nd} turns into 
\begin{align*}
& \Sigma_{k+4}'  -\sfc_{j,i}[B]_+ -E       +[H]_+      +\sfc_{j,i}B + \sfc_{j,i}[-C]_+ + 3E - \sfc_{j,i}[F]_+ \\
   & \overset{\ddagger}{=} \Sigma_{k+4}'  -\sfc_{j,i}[-B]_+ +2E +[H]_+- \sfc_{j,i}[F]_+  + \sfc_{j,i}[-C]_+.
\end{align*}
Here $ \overset{\ddagger}{=} $ follows from (1) $[B]_+-B=[-B]_+$. Finally assume 
$$0>C=-g'_{k+1}  -\sfc_{i,j} A  -[-B]_+.$$
Then we have
\begin{align*}
&  \Sigma_{k+4}'  -\sfc_{j,i}[-B]_+ +2E +[H]_+- \sfc_{j,i}[F]_+ 
  + \sfc_{j,i}g'_{k+1} +3A + \sfc_{j,i}[-B]_+ \\
& \overset{\star}{=} \Sigma_{k+4}'- 2A   +2[D ]_+ +[H]_+  
- \sfc_{j,i}[F]_+ + \sfc_{j,i}g'_{k+1} +3A  \\
& \overset{\sharp}{=}   \Sigma_{k+4}'+g'_{k+2} -\sfc_{j,i} [g'_{k+1}]_+   
 +2[D ]_+ +[H]_+ - \sfc_{j,i}[F]_+ + \sfc_{j,i}g'_{k+1}  \\
& \overset{\diamond}{=}   \Sigma_{k+2}'  -\sfc_{j,i} [-g'_{k+1}]_+  
 +2[D ]_+ +[H]_+ - \sfc_{j,i}[F]_+. 
\end{align*}
Here 
\ben
\item[(2)] $\overset{\star}{=}$ follows from the replacement of $E$ with $-A   +[D ]_+$,
\item[(3)] $\overset{\sharp}{=}$ follows from the replacement of $A$ with $g'_{k+2} -\sfc_{j,i} [g'_{k+1}]_+$,
\item[(4)] $\overset{\diamond}{=}$ follows from $\Sigma_{k+4}'+g'_{k+2} =\Sigma_{k+2}'$ and $[g'_{k+1}]_+- g'_{k+1}=[-g'_{k+1}]_+$. 
\ee

\smallskip 
\noindent
\underline{($A\ge 0,E \ge 0$)}: In this case,
we have 
$$ [-E]_+=0 , \quad -[-A]_+=0$$
and hence ~\eqref{eq: k+3 2nd}  becomes
\begin{align*}
& \Sigma_{k+4}'   -\sfc_{j,i}[B]_+  +\sfc_{j,i} [-G]_+   +[H]_+.
\end{align*}
Similarly, it suffices to consider when 
$$0>G =-B   -[-C]_+   +[F]_+.$$ 
Replacing $G$ with the above formula, ~\eqref{eq: k+3 2nd} turns into 
\begin{align*}
&\Sigma_{k+4}'   -\sfc_{j,i}[B]_+ +[H]_+  +\sfc_{j,i}B + \sfc_{j,i}[-C]_+ - \sfc_{j,i}[F]_+ \\
& = \Sigma_{k+4}'   -\sfc_{j,i}[-B]_+ +[H]_+ + \sfc_{j,i} [-C]_+ - \sfc_{j,i}[F]_+.
\end{align*}
Finally assume 
$$0>C=-g'_{k+1}  -\sfc_{i,j} A  -[-B]_+.$$
Then we have
\begin{align*}
&\Sigma_{k+4}'   -\sfc_{j,i}[-B]_+ +[H]_+ - \sfc_{j,i}[F]_+  + \sfc_{j,i} g'_{k+1} +3A +\sfc_{j,i} [-B]_+ \allowdisplaybreaks\\
& \overset{\wedge}{=}  \Sigma_{k+4}'   +[H]_+ - \sfc_{j,i}[F]_+  + \sfc_{j,i} g'_{k+1} +2A +g'_{k+2} -\sfc_{j,i} [g'_{k+1}]_+ \allowdisplaybreaks\\
& \overset{\oplus}{=} \Sigma_{k+2}' +[H]_+ - \sfc_{j,i}[F]_+ -\sfc_{j,i} [-g'_{k+1}]_+ + 2A \ge 0,
\end{align*}
which completes a proof of the assertion for $k+3$. 
Here
\bnum
\item $\overset{\wedge}{=}$ follows from $[-B]_+-[-B]_+=0$ and $A= g'_{k+2} -\sfc_{j,i} [g'_{k+1}]_+$,
\item $\overset{\oplus}{=} $ follows from $\Sigma_{k+2}'=\Sigma_{k+4}'+g'_{k+2} $ and $[g'_{k+1}]_+- g'_{k+1}= [-g'_{k+1}]_+$. 
\ee

\medskip
\noindent
(e)  If $u=k+2$, we have 
\begin{align*}
\Sigma_{k+2} & = \Sigma_{k+5}'  -[-B]_+ +[G]_+ -G  +[I]_+ \\
& = \Sigma_{k+5}'  -[-B]_+ +[-G]_++[I]_+. 
\end{align*}
Thus the assertion for $k+2$ is a direct consequence of the assertion for $k+4$.

 \medskip
\noindent
(f)  If $u=k+1$, we have 
 \begin{equation} \label{eq: k+1 1st} 
\begin{aligned}  
\Sigma_{k+1} & =   \Sigma_{k+4}'  -[-A]_+ -\sfc_{j,i}[B]_+ + [-E]_+ \\
 & \hspace{10ex} +\sfc_{j,i} [-G]_+   +\underline{[H]_+ -H} +\sfc_{j,i} [-I]_+\\
 & =\Sigma_{k+4}'  -[-A]_+ -\sfc_{j,i}[B]_+ + [-E]_+ \\
 & \hspace{10ex} +\sfc_{j,i} [-G]_+   +\underline{[-H]_+}+\sfc_{j,i} [-I]_+.
\end{aligned} 
\end{equation}
If $I \ge 0$, the assertion for $k+1$ follows from the one for $k+3$. Thus it suffices to assume that $I <0$. 
As in the case (c), we divide into cases according to the signs of 
$A$ and $E$.

\smallskip
\noindent
\underline{($A<0,E\ge 0, I <0$)}: By replacing $I$ with $-F +[G]_+   +\sfc_{i,j}[-H]_+$,~\eqref{eq: k+1 1st} 
becomes 
\begin{align*}
&\Sigma_{k+4}'  +A -\sfc_{j,i}[B]_+ +\sfc_{j,i} [-G]_+   +[-H]_+     
+\sfc_{j,i}F  -\sfc_{j,i}[G]_+ -3[-H]_+  \\
& \overset{\square}{=}   \Sigma_{k+4}'  +g'_{k+2} -\sfc_{j,i} [g'_{k+1}]_+ -\sfc_{j,i}[B]_+   +\sfc_{j,i} (F-G) 
-2[-H]_+ \\
& = \Sigma_{k+2}' -\sfc_{j,i} [g'_{k+1}]_+ -\sfc_{j,i}[B]_+   +\sfc_{j,i} (F-G) -2[-H]_+.
\end{align*}
Here $\overset{\square}{=}$ follows from $-[-G]_++[G]_+=G$.
Note that the assumption $I <0$ is equivalent to 
\begin{align}\label{eq: I-inequalty}
[G]_+ +\sfc_{i,j}[-H]_+ < F.    
\end{align}

Using~\eqref{eq: I-inequalty}, we have 
\begin{align*}
\sfc_{j,i} (F-G) -2[-H]_+  & > \sfc_{j,i} ([G]_+    +\sfc_{i,j}[-H]_+-G) -2[-H]_+ \\
&   =\sfc_{j,i} [-G]_+ +[-H]_+ = \sfc_{j,i} [-G]_+ + [-H]_+. 
\end{align*}
Then 
\begin{equation} \label{eq: k+1 case1 final}
\begin{aligned}
&\Sigma_{k+1}  >  
 \left\{ \Sigma_{k+2}' -\sfc_{j,i} [g'_{k+1}]_+ -\sfc_{j,i}[B]_+  +\sfc_{j,i} [-G]_+ \right\}  + [-H]_+.    
\end{aligned}
\end{equation}
Then~\eqref{eq: k+3 critical part} implies that the $\{ \}$-part in~\eqref{eq: k+1 case1 final} is non-negative and hence our assertion for this case follows. 

\smallskip
\noindent
\underline{($A\ge 0,E< 0, I <0$)}: By replacing $I$ with $-F +[G]_+    +\sfc_{i,j}[-H]_+$,~\eqref{eq: k+1 1st} 
becomes 
\begin{align}
&\Sigma_{k+4}'    -\sfc_{j,i}[B]_+ - E +\sfc_{j,i}[-G]_+   +[-H]_+   +\sfc_{j,i}F -\sfc_{j,i}[G]_+  -3[-H]_+ \nonumber \\
 & =  \Sigma_{k+4}'-\sfc_{j,i}[B]_+ -E +\sfc_{j,i}(F-G) -2[-H]_+  \label{eq: k+1 second critical}
\end{align}
as in the former case. Here we remark that $-E>0$. 

Using~\eqref{eq: I-inequalty} again and \eqref{eq: k+1 second critical}, we have an inequality 
\begin{align}
&\Sigma_{k+4}'-\sfc_{j,i}[B]_+  +\sfc_{j,i}(F-G) -2[-H]_+-E  \nonumber \\
& \hspace{3ex} > \Sigma_{k+4}'-\sfc_{j,i}[B]_+-E + \sfc_{j,i}( [G]_+ +\sfc_{i,j}[-H]_+ -G) -2[-H]_+\nonumber \\
& \hspace{3ex}  = \Sigma_{k+4}'-\sfc_{j,i}[B]_+  + \sfc_{j,i}[-G]_+ +[-H]_+-E. \label{eq: final ineq} 
\end{align}
Now let us assume that 
$$
0> G =  -B   -[-C]_+  +\sfc_{i,j}[-E]_+ +[F]_+.
$$
Then~\eqref{eq: final ineq} becomes
\begin{align*}
& \Sigma_{k+4}'-\sfc_{j,i}[B]_+ +[-H]_+ -E  
  + \sfc_{j,i}B + \sfc_{j,i}[-C]_+ + 3E  - \sfc_{j,i}[F]_+ \allowdisplaybreaks\\
& \overset{\triangle}{=} \Sigma_{k+4}'-\sfc_{j,i}[-B]_+ +[-H]_+ +2E  
 + \sfc_{j,i}[-C]_+  - \sfc_{j,i}[F]_+ \allowdisplaybreaks\\
& \overset{\circ}{=}  \Sigma_{k+4}'-\sfc_{j,i}[-B]_+ +[-H]_+-2A  -2\sfc_{j,i}[C]_+ 
+[D ]_+ + \sfc_{j,i}[-C]_+  - \sfc_{j,i}[F]_+\allowdisplaybreaks \\
& \overset{\bullet}{=}  \Sigma_{k+4}'-\sfc_{j,i}[-B]_+ +[-H]_+-2A  -\sfc_{j,i}[C]_+    
+[D ]_+ - \sfc_{j,i}C  - \sfc_{j,i}[F]_+ \allowdisplaybreaks\\ 
& \overset{*}{=}  \Sigma_{k+4}'-\sfc_{j,i}[-B]_+ +[-H]_+-2A  -\sfc_{j,i}[C]_+ \\
& \hspace{35ex}    - \sfc_{j,i}[F]_+ +[D ]_+ + \sfc_{j,i}g'_{k+1} +3A +\sfc_{j,i} [-B]_+ \allowdisplaybreaks \\ 
& \overset{\nabla}{=}  \Sigma_{k+4}' +[-H]_+ -\sfc_{j,i}[C]_+  - \sfc_{j,i}[F]_+  +[D ]_+   +A + \sfc_{j,i}g'_{k+1}  \allowdisplaybreaks\\
& \overset{\odot}{=} \Sigma_{k+4}' +[-H]_+ -\sfc_{j,i}[C]_+  - \sfc_{j,i}[F]_+  +[D ]_+ + g'_{k+2} -\sfc_{j,i} [g'_{k+1}]_+ + \sfc_{j,i}g'_{k+1}  \allowdisplaybreaks \\
& \overset{\vee}{=} \Sigma_{k+2}' +[-H]_+ -\sfc_{j,i}[C]_+  - \sfc_{j,i}[F]_+  +[D ]_+ -\sfc_{j,i} [-g'_{k+1}]_+. 
\end{align*}
Here 
\bnum
\item $\overset{\triangle}{=}$ follows from $[B]_+-B=[-B]_+$,
\item $\overset{\circ}{=}$ follows from $E  =  -A -\sfc_{j,i}[C]_+   +[D ]_+$,
\item $\overset{\bullet}{=}$ follows from $[C]_+ - [-C]_+=C$,
\item $\overset{*}{=}$ follows from $C = -g'_{k+1}  -\sfc_{i,j} [A]_+  -[-B]_+$,
\item $\overset{\nabla}{=}$ follows from $[-B]_+ - [-B]_+=0$,
\item $\overset{\odot}{=}$ follows from $A = g'_{k+2} -\sfc_{j,i} [g'_{k+1}]_+$,
\item $\overset{\vee}{=}$ follows from $\Sigma_{k+2}'=\Sigma_{k+4}'+g'_{k+2}$ and $[g'_{k+1}]_+-g'_{k+1}=[-g'_{k+1}]_+$. 
\ee

\smallskip
\noindent
\underline{($A\ge 0,E \ge 0, I <0$)}: By replacing $I$ with $-F +[G]_+  +\sfc_{i,j}[-H]_+$,~\eqref{eq: k+1 1st} 
becomes 
\begin{align}
&\Sigma_{k+4}'  -\sfc_{j,i}[B]_+ +\sfc_{j,i} [-G]_+   +[-H]_+  \nonumber  \\
 & \hspace{25ex}  +\sfc_{j,i}F  -\sfc_{j,i}[G]_+ -3[-H]_+ \nonumber \\
&= \Sigma_{k+4}'  -\sfc_{j,i}[B]_+-2[-H]_+ +\sfc_{j,i} (F-G)   \label{eq: k+1 3rd}   
\end{align}
as in the previous cases. Using~\eqref{eq: I-inequalty} and~\eqref{eq: k+1 3rd}, we have 
\begin{align}
  & \Sigma_{k+4}'  -\sfc_{j,i}[B]_+-2[-H]_+ +\sfc_{j,i} (F-G)  \nonumber \\  
  & \hspace{5ex} > \Sigma_{k+4}'  -\sfc_{j,i}[B]_+ +\sfc_{j,i}[-G]_++[-H]_+ . \label{eq: k+1 3rd critical}
\end{align}
Now let us assume that 
$$
0> G =  -B   -[-C]_+   +[F]_+.
$$
Then~\eqref{eq: k+1 3rd critical} becomes
\begin{align*}
& \Sigma_{k+4}'  -\sfc_{j,i}[B]_+ +[-H]_+ 
+\sfc_{j,i}B + \sfc_{j,i} [-C]_+  - \sfc_{j,i} [F]_+ \\
& = \Sigma_{k+4}'  -\sfc_{j,i}[-B]_+ +[-H]_+ + \sfc_{j,i} [-C]_+  - \sfc_{j,i} [F]_+
\end{align*}
Finally, it suffices to assume that 
$$0> C = -g'_{k+1}  -\sfc_{i,j}A  -[-B]_+.$$
Then we have
\begin{align*}
& \Sigma_{k+4}'  -\sfc_{j,i}[-B]_+ +[-H]_+- \sfc_{j,i} [F]_+
 + \sfc_{j,i} g'_{k+1}   +3A +\sfc_{j,i}[-B]_+ \\
& = \Sigma_{k+4}'  +[-H]_+- \sfc_{j,i} [F]_+  +2A  + \sfc_{j,i} g'_{k+1}   +   g'_{k+2} -\sfc_{j,i} [g'_{k+1}]_+ \\
& = \Sigma_{k+2}'  +[-H]_+- \sfc_{j,i} [F]_+  +2A -\sfc_{j,i} [-g'_{k+1}]_+.
\end{align*}
Since $A\ge 0$, it proves the assertion for $k+1$.

\medskip
\noindent
(g)  If $u=k$, we have 
\begin{align}
\Sigma_{k} & = \Sigma_{k+5}'  -[-B]_+ +[G]_+ -G  +[I]_+ -I  \nonumber \\
& = \Sigma_{k+5}'  -[-B]_+ +[-G]_++[-I]_+. \label{eq: k}
\end{align}
Thus the assertion for $k$ is also a direct consequence of the assertion for $k+4$. 

\medskip
\noindent
(h)  For $u<k$,  the assertion follows from  Lemma~\ref{lem: i-degree 6g} below.
\end{proof}

\begin{lemma} \label{lem: i-degree 6g}  
Assume that $k^- \in \N$. Then we have
$$
\sfp_{\ii'}(\bsg'; i) = \sfp_{\ii}(\bsg; i) \quad \text{ for all } i\in I, 
$$
where $\bsg'\mapsto \bsg$ under the map given by~\eqref{eq:6g}.  
\end{lemma}

\begin{proof} 
It suffices to show that
\bnum
\item \label{it: first case} $g_{k+5}+g_{k+3}+g_{k+1}+g_{(k+1)^-_\ii} = g'_{k+4}+g'_{k+2}+g'_{k}+g'_{k^-_{\ii'}}$,
\item \label{it: second case} $g_{k+4}+g_{k+2}+g_{k}+g_{k^-_\ii} = g'_{k+5}+g'_{k+3}+g'_{k+1}+g'_{(k+1)^-_{\ii}}$.
\ee

\noindent 
\underline{\eqref{it: first case}}: By~\eqref{eq: k+1 1st} and applying   the techniques we have used,   
\begin{align*}
\Sigma_{(k+1)^-_\ii} & = ~\eqref{eq: k+1 1st} + g_{(k+1)^-_\ii}\\
& = \Sigma_{k+4}' -[-A]_+ -\sfc_{j,i}[B]_+ + [-E]_+  +\sfc_{j,i} [-G]_+   + [-H]_+  +\sfc_{j,i}[-I]_+ + (g'_{k_{\ii'}^-}  +[D ]_+ \allowdisplaybreaks \\
  &\hspace{10ex} \left. -\sfc_{j,i} [F]_+ + 2[H]_+  -\sfc_{j,i} [I]_+ \right) \quad (\because~\eqref{eq:6g}) \allowdisplaybreaks\\
  & = \Sigma_{k+4}' + g'_{k_{\ii'}^-}  -[-A]_+ -\sfc_{j,i}[B]_+ + [-E]_+  +\sfc_{j,i} [-G]_+  + [-H]_+\allowdisplaybreaks\\
 & \hspace{5ex}   +[D ]_+ -\sfc_{j,i} [F]_+ + 2[H]_+  -\sfc_{j,i}I \qquad 
 (\because [I]_+ -[-I]_t =I )\allowdisplaybreaks\\
  & =  \Sigma_{k+4}' + g'_{k_{\ii'}^-} -[-A]_+ -\sfc_{j,i}[B]_+ + [-E]_+  +\sfc_{j,i} [-G]_+   + [-H]_+ +[D ]_+ -\sfc_{j,i} [F]_+  
   \allowdisplaybreaks\\
 & \hspace{5ex} + 2[H]_+  + \sfc_{j,i}F - \sfc_{j,i}[G]_+ -3[-H]_+ 
 \quad (\because I =  -F+[G]_+ +\sfc_{i,j}[-H]_+) \allowdisplaybreaks\\
  & = \Sigma_{k+4}' + g'_{k_{\ii'}^-}  -[-A]_+ -\sfc_{j,i}[B]_+ +[-E]_+     +[D ]_+ -\sfc_{j,i} [F]_+ + 2H  + \sfc_{j,i}F - \sfc_{j,i}G \allowdisplaybreaks\\
  & \qquad (\because [x]_+ - [-x]_+ = x) \allowdisplaybreaks\\
 & = \Sigma_{k+4}' + g'_{k_{\ii'}^-}  -[-A]_+ -\sfc_{j,i}[B]_+ +[-E]_+  +[D ]_+  -\sfc_{j,i} [F]_+  - 2D   +2[E]_+   
 \allowdisplaybreaks\\
 & \hspace{5ex} +2\sfc_{j,i} [-F]_+ - \sfc_{j,i}C +3[-D ]_+  + \sfc_{j,i}B + \sfc_{j,i}[-C]_+ -3[-E]_+- \sfc_{j,i}[F]_+\allowdisplaybreaks\\
 &\quad (\because \eqref{eq: 6moves form} \text{ for } H,F \text{ and } G)\allowdisplaybreaks\\
  & = \Sigma_{k+4}' + g'_{k_{\ii'}^-}  -[-A]_+ -\sfc_{j,i}[B]_+   +[D ]_+ +3[-D ]_+ +\sfc_{j,i}[-C]_+   - 2D   +2E  
 \allowdisplaybreaks\\
 & \hspace{5ex}    -2\sfc_{j,i} F - \sfc_{j,i}C     + \sfc_{j,i}B  \qquad (\because [x]_+ - [-x]_+ = x)  \allowdisplaybreaks\\
   & = \Sigma_{k+4}' + g'_{k_{\ii'}^-}  -[-A]_+ -\sfc_{j,i}[B]_+   +[D ]_+ +3[-D ]_+ + \sfc_{j,i}[-C]_+    - 2g'_{k}  \allowdisplaybreaks \\
 & \hspace{5ex}  -2\sfc_{j,i}[-g'_{k+1}]_+  +4[-A]_+  -2 \sfc_{j,i}[-C]_+ -2A -2\sfc_{j,i}[C]_+   +2[D ]_+ +2\sfc_{j,i} C \\
 & \hspace{5ex}   -6[-D ]_+ + \sfc_{j,i}g'_{k+1}  + 3[A]_+ + \sfc_{j,i}[-B]_++ \sfc_{j,i}g'_{k+3} -\sfc_{j,i}[-g'_{k+1}]_+  \allowdisplaybreaks\\
   &\quad (\because \eqref{eq: 6moves form} \text{ for } B,C,D,E \text{ and } F)\allowdisplaybreaks\\
  & = \Sigma_{k+4}' + g'_{k_{\ii'}^-}     - 2g'_{k}  -3\sfc_{j,i}[-g'_{k+1}]_+  +3[-A]_+  -3  \sfc_{j,i}[-C]_++ \sfc_{j,i}g'_{k+1}  + 3[A]_+  \allowdisplaybreaks\\
 & \hspace{5ex}       + \sfc_{j,i}g'_{k+3} -2A  -\sfc_{j,i}B   +3D  
 \qquad (\because [x]_+ - [-x]_+ = x) \allowdisplaybreaks \\
 & = \Sigma_{k+4}' + g'_{k_{\ii'}^-}     - 2g'_{k}  -3\sfc_{j,i}[-g'_{k+1}]_+  +3[-A]_+  -3  \sfc_{j,i}[-C]_+
+ \sfc_{j,i}g'_{k+1}  + 3[A]_+    \allowdisplaybreaks \\
 & \hspace{5ex}    + \sfc_{j,i}g'_{k+3} -2A  -\sfc_{j,i}B +3 g'_{k}  +3\sfc_{j,i}[-g'_{k+1}]_+  -6[-A]_+  + 3\sfc_{j,i}[-C]_+  \allowdisplaybreaks \\
 &\quad (\because \eqref{eq: 6moves form} \text{ for } D)\allowdisplaybreaks\\
  & = \Sigma_{k+4}' + g'_{k_{\ii'}^-}    +g'_{k}     +A   + \sfc_{j,i}g'_{k+1}     + \sfc_{j,i}g'_{k+3}   -\sfc_{j,i}B  
  \qquad (\because [x]_+ - [-x]_+ = x)  \allowdisplaybreaks\\
   & = \Sigma_{k+4}' + g'_{k_{\ii'}^-}    +g'_{k}     +g'_{k+2} -\sfc_{j,i} [g'_{k+1}]_+  + \sfc_{j,i}g'_{k+1}    \allowdisplaybreaks \\
 & \hspace{5ex}     + \sfc_{j,i}g'_{k+3}   -\sfc_{j,i}g'_{k+3}  +\sfc_{j,i} [-g'_{k+1}]_+ \quad (\because \eqref{eq: 6moves form} \text{ for } A \text{ and } B) \allowdisplaybreaks\\
    & = \Sigma_{k+4}' + g'_{k_{\ii'}^-}    +g'_{k}     +g'_{k+2} 
    (\because [x]_+ - [-x]_+ = x) 
    \allowdisplaybreaks\\
    & = \Sigma'_{k_{\ii'}^-}   
\end{align*}
\noindent
In the above computation, we repeatedly substitute the relations in~\eqref{eq: 6moves form}
and use $[x]_+-[-x]_+=x$ to cancel all auxiliary terms.

\medskip
\noindent
\underline{\eqref{it: second case}}: By~\eqref{eq: k}, we have
\begin{align*}
\Sigma_{k^-_\ii} & = ~\eqref{eq: k} + g_{k^-_\ii} \\
& =  \Sigma_{k+5}'  -[-B]_+ +[-G]_++[-I]_+ + (g'_{(k+1)_{\ii'}^-} + [g'_{k+1}]_+  + [B]_+ -[-G]_+ - [-I]_+ ) \allowdisplaybreaks\\
& = \Sigma_{k+5}'  + B   + g'_{(k+1)_{\ii'}^-} + [g'_{k+1}]_+    
\qquad (\because [x]_+-[-x]_+ = x)\allowdisplaybreaks\\
& = \Sigma_{k+5}'  + g'_{k+3} -[-g'_{k+1}]_+    + [g'_{k+1}]_+  + g'_{(k+1)_{\ii'}^-} \qquad (\because \eqref{eq: 6moves form} \text{ for } B) \allowdisplaybreaks\\
& = \Sigma_{k+5}'  + g'_{k+3} + g'_{k+1}  + g'_{(k+1)_{\ii'}^-}
\qquad (\because [x]_+-[-x]_+ = x) \allowdisplaybreaks\\
& = \Sigma'_{(k+1)^{-}_{\ii'}}. \qedhere
\end{align*}
 
\end{proof}

The following lemma can be obtained as in other moves and we omit the proof.

\begin{lemma} \label{lem: 6g Pgi}
Let  $\bsg,\bsg' \in \Z^{\oplus\N}$ be elements related by~\eqref{eq:4g}. If $\bsg'_s=0$ for $s \in [k,k+3]$, we have
$$
\sfp_\ii(\bsg;i) = \sfp_{\ii'}(\bsg';i) \qquad \text{ for any } i \in I. 
$$
\end{lemma}

\begin{example} Take $\ii'=(2,1,2,1,2,1,\ldots)$ and $\ii=(1,2,1,2,1,2,\ldots)$. Using the above formulae, the map  $\uptheta_1^{(1)*} : \bsg'\mapsto \bsg$ in~\eqref{eq:6g} that is induced by the isomorphism $\zeta_1^{(1)*} : \calA_{\ii'} \isoto \calA_{\ii}$ can be descried for special $\bsg$'s as follows:
\begin{equation} \label{eq: g-mut G2 part1}
\begin{aligned}
& \uptheta_1^{(1)*}(\bse_1) =  \bse_6- \bse_4 ,     && \uptheta_1^{(1)*}( \bse_2 ) =   (\bse_6- \bse_4) +3\bse_1, \\
& \uptheta_1^{(1)*}(\bse_3-\bse_1) = 2 (\bse_6- \bse_4) +3\bse_1,  && \uptheta_1^{(1)*}( \bse_4-\bse_2) =  (\bse_6- \bse_4) +2\bse_1 , \\
& \uptheta_1^{(1)*}(\bse_5-\bse_3) =  (\bse_6- \bse_4) +3\bse_1 ,  && \uptheta_1^{(1)*}(\bse_6-\bse_4) =   \bse_1.
\end{aligned}
\end{equation}
Similarly, the map $\uptheta_1^{(2)*}$ induced by the isomorphism $\zeta_1^{(2)*} : \calA_{\ii} \isoto \calA_{\ii'}$  can be descried as follows:
\begin{equation} \label{eq: g-mut G2 part2}
\begin{aligned}
& \uptheta_1^{(2)*}(  \bse_1 ) =  (\bse_6- \bse_4),  && \uptheta_1^{(2)*}(\bse_2) =  3(\bse_6- \bse_4) +\bse_1 , \\
& \uptheta_1^{(2)*}( \bse_3-\bse_1 ) =  2(\bse_6- \bse_4) +\bse_1 ,  && \uptheta_1^{(2)*}(\bse_4-\bse_2) =  3(\bse_6- \bse_4) +2\bse_1 , \\
& \uptheta_1^{(2)*}( \bse_5-\bse_3) =  (\bse_6- \bse_4) +\bse_1 ,  && \uptheta_1^{(2)*}(\bse_6-\bse_4) =  \bse_1.
\end{aligned}
\end{equation}
\end{example}

\subsection{Forward shifts}
In this subsection, we consider a forward shift of $\ii\in I^\N$
introduced in \cite{FHOO2} when $\g$ is of simply-laced type. However, in this subsection,
we do not impose any restriction on the type of $\g$. 

We say that $\ii'=(i'_k)_{k \in \N}$ is obtained from $\ii=(i_k)_{k \in \N}$ via a forward shift if  \[ i'_k = i_{k+1} \quad \text{for all $k \in \N$}. \]
In this case, we write $\ii'=\partial_+\ii$.
Throughout this subsection, we assume that $\ii,\ii' \in I^{(\infty)}$ are related by $\ii'=\partial_+\ii$ and $i_1\seteq i$.  Define  
\bna
\item a sequence $(x_u)_{u \in \N}$ such that
$x_1=1$ and $x_{u}=(x_{u-1})_\ii^+$ for $u \ge 2$,
\item a permutation $\sigma_+ \in \frakS_\N$ by
$\sigma_+(k) = \bc
k_\ii^+ -1 & \text{ if } i_k =i, \\
k-1 & \text{ if } i_k \ne i. 
\ec $
\ee
 
\begin{remark}
Even though \cite{FHOO2} considered only simply-laced types, the statements and proofs in \cite[\S 2.3]{FHOO2} can be extended to non simply-laced types. Thus we only give
the statements and leave their proofs to the reader. 
\end{remark}

\begin{proposition} \hfill  \label{prop: fshift}
\bna
\item The limit
$ \lim_{n \to \infty} \mu_{x_n}\cdots \mu_{x_2}\mu_{x_1} \tB^{\ii} $
yields a well-defined matrix $\mu_+\tB^{\ii}$ and we have $\sigma_+ \mu_+ \tB^{\ii} = \tB^{\ii'}$.
\item The limit
$\lim_{n \to \infty} \mu_{x_n}\cdots \mu_{x_2}\mu_{x_1} \La^{\ii}$ 
yields a well-defined matrix $\mu_+\La^{\ii}$ and we have $\sigma_+ \mu_+ \La^{\ii} = \La^{\ii'}$.
\item We have a homomorphism of $\bbA$-algebras
\[ \partial_+^* (= \mu_+^* \sigma_+^*) \colon \calA_{\ii'} \to \calA_\ii \qquad \text{given by $Z_u \mapsto \mu_+(Z_{\sigma_+^{-1}(u)})$ for all $u \in \N$},\]
which sends each cluster monomial in $\calA_{\ii'}$ to a cluster monomial in $\calA_\ii$.
\item \label{it: d} If $\sfz \in \calA_{\ii'}$ is a quantum cluster monomial with $\bfdeg(\sfz)=\bsg'=(g'_u)_{u \in \N}$, then the quantum cluster monomial
$\partial_+^*\sfz \in \calA_{\ii}$ has $\bfdeg(\partial_+^*\sfz)=\bsg=(g_u)_{u \in \N}$ as follows:
\begin{equation} \label{eq:dg}
g_u = \begin{cases}
-\sum_{v \in \N; i'_v = i_1} g'_v & \text{if $u=1$}, \\
g'_{u-1} & \text{if $u > 1$}. \\
\end{cases}
\end{equation}
In other words, the assignment $\bsg' \mapsto \bsg$ is given by the $\N_0$-linear map $C_{\ii'} \to C_\ii$ which sends $\bse_u - \bse_{u^-_{\ii'}} \in C_{\ii'}$ to $
\bse_{u+1} - \bse_{(u+1)^-_{\ii}} \in C_{\ii}$ for all $u \in \N$.
\ee 
\end{proposition}

\section{Bases and braid moves}  \label{sec: basis and moves}

In this section, we first recall the quantum unipotent coordinate algebra
$\calA_q(\n)$ of $\g$ and its bases.
Then we apply the result in the previous section to investigate the relationship among braid moves and the bases.

\subsection{Quantum unipotent coordinate algebra}
Let $\calU_q(\g)$ be the quantum group of $\g$ over $\Q(q^{1/2})$. It is generated by $e_i,f_i$ $(i \in I)$
and $q^h$, where $h$ belongs to the dual weight lattice $\wl^\vee = \Hom_\Z(\wl,\Z)$.  We denote by 
$\calU_q^+(\g)$ the subalgebra of $\calU_q(\g)$ generated by $e_i$ $(i\in I)$ and
$\calU_\bbA^+(\g)$ the $\bbA$-subalgebra of $\calU_q^+(\g)$ generated by
$e_i/[n]_i!$, where 
$$q_i \seteq q^{\sfd_i}, \ \ \ [k]_i \seteq \dfrac{q_i^k-q_i^{-k}}{q_i-q_i^{-1}}, \  \ \text{ and } \ \ [k]_i! \seteq \prod_{s=1}^k [s_i] \ \ 
\text{ for } k \in \N.$$

Set
$$
\calA_q(\n) = \soplus_{\be \in \rl^-} \calA_q(\n)_\be \quad
\text{ where }  \calA_q(\n)_\be \seteq \Hom_{\Q(q^{1/2})}(\calU^+_q(\g)_{-\be},
\Q(q^{1/2})),
$$
and $\calU^+_q(\g)_{-\be}$ denotes the $(-\be)$-root space of 
$\calU^+_q(\g)$.
For an element $a \in \calA_q(\n)_\be$, we set $\wt(a)\seteq \be$. 
It is known that  $\calA_q(\n)$ has an algebra structure and
is called the \emph{quantum unipotent coordinate algebra} of $\g$. 
We denote by $\calA_\bbA(\n)$ the $\bbA$-submodule of $\calA_q(\n)$
generated by $\uppsi \in \calA_q(\n)$ such that $\uppsi(\calU_\bbA^+(\g)) \subset \bbA$. Then $\calA_\bbA(\n)$ also has an $\bbA$-algebra structure. 

For each $\la \in \wl^+ \seteq \sum_{i \in I} \N \varpi_i$ and Weyl group elements
$w,w' \in \weyl$, we can define a specific homogeneous element $D(w\la,w'\la)$
of $\calA_\bbA(\n)$, called a \emph{quantum unipotent minor} (see for example \cite[Section 9]{KKKO18}). Note that $D(w\la,w'\la)$ vanishes if $w\la -w'\la \not \in \rl^-$. We set
$$
\tD(w\la,w'\la) \seteq q^{-(w'\la-w\la,w'\la-w\la)/4-(w'\la-w\la,\rho)/2}D(w\la,w'\la)
$$
where $\rho \seteq \sum_{i \in I}\varpi_i \in \wl^+$, and call it a \emph{normalized quantum unipotent minor}. 

Let us fix a reduced sequence $\ii$ of $w_\circ$. For 
$1 \le k \le \ell$, we set
$$
\tF_\ii(k) \seteq \tD(w^\ii_k \varpi_{i_k}, w^\ii_{k^-} \varpi_{i_k}).
$$
Note that $\wt(\tF_\ii(k)) =- \be^\ii_k \in \Phi^-$ as an element in $\calA_\bbA(\n)$. For $\bsc=(c_1,\ldots,c_\ell) \in \N_0^{\ell}$,
we set
$$
\nu(\bsc) = -\dfrac{1}{2}\sum_{1\le k,l \le \ell} c_kc_l (\be^\ii_k,\be^\ii_l)
+ \sum_{s=1}^\ell c_s^2 \in \Z
$$
and define
$$
\tF_\ii(\bsc) \seteq q^{\nu(\bsc)/2} \prod^\to_{1 \le k \le \ell} \tF_\ii(k)^{c_k}.
$$

\begin{proposition} $($see, for instance, \cite[Chapters 40,41]{LusztigBook}$)$
The set $\{ \tF_\ii(\bsc) \ | \ \bsc \in \N_0^\ell \}$ forms an $\bbA$-basis of $\calA_\bbA(\n)$.    
\end{proposition}

There exist an  anti-involution $\overline{(\cdot)}$   on $\calA_\bbA(\n)$
defined by 
$$
q^{\pm 1/2} \longmapsto q^{\mp 1/2} \quad \text{ and } \quad \tF_\ii(k) \longmapsto
\tF_\ii(k)
$$
for all $1 \le k \le \ell$, which is called the \emph{twisted bar-involution}.

\subsection{Comparison with braid moves}
Lusztig \cite{L90,L91} and Kashiwara \cite{K91} have constructed a distinguished $\bbA$-basis $\bfB^\up$, which is called the \emph{dual-canonical/upper-global  basis}.
In this paper, we consider the normalized one $\tbfB^\up$ of $\bfB^\up$ (see \cite{Kimura12} for more details). 
Note that the normalized quantum unipotent minor $\tD(w\la,w'\la)$ 
belongs to $\tbfB^\up$ whenever it is non-zero.

\begin{theorem}[{\cite[Theorem 4.29]{Kimura12}}]
For a reduced sequence $\ii$ of $w_\circ$, we have
$$
\tbfB^\up = \{ \tG_{\ii}(\bsc) \ | \  \bsc \in \N_0^{\ell} \}
$$
where $\tG_{\ii}(\bsc)$ is the unique element of $\calA_\bbA(\n)$ satisfying the following properties for each $\bsc \in \N_0^{\Phi^+}$:
\bna
\item $\overline{\tG_{\ii}(\bsc)}=\tG_{\ii}(\bsc)$,
\item $\tG_{\ii}(\bsc) - \tF_{\ii}(\bsc) \in \sum_{\bsc' <_{\ii} \bsc} q\Z[q]\tF_{\ii}(\bsc')$,
\ee
where $<_{\ii}$ is a certain bi-lexicographical order determined by
$\ii$. In particular, $\tG_\ii(\bse_k)=\tD(w^\ii_k\varpi_{i_k},w^\ii_{k^-}\varpi_{i_k})$. 
\end{theorem}

\begin{proposition}[{\cite[Chapter 42]{LusztigBook}, \cite[Theorem 3.1]{BZ97}  }]     \label{Prop:Lusztig}
Let $\ii$ and $\ii'$ be two reduced sequences for $w_\circ$, and $\bsc = (c_1, \ldots, c_{\ell}), \bsc'=(c'_1, \ldots, c'_{\ell}) \in \N_0^{\ell}$.
\ben
\item Assume $\ii'=\ga_k\ii$ for some $1 \le k \le \ell-1$. Then we have
$$
\tG_\ii(\bsc) = \tG_{\ii'}(\bsc) \quad \text{ if and only if } \quad
\bc
(c_{k},c_{k+1}) = (c'_{k+1},c'_{k}), \\
c_u = c'_u \text{ if } u \not\in \{k,k+1\}.
\ec
$$
\item Assume $\ii'=\beta_k\ii$ for some $1 \le k \le \ell-2$. Then we have 
$$
\tG_\ii(\bsc) = \tG_{\ii'}(\bsc) \quad \text{ if and only if } \quad
\bc
c'_k = c_{k+1}+c_{k+2}-\min(c_k,c_{k+2}), \\
c'_{k+1} = \min(c_k,c_{k+2}), \\
c'_{k+2} = c_{k+1}+c_{k}-\min(c_k,c_{k+2}), \\
c_u = c'_u \text{ if } u \not\in \{k,k+1,k+2\}.
\ec
$$
\item \label{it: eta} Assume   $\ii'=\eta_k\ii$ for some $1 \le k \le \ell-3$. Note that  $(i_k,i_{k+1},i_{k+2},i_{k+3})$ and
$(i'_k,i'_{k+1},i'_{k+2},i'_{k+3})$ can be regarded as reduced sequences of $w_\circ$ of type $B_2$, 
and hence we have parametrizations of these sequences in terms of $\Phi^+_{B_2}$ $($see {\rm Example~\ref{ex: B2 longest}}$)$.
Then we have $($see also~\cite[Lemma 2.15]{SST2018}$)$  
\[ \tG_\ii(\bsc) = \tG_{\ii'}(\bsc') \quad \Longleftrightarrow \quad
\begin{cases}
c'_{\al_2} = c_{\al_1+\al_2} + 2c_{\al_1+2\al_2} + c_{\al_2} -\pi_2, \\
c'_{\al_1+2\al_2} = \pi_2-\pi_1, \\
c'_{\al_1+\al_2} = 2\pi_1-\pi_2, \\
c'_{\al_1} = c_{\al_1}+c_{\al_1+\al_2}+c_{\al_1+2\al_2}-\pi_1, \\
c'_u = c_u \text{ if $u \not \in \{ k, k+1, k+2,k+3\}$},
\end{cases}
\]
where
\begin{align*}
&\pi_1 = \min \{ c_{\al_1}+ c_{\al_1+\al_2}, c_{\al_1}+c_{\al_2},c_{\al_1+2\al_2}+c_{\al_2}  \}, \\
&\pi_2 = \min \{ 2c_{\al_1}+ c_{\al_1+\al_2}, 2c_{\al_1}+c_{\al_2},2c_{\al_1+2\al_2}+c_{\al_2}  \}.
\end{align*}
\item \label{it: zeta} For $\ii'=(2,1,2,1,2,1)$ and $\ii=(1,2,1,2,1,2)$ of $G_2$-type, we have
\begin{equation} \label{eq: zeta mutation part1}
\begin{aligned}
& \tG_{\ii'}(\bse_1) =  \tG_{\ii}(\bse_6),  && \tG_{\ii'}(\bse_2) =  \tG_{\ii}(\bse_1  +\bse_6), \\
& \tG_{\ii'}(\bse_3) =  \tG_{\ii}(3\bse_1 +2\bse_6),  && \tG_{\ii'}(\bse_4) =  \tG_{\ii}(2 \bse_1  +3\bse_6), \\
& \tG_{\ii'}(\bse_5) =  \tG_{\ii}( 3\bse_1+\bse_6),  && \tG_{\ii'}(\bse_6) =  \tG_{\ii}(\bse_1),
\end{aligned}
\end{equation}
and
\begin{equation} \label{eq: zeta mutation part2}
\begin{aligned}
& \tG_{\ii}(\bse_1) =  \tG_{\ii'}(\bse_6),                                    && \tG_{\ii}(\bse_2) =  \tG_{\ii'}( \bse_1+3\bse_6), \\
& \tG_{\ii}(\bse_3) =  \tG_{\ii'}( \bse_1 +2\bse_6),  && \tG_{\ii}(\bse_4) =  \tG_{\ii'}( 2\bse_1 +3\bse_6), \\
& \tG_{\ii}(\bse_5) =  \tG_{\ii'}( \bse_6 +\bse_1),  && \tG_{\ii}(\bse_6) =  \tG_{\ii'}(\bse_1).
\end{aligned}
\end{equation}
\ee
\end{proposition}

\begin{theorem}[\cite{GLS13,GY17,KKKO18,KK19,GY20,Qin21,KKOP23}] \label{Thm:CAN}  
For each reduced sequence $\ii=(i_1,\ldots,i_\ell)$ of $w_\circ$, there is an isomorphism of $\bbA$-algebras
\[ \varphi_\ii \colon \calA_\ii \isoto \calA_\bbA(\n) \]
which satisfies the following properties:
\ben
\item For any $1 \le v < u \le \ell$ with $i_u = i_v$, there exists a cluster variable which corresponds to $\tD(w^\ii_u \varpi_{i_u}, w^\ii_v\varpi_{i_v})$ under $\varphi_\ii$.  
\item The initial seed of $\calA_{\bbA}(\n)$ corresponding to that of $\calA_\ii$ is given by
\begin{align}
\big(\{ \tD(w_{k}^\ii\varpi_{i_k},\varpi_{i_k}  ) \}_{1 \le k \le \ell}, \La^\ii,\tB^\ii\big). 
\end{align}
\item Every cluster monomial of $\calA_\ii$ corresponds to an element of the basis $\tbfB^\up$ under $\varphi_\ii$.
\item For each $\bsc = (c_u) \in \N_0^{\ell}$, the element $\varphi_{\ii}^{-1}\tG_\ii(\bsc)$ is pointed and we have
\begin{equation} \label{eq:degG}
\bfdeg \big(\varphi_{\ii}^{-1}\tG_\ii(\bsc) \big) = \sum_{1 \le u \le \ell}c_u(\bse_u - \bse_{u^-})  \in \Z^{\ell},
\end{equation}
where we understand $\bse_0 = 0$.
In particular, the map $\bfdeg \circ \varphi_\ii^{-1} \colon \tbfB^\up \to \Z^{\ell}$ is injective.
\ee
\end{theorem}

The following proposition is proved for $2$-moves and $3$-moves in \cite[Proposition 3.3]{FHOO2}, and we extend it to $4$-moves and $6$-moves:

\begin{proposition}
Let $\ii, \ii'$ be two reduced sequences for $w_\circ$.
Assume that we have $\ii' = \tau \ii$ for some $\tau \in 
\{ \ga_1, \ldots ,\ga_{\ell-1} \} \sqcup
\{ \beta_1, \ldots ,\beta_{\ell-2} \} \sqcup
\{ \eta_1, \ldots ,\eta_{\ell-3} \} \cup \{ \zeta_1 \}$.
Then the following diagram commutes:
\[ \xymatrix@R=4ex@C=5ex{
\calA_{\ii'} \ar[dr]_{\varphi_{\ii'}}\ar[rr]_-{\tau^*}   && \calA_{\ii} \ar[dl]^{\varphi_\ii}  \\
& \calA_\bbA(\n)  
}
\]
\end{proposition}

\begin{proof}
Since $\calA_\bbA(\n)$ 
is generated by $\{\tG_{\ii'}(\bse_u) \ | \  1 \le u \le \ell \}$,
it is enough to show that 
\begin{equation} \label{eq:tau}
\varphi_\ii \tau^* \varphi_{\ii'}^{-1}\tG_{\ii'}(\bse_u) = \tG_{\ii'}(\bse_u)
\end{equation}
for all $u \in [1,\ell]$.
By Theorem~\ref{Thm:CAN} and the fact that $\tau^*$ preserves cluster monomials,
we have
$$\varphi_\ii \tau^* \varphi_{\ii'}^{-1}\tG_{\ii'}(\bse_u) \in \tbfB^\up \quad 
\text{ and } \quad 
\bfdeg \big(\varphi_{\ii'}^{-1}\tG_{\ii'}(\bse_u)\big) = (\bse_u - \bse_{u^-_{\ii'}}).
$$

\noindent
\underline{($\eta$)}: Let us consider $\tau=\eta_k$ for some $ 1 \le u \le \ell-2$. 
Recall $A$ and $B$ in~\eqref{eq:4g AB}. 

\smallskip

(i) $u=k+3$. In this case
 $$ \bfdeg \big(\varphi_{\ii'}^{-1}\tG_{\ii'}(\bse_{k+3})\big) = (\bse_{k+3} - \bse_{k+1}), \quad A =  \sfc_{j,i} <0
 \quad \text{ and } B = -1.$$
Then we have
$g_k = -1 +  \sfc_{j,i}\sfc_{i,j}  = 1$, $g_{ k^-_{\ii}} =  -1$, $g_{k+3} =  g_{k+2} =      g_{k+1} =   g_{(k+1)^-_{\ii}}  = 0$ by~\eqref{eq:4g}. 
Hence
$$    \bfdeg \big(\eta_k^*\varphi_{\ii'}^{-1}\tG_{\ii'}(\bse_{k+3}) \big)=   \bse_{k}-\bse_{k^-_{\ii}} = \bfdeg  \big(\varphi_{\ii}^{-1}\tG_{\ii}( \bse_k )\big).$$

(ii) $u=k+2$. In this case
 $$ \bfdeg \big(\varphi_{\ii'}^{-1}\tG_{\ii'}(\bse_{k+2})\big) =  \bse_{k+2} - \bse_{k}, \quad A = -1 \quad \text{ and } \quad B=\sfc_{i,j}<0.$$
Then we have
$  g_k = - g_{ k^-_{\ii}} = -\sfc_{i,j}$, $g_{k+3} = -g_{k+1} = 1,$  $g_{k+2} = g_{(k+1)^-_{\ii}}  =  0$. 
Hence
\begin{align*}
 \bfdeg \big(\eta_k^*\varphi_{\ii'}^{-1}\tG_{\ii'}(\bse_{k+2}) \big)& = (\bse_{k+3}-\bse_{k+1}) -\sfc_{i,j}(\bse_{k}-\bse_{k^-_{\ii}})  = \bfdeg  \big(\varphi_{\ii}^{-1}\tG_{\ii}(-\sfc_{i,j}\bse_k+\bse_{k+3}  )\big).
\end{align*}

(iii) $u=k+1$. In this case
 $$ \bfdeg \big(\varphi_{\ii'}^{-1}\tG_{\ii'}(\bse_{k+1})\big) = (\bse_{k+1} - \bse_{(k+1)^-_{\ii'}}) = (\bse_{k+1} - \bse_{k^-_\ii}),
 \quad A =0 \quad \text{ and } \quad B=-1.$$
Then we have
$g_k = -g_{ k^-_{\ii}} = 1$, $g_{k+3}= -g_{k+1}  = -\sfc_{j,i}$, $g_{k+2} =    g_{(k+1)^-_{\ii}}  = 0$.  
Hence
\begin{align*}
  \bfdeg \big(\eta_k^*\varphi_{\ii'}^{-1}\tG_{\ii'}(\bse_{k+1})\big) & = (\bse_{k}-\bse_{k^-_{\ii}}) -\sfc_{j,i}(\bse_{k+3}-\bse_{k+1})  
 = \bfdeg  \big(\varphi_{\ii}^{-1}\tG_{\ii}(\bse_{k} -\sfc_{j,i}\bse_{k+3}  )\big).
\end{align*}

(iv) $u=k$. In this case
$$ \bfdeg \big(\varphi_{\ii'}^{-1}\tG_{\ii'}(\bse_k) \big)= (\bse_k - \bse_{k^-_{\ii'}}) = (\bse_k - \bse_{(k+1)^-_{\ii}}),$$
Then we have 
$g_{k+3} =  1$,   $g_{k+1}  = -1$, $g_{k+2}= g_k =   g_{k^-_{\ii'}}= g_{(k+1)^-_{\ii}} = g_{k^-_{\ii}}= 0$, 
by~\eqref{eq:4g} and~\eqref{eq:4g AB}.
Hence
$$    \bfdeg \big(\eta_k^*\varphi_{\ii'}^{-1}\tG_{\ii'}(\bse_k) \big)=  (\bse_{k+3}-\bse_{k+1})= \big(\varphi_{\ii}^{-1}\tG_{\ii}( \bse_{k+3} )\big).$$

(v) $u \not\in \{k,k+1,k+2,k+3\}$. In this case, one can easily see that
\begin{align*}
 \bfdeg \big(\eta_k^*\varphi_{\ii'}^{-1}\tG_{\ii'}(\bse_{u})\big) & = (\bse_{u}-\bse_{u^-_{\ii}})   = \bfdeg  \varphi_{\ii}^{-1}\tG_{\ii}( \bse_{u} ).
\end{align*}

In conclusion, we have  
\begin{align*}
\bfdeg \big(\eta^*_k \varphi_{\ii'}^{-1}\tG_{\ii'}(\bse_u) \big)= \begin{cases}
\bse_{k}-\bse_{k^-_{\ii}} & \text{ if } u =k+3, \\
 (\bse_{k+3}-\bse_{k+1}) -\sfc_{i,j}(\bse_{k}-\bse_{k^-_{\ii}}) & \text{ if } u =k+2, \\
 (\bse_{k}-\bse_{k^-_{\ii}}) -\sfc_{j,i}(\bse_{k+3}-\bse_{k+1})& \text{ if } u =k+1, \\
 \bse_{k+3} -  \bse_{k+1} & \text{ if } u =k, \\
\bse_{u} - \bse_{u^-_\ii} & \text{ otherwise},
\end{cases}
\end{align*}
and get  $ \varphi_{\ii} \eta^*_k \varphi_{\ii'}^{-1}\tG_{\ii'}(\bse_u) = \tG_{\ii}(\bfc_u)$, where 
\begin{align*}
\bfc_u = \begin{cases}
\bse_{k}  & \text{ if } u =k+3, \\
  \bse_{k+3} -\sfc_{i,j} \bse_{k}  & \text{ if } u =k+2, \\
 \bse_{k}  -\sfc_{j,i} \bse_{k+3} & \text{ if } u =k+1, \\
 \bse_{k+3}   & \text{ if } u =k, \\
\bse_{u} & \text{ otherwise}.
\end{cases}
\end{align*}
Then our assertion for $\eta_k$ follows from Proposition~\ref{Prop:Lusztig}~\eqref{it: eta}.

\noindent
\underline{($\zeta$)}: The assertion for $\zeta$-case follows from~\eqref{eq: g-mut G2 part1},~\eqref{eq: g-mut G2 part2},~\eqref{eq: zeta mutation part1},~\eqref{eq: zeta mutation part2} and Proposition~\ref{Prop:Lusztig}~\eqref{it: zeta}.
\end{proof}

Note that any two reduced sequences of $w_\circ$ are connected by braid moves; i.e.,
there exists a finite sequence  $\bstau=(\tau_1,\ldots,\tau_l)$ in $
\{ \ga_1, \ldots ,\ga_{\ell-1} \} \sqcup
\{ \beta_1, \ldots ,\beta_{\ell-2} \} \sqcup
\{ \eta_1, \ldots ,\eta_{\ell-3} \} \cup \{ \zeta_1 \}$ such that
$\ii'=\tau_1\cdots \tau_l \ii$. Then we have the isomorphism
$$
\bstau^*= \tau_l^* \circ  \cdots \circ \tau_1^*: \calA_{\ii'} \simeq \calA_\ii.
$$

\begin{corollary} \label{Cor:bftau}
With the above notation, the following diagram commutes:
\[  \xymatrix@R=4ex@C=5ex{
\calA_{\ii'} \ar[dr]_{\varphi_{\ii'}}\ar[rr]_-{\bstau^*}   && \calA_{\ii} \ar[dl]^{\varphi_\ii}  \\
& \calA_\bbA(\n)  
}
\]
Moreover, the isomorphism $\bstau^* = \varphi_{\ii'}^{-1} \circ \varphi_{\ii}$ depends only on the pair $(\ii, \ii')$, and not on the sequence $\bstau$ satisfying $\ii' = \tau_1 \cdots \tau_l \ii$.
\end{corollary}

\section{Homomorphisms and substitution formulas} \label{sec: substitution}
In this section, we first review the quantum (virtual) Grothendieck ring $\frakK_q(\g)$ 
associated with $\g$, its  bases and subrings. 
For $\g$ of simply-laced type, the ring $\frakK_q(\g)$ is introduced in \cite{Nak04,VV02,Her04} 
while it is introduced in \cite{KO22,JLO1,JLO2} for $\g$ of non-simply-laced type. 
Then we study homomorphisms on subrings of $\frakK_q(\g)$ and the skew field of fraction $\bbF(\calX_q)$, which are closely related to braid moves and the theory of quantum cluster algebras. 

\subsection{Quantum torus $\calX_q(\g)$ and the ring $\frakK_q(\g)$}  

Let $t$ be an indeterminate. For $i,j \in I$, we set 
$$
\usfc_{i,j}(t) \seteq \delta_{i,j}(t+t^{-1})+ \delta(i \ne j)\sfc_{i,j}
$$
and call the matrix $\usfC(t) \seteq (\usfc_{i,j}(t))$ the \emph{$t$-quantized Cartan matrix}. Then
$$
\usfB(t) \seteq \usfC(t)\sfD^{-1} \text{ is a symmetric matrix in }{\rm GL}_n(\Z[t^{\pm 1}]).
$$

Regarding $\usfB(t)$ as a $\Q(t)$-valued matrix, it is invertible and we denote the inverse of $\usfB(t)$ by 
$\tusfB(t) = (\tusfB_{i,j}(t))_{i,j\in I}$. Write
$$
\tusfB_{i,j}(t) = \sum_{u \in \Z} \tusfb_{i,j}(u)t^u \in  \Z(\!( t )\!)
$$
for the Laurent expansion of $\tusfB_{i,j}(t)$ at $t=0$. It is proved in
\cite{HL15,KO22} that,
for any $i,j \in I$ and $u \in \Z$, we have
\begin{align} \label{eq: non-zero}
\tusfb_{i,j}(u) =0 \quad \text{ if {\rm (i)} $u \le d(i,j)$} \quad \text{ or } \quad \text{{\rm (ii)} $d(i,j) \equiv u$ $\pmod 2$.}
\end{align}

We fix a parity function $\ep: I \to \{ 0,1\}$ such that $\ep_i \ne \ep_j$ for $i \sim j$, based on~\eqref{eq: non-zero}, and 
define 
\begin{align} \label{eq: hDynkin0}
\hDynkin_0 \seteq \{ (i,p) \in I \times \Z \ | \  p \equiv \ep_i \pmod 2 \} \subsetneq I \times \Z.
\end{align}

\begin{definition}[\cite{Nak04,Her04,VV03,KO22}] \label{def:quantum torus Xqg}
We define a quantum torus $(\calX_q(\g),*)$ to be the $\bbA$-algebra generated by $\{ \tX_{i,p}^{\pm} \ | \ (i,p) \in \hDynkin_0\}$ subject to the following relations:
$$   \tX_{i,p} * \tX_{i,p}^{-1} = \tX_{i,p}^{-1} * \tX_{i,p} =1 \quad \text{ and } \quad \tX_{i,p} *\tX_{j,s} = q^{\ucalN(i,p;j,s)}\tX_{j,s}*\tX_{i,p},$$
where
\begin{align*}
\ucalN(i,p;j,s) = \tusfb_{i,j}(p-s-1)-\tusfb_{i,j}(s-p-1)-\tusfb_{i,j}(p-s+1)+\tusfb_{i,j}(s-p+1).
\end{align*}
\end{definition}

Note that there exists a $\Z$-algebra anti-involution on $\calX_q(\g)$ given by 
\begin{align*}
q^{1/2} \longmapsto q^{-1/2} \quad \text{ and } \quad 
\tX_{i,p} \longmapsto q^{\sfd_i} \tX_{i,p}. 
\end{align*}
We will write $\calX_q$ instead of $\calX_q(\g)$ if there is no danger of confusion.
We denote by $\calX$ the commutative Laurent polynomial ring obtained from $\calX_q$ by specializing $q$ at $1$. We write the map $\evq: \calX_q \to \calX$
and set $X_{i,p} \seteq \evq(q^{a/2}\tX_{i,p})$ for any $a \in \Z$. We call a product $\tm$ of $\tX_{i,p}^{\pm 1}$ and $q^{\pm 1/2}$ in $\calX_q$ a \emph{monomial}; i.e.,
it is of the form
$$
\tm = q^{a/2} \st^{\to}_{(i,p) \in \hDynkin_0} \tX_{i,p}^{u_{i,p}(\tm)} \quad \text{for } a,  u_{i,p}(\tm) \in \Z. 
$$
We also call a monomial  $\tm\in \calX_q$ \emph{commutative} if $\tm = \overline{\tm}$, and  \emph{dominant} (resp.\emph{anti-dominant}) if it is a product of $\tX_{i,p}$ (resp. $\tX_{i,p}^{-1}$) and $q^{\pm 1/2}$. 

For a monomial $m \in \calX$, we can associate the commutative monomial
$\um \in \calX_q$ such that $\evq(\um)=m$. Thus we write $\um$ for commutative monomials in $\calX_q$. Let $\calM$ be the set of all monomials in $\calX$,
and $\calM_+$ be the set of all dominant monomials in $\calX$.

\smallskip

For monomials $\tm,\tm'$ in $\calX_q$, define
$$
\ucalN(\tm,\tm') \seteq \sum_{(i,p),(j,s) \in \hDynkin_0} u_{i,p}(\tm)u_{j,s}(\tm')\ucalN(i,p;j,s). 
$$
For each $(i,p) \in I \times \Z$ such that $(i,p\pm1) \in \hDynkin_0$, define the commutative monomial 
$$\tscrB_{i,p}\seteq \uscrB_{i,p},$$ 
where 
$\scrB_{i,p} = X_{i,p-1}X_{i,p+1} \prod_{j \sim i} X_{j,p}^{\sfc_{j,i}} \in \calX$. 

\begin{definition}[\cite{Her04,KO22,JLO1}] \label{def: qvgr} \hfill 
\bna
\item For each $i \in I$, denote by $\frakK_{i,q}(\g)$ the $\bbA$-subalgebra of $\calX_q$ generated by 
$$
\tX_{i,p} * (1+q^{-\sfd_i}\tscrB_{i,p+1}^{-1}) \ \text{ and } \ \tX_{j,s}^{\pm 1} \quad \text{ for } j \in I \setminus \{ i \} \text{ and } (i,p),(j,s) \in \hDynkin_0.
$$
\item Define
$$
\frakK_{q}(\g) \seteq \bigcap_{i \in I} \frakK_{i,q}(\g)
$$
and call it the \emph{quantum virtual Grothendieck ring} associated with $\usfC(t)$. If there is no confusion, we write $\frakK_q$ instead of $\frakK_q(\g)$.
\ee
\end{definition}

\subsection{Bases}
For $m,m' \in \calM$, there exists a partial order $\leN$, called \emph{Nakajima order}, defined by
$$
m \leN m' \text{ if and only if } m^{-1}m' \text{ is a product of elements in } \{ \scrB_{i,p} \ | \ (i,p\pm1)\in \hDynkin_0 \}. 
$$
For monomials $\tm,\tm' \in \calX_q$, we write $\tm \leN \tm'$ if $\evq(\tm) \leN \evq(\tm')$.

\begin{theorem} [\cite{Her04,JLO1}, cf.~\cite{FM01}] \label{thm: F_q}
For each dominant monomial $m \in \calM_+$, there exists a unique element $F_q(m)$ of $\frakK_q$ such that
$$\text{$\um$ is the unique dominant monomial appearing in $F_q(m)$
and $\overline{F_q(m)}=F_q(m)$.}$$ Furthermore,
\bna
\item a monomial appearing in $F_q(m)-\um$ is strictly less than $\um$ with respect to $\leN$,
\item $\sfF_q \seteq \{ F_q(m) \ | \  m \in \calM_+ \}$ forms an $\bbA$-basis of $\frakK_q$,
\item $\frakK_q$ is generated by the set $\{ F_q(X_{i,p}) \ | \ (i,p) \in \hDynkin_0 \}$ as an $\bbA$-algebra.
\ee
\end{theorem}

For $m \in \calM_+$, we set
$$ E_q(m) \seteq q^b \left( \st^{\to}_{p \in \Z} \left(   \st_{i; \in I, (i,p) \in \hDynkin_0}  F_q(X_{i,p})^{u_{i,p}(m)} \right) \right),$$
where $b$ is an element in $\tfrac{1}{2}\Z$ such that $\um$ appears in $E_q(m)$ with coefficient $1$.
Then the set $\sfE_q\seteq \{ E_q(m) \ | \  m \in \calM_+ \}$ also forms an $\bbA$-basis of $\frakK_q(\g)$, called the \emph{standard basis}, and satisfies
\begin{align} \label{eq: uni E F}
E_q(m) = F_q(m) + \sum_{m'\lN m} C_{m,m'}F_q(m') \quad \text{for some $C_{m,m'} \in \bbA$.}
\end{align}

\begin{theorem}[\cite{Nak04,Her04,JLO1}]\label{thm: L_q}
For each $m \in \calM_+$, there exists a unique element $L_q(m)$ in $\frakK_q(\g)$ such that
\begin{eqnarray} &&
\parbox{90ex}{
\bnum
\item $\overline{L_q(m)} = L_q(m)$ and
\item \label{it: Lq(m)} $L_q(m)  = E_q(m) + \displaystyle \sum_{m' \lN m}  Q_{m,m'}(q) E_q(m') $ for some $ Q_{m,m'}(q) \in q\Z[q]$.
\ee
}\label{eq: L_q}
\end{eqnarray}
Furthermore,
\bna
\item a monomial appearing in $L_q(m)-\um$ is strictly less than $\um$ with respect to $\leN$,
\item $\sfL_q \seteq \{ L_q(m) \ | \  m \in \calM_+  \}$ forms an $\bbA$-basis of $\frakK_q$ called the \emph{canonical basis} of $\frakK_q(\g)$,
\item $L_q(X_{i,p}) = F_q(X_{i,p})$ for all $(i,p) \in \hDynkin_0$.
\ee
\end{theorem}

Let $\frakD_q^{\pm 1}$ be the $\bbA$-algebra automorphism of $\calX_q$
given by $\tX_{i,p} \mapsto \tX_{i^*,p\pm\sfh}$, where $\sfh$ is the Coxeter number of $\g$. It is well-known that it 
induces a $\Z$-algebra automorphism $\frakD$ on $\calX$ satisfying $\evq \circ \frakD^{\pm1}_q = \frakD^{\pm1} \circ \evq$ and 
preserves the subalgebra $\frakK_q$ of $\calX_q$. Moreover, we have
\begin{align}
\frakD_q^{\pm 1}(F_q(m)) = F_q(\frakD^{\pm1}m) \quad \text{ and }
\quad 
\frakD_q^{\pm 1}(L_q(m)) = L_q(\frakD^{\pm1}m) \quad \text{ for any } m \in \calM_+.
\end{align}

\begin{remark} \label{rem: Tr}
Note that, for $r \in 2\Z$, there exists an automorphism $\sfT_r$ of $\calX_q$ sending 
$\tX_{i,p}$ to $\tX_{i,p+r}$ satisfying the following properties:
\bna
\item The restriction to $\frakK_q$ gives an automorphism of $\frakK_q$.
\item For each $m \in \calM_+$, we have $\sfT_r(F_q(m))=F_q(\mathtt{T}_r(m))$, $\sfT_r(E_q(m))=E_q(\mathtt{T}_r(m))$ and $\sfT_r(L_q(m))=L_q(\mathtt{T}_r(m))$.
Here $\mathtt{T}_r$ denotes the induced automorphism on $\calX$. 
\ee
\end{remark}

\subsection{Kirillov--Reshetikhin and fundamental polynomials} \label{subsec: KR}
For a pair $(i,p), (i,s) \in \hDynkin_0$ with $p \le s$, we set
\begin{equation} \label{eq:dominant monomials for KR module}
m^{(i)}[p,s]  \seteq  \prod_{  (i,u) \in \hDynkin_0, \  p \le u \le s}  X_{i,s} \in \calM_+. 
\end{equation}
We call $m^{(i)}[p,s]$ the \emph{Kirillov--Reshetikhin$($KR$)$  monomial} and 
$F_q(m^{(i)}[p,s]) \in \frakK_q(\g)$ the \emph{KR-polynomial} for the pair $(i,p)$ and $(i,s)$, respectively.
In particular, we call $F_q(X_{i,p})=F_q(m^{(i)}[p,p]) $ the \emph{fundamental polynomial} for each $(i,p) \in \hDynkin_0$. 

\begin{theorem}[\cite{Nak03,Her06,FHOO}]\label{thm: KR-poly}
 For each KR-monomial $m^{(i)}[p,s]$, the KR-polynomial $F_q(m^{(i)}[p,s])$ satisfies the following properties:
\bna
\item $F_q(m^{(i)}[p,s])$ contains the unique dominant monomial $\underline{m^{(i)}[p,s]}$.
\item $F_q(m^{(i)}[p,s])$ contains the unique anti-dominant monomial $\underline{m^{(i^*)}[p+\sfh,s+\sfh]}$.  
\item \label{it: range of KR} 
Each $\calX_q$-monomial of $F_q(\underline{m^{(i)}[p,s]}) - \underline{m^{(i)}[p,s]}-\underline{m^{(i^*)}[p+\sfh,s+\sfh]}$ is a product of $\tX_{j,u}^{\pm1}$ with $p \le u \le s+\sfh$, having at least one of its factors from $p < u < s+\sfh$. In particular, each monomial of $F_q(\underline{\tX_{i,p}}) - \underline{\tX_{i,p}}-\underline{\tX_{i^*,p+\sfh}}$ is a product of $\tX_{j,u}^{\pm1}$ with $p < u < p+\sfh$.
\ee
\end{theorem}
 
\subsection{Height functions and their negative subrings} 
A \emph{height function $\xi$ on $\Dynkin$} is a function $\xi:I \to \Z$ given by $i \mapsto \xi_i$ satisfying 
$$\xi_i \equiv \ep_i \pmod 2 \quad \text{ and } \quad |\xi_i-\xi_j|=1 \quad \text{if }i \sim j.$$ For a height function $\xi$ on $\Dynkin$, we set
$$
\lxi\hDynkin_0 \seteq \{ (i,p) \in \hDynkin_0 \ | \  p \le \xi_i \}. 
$$

We call a bijective function $\tbrho: \N \to \lxi\hDynkin_0$ \emph{a compatible reading of $\lxi\hDynkin_0$} if $k > l$ when $\tbrho(k)=(i,p)$, $\tbrho(l)= (j,p+1) \in \lxi\hDynkin_0$ and $i \sim j$. For a  compatible reading $\tbrho$ of $\lxi\hDynkin_0$, we set $\brho: \N \to I^{(\infty)}$ where $\brho(k)=i_k$
for $\tbrho(k)=(i_k,p_k) \in \lxi\hDynkin_0$. We abuse $\brho$ to be the image $(i_1,i_2,\ldots) \in I^{\N}$
of $\N$ under $\brho$ and call it an \emph{$I$-compatible reading of $\lxi\hDynkin_0$}. 

\smallskip

For each $(i,p) \in \lxi\tDynkin_0$, set  the KR-monomial
$$ Z^\xi_{i,p} \seteq \underline{m^{(i)}[p,\xi_i]} \in \calX_q.$$
When there is no danger of confusion, we write $Z_{i,p}$ instead of 
$Z^\xi_{i,p}$ for notational simplicity.
Denote by $\calX_{q,\le \xi}$ the quantum subtorus of $\calX_q$ generated by 
$\{\tX_{i,p} \}_{(i,p) \in \lxi\hDynkin_0}$.

\begin{proposition} [{\cite[Theorem 7.1]{KO22}, see also \cite[Proposition 5.1.1]{B21}}]
Let $\brho=(i_1,i_2,\ldots)$ and 
$\brho'=(j_1,j_2,\ldots)$
be $I$-compatible readings of $\lxi\hDynkin_0$.  
\bnum
\item
For $u \le v$, we have
$$
\ucalN(Z^\xi_{i_u,p_u},Z^\xi_{i_v,p_v}) 
= (\varpi_{i_u}-w^\brho_u\varpi_{i_u}, \varpi_{i_v}+w^\brho_u\varpi_{i_v}),
$$
where $\tbrho(u)=(i_u,p_u)$ and  $\tbrho(v)=(i_v,p_v)$. Hence we have an $\bbA$-algebra isomorphism 
\begin{align} \label{eq: teta iso quantum torus}
\tka_\brho: \calT(\La^{\brho}) \isoto \calX_{q,\le \xi} \quad \text{ given by } \quad
Z_u \longmapsto Z^\xi_{i_u,p_u} =\underline{m^{(i_u)}[p_u,\xi_{i_u}]}.    
\end{align}
\item We have $\brho \overset{c}{\sim}\brho'$ and there exists
$\pi \in \frakS_\N$ such that 
$$
( \La^{\brho},\tB^{\brho}) = \pi( \La^{\brho'},\tB^{\brho'}).
$$
\ee
\end{proposition}

 For an element $\tx \in \calX_q$,
we denote by $\tx_{\le \xi}$ the element of $\calX_{q,\le \xi}$ obtained from $\tx$
by discarding all the monomials containing factors $\tX_{i,p}^{\pm 1}$
with $(i,p) \in \hDynkin_0 \setminus \lxi\hDynkin_0$. We write the $\bbA$-linear map
\begin{align*}
(\cdot)_{\le \xi}: \calX_q \longrightarrow  \calX_{q,\le \xi},  \qquad \tx \longmapsto 
\tx_{\le \xi}.    
\end{align*}

\begin{definition} \label{def: negative subring}
Let $\frakK_{q,\xi}$ be the subring of $\frakK_q$ generated by $\{ F_q(X_{i,p}) \ | \  (i,p) \in \lxi\hDynkin_0\}$ and call it the \emph{$\xi$-negative subring}
of $\frakK_{q}$.    
\end{definition}

We set
$$
(\La^{\xi},\tB^\xi) \seteq \left(  \La^{[\brho]}, \tB^{[\brho]} \right).
$$

\begin{remark} 
In~\cite{JLO2}, the subring $\frakK_{q,\xi}$ was written as $\frakK_{q,\le\xi}$. We hope there is no confusion by these notations.
\end{remark}

\begin{proposition} [{\cite[Proposition 6.1]{HL10}, \cite[Proposition 6.3]{JLO1}}]
For a height function $\xi$ on $\Dynkin$, the $\bbA$-linear map
$ (\cdot)_{\le \xi}$ induces the injective $\bbA$-algebra homomorphism
\begin{align} \label{eq: truncation}
(\cdot)_{\le \xi} : \frakK_{q,\xi} \hookrightarrow \calX_{q,\le \xi}. 
\end{align}
\end{proposition}

We call $(\cdot)_{\le \xi}$ in~\eqref{eq: truncation} the \emph{truncation homomorphism}. Note that, for
a monomial $m^{(i)}[p,\xi_i]$ with $p \le \xi_i$ and $p\equiv_2 \xi_i$, we have
\begin{align}\label{eq: KR-imgae under truncation}
(F_q(m^{(i)}[p,\xi_i]))_{\le \xi} =\underline{m^{(i)}[p,\xi_i]}
\end{align}
by Theorem~\ref{thm: KR-poly}~\eqref{it: range of KR}.

\begin{proposition} [{cf.~\cite{HL16,KO22} see also~\cite[Proposition 5.14]{FHOO2}}] \label{prop: bold b}
For a Dynkin quiver $Q=(\Dynkin,\xi)$, $\bse_u$ $(u \in \N)$ and 
an $I$-compatible reading $\brho$ of $\lxi\hDynkin_0$, let us set
$$\bsb^\brho_u =  \tB^{[\brho]} \bse_u = \sum_{k \in \N} b_{k,u}\bse_k .$$ Then we have
\begin{align} \label{eq: B as image}
\tka_\brho(Z^{\bsb^\brho_u}) = \tscrB_{i,p-1}^{-1}    
\end{align}
where $\tbrho(u)=(i,p)$. 
\end{proposition}

\begin{proof}
We claim that
$$
\tka_\brho(Z^{\bsb^\brho_u}) \cong \tscrB_{i,p-1}^{-1}.
$$
Note that  $b_{u^+,u}=-1$, $b_{u^-,u}=1$, $\tbrho(u^{+})=(i,p-2)$ and $\tbrho(u^{-})=(i,p+2)$ by the assumption on $\brho$. Then we have
$$\tka_\brho(Z^{-\bse_{u^+}+ \bse_{u^-} }) \cong \tX_{i,p}^{-1}\tX_{i,p-2}^{-1}.$$
For $j \sim i$, we see $u^\pm(j)$ satisfy $u<u^+(j)<u^+ < u^+(j)^+$, $u^-(j)<u < u^-(j)^+=u^+(j)<u^+$, $\tbrho(u^{+}(j))=(j, p-1 )$
and $\tbrho(u^{-}(j))=(j,p + 1)$. Since
$b_{u^+(j),u}=-\sfc_{j,i}$ and $b_{u^-(j),u}=\sfc_{j,i}$, we have
$$\tka_\brho(Z^{-\sfc_{ j,i }\bse_{u^+(j)}+ \sfc_{ j,i }\bse_{u^-(j)} }) \cong \tX_{j,p-1}^{-\sfc_{j,i}}.$$  
Hence the claim follows. Since $\tka_{\brho}$
is compatible with the bar involutions on $\calT(\La^{\brho})$ and $\calX_{q,\le \xi}$, we obtain~\eqref{eq: B as image}.
\end{proof}

\begin{theorem} [{\cite[Theorem 5.15]{FHOO2}, \cite[Theorem 8.8]{JLO1}}] \label{thm: HLO iso truncation}
\hfill
\bnum
\item For an $I$-compatible reading $\brho$ of $\lxi\hDynkin_0$, there exists
a unique $\bbA$-algebra isomorphism $\kappa_\brho$ which makes the following diagram commutative:
\begin{align*}
\xymatrix@R=6ex@C=10ex{
\calA_{\brho} \ar[r]^-{\kappa_{\brho}} \ar@{^{(}->}[d] & \frakK_{q,\xi}\ar@{^{(}->}[d]^{(\cdot)_{\le \xi}}\\
\calT(\La^{\brho}) \ar[r]^-{\tka_\brho} &  \calX_{q,\le \xi}
}
\end{align*}
\item \label{it: image of KR under truncation} For $k \in \N$, we have $\kappa_{\brho}(Z_k)= F_q(\underline{m^{(i)}[p,\xi_i]})$
where $\tbrho(k)=(i,p)$.  Moreover every KR-polynomial $F_q(m)$ in $\frakK_{q,\xi}$ 
appears as a cluster variable. 
\item
$\sfF_{q,\xi} \seteq \sfF_{q} \cap \frakK_{q,\xi}$,
$\sfE_{q,\xi} \seteq \sfE_{q} \cap \frakK_{q,\xi}$ and 
$\sfL_{q,\xi} \seteq \sfL_{q} \cap \frakK_{q,\xi}$ are bases of $\frakK_{q,\xi}$. 
\ee
\end{theorem}

\begin{remark} \hfill
\bna
\item The entire ring $\frakK_q(\g)$ has also a quantum cluster algebra structure  (see \cite[Theorem 6.17]{FHOO2}, \cite[Theorem 9.6]{JLO1}), but we will not use the structure in
this paper. 
\item The image $(\frakK_{q,\xi})_{\le\xi}$ of $\frakK_{q,\xi}$ under $(\cdot)_{\le\xi}$ has a cluster algebra structure with initial cluster variables 
$\{ \underline{m^{(i)}[p,\xi_i]} \ | \ i \in I, \ p \le \xi_i \} $ by~\eqref{eq: KR-imgae under truncation} and Theorem~\ref{thm: HLO iso truncation}~\eqref{it: image of KR under truncation}.
\ee
\end{remark}

\begin{theorem}[\cite{HL16,B21,JLO1}] \label{thm: cluster algorithm}
For each KR-monomial $m^{(i)}[p,s]$ and any height function $\xi$ on $\Dynkin$ such that $\xi_i \ge s$,
there exists a quantum cluster algebra algorithm for calculating
$(F_q(m^{(i)}[p,s]))_{\le \xi}$, which can be obtained from the initial cluster variable $\underline{m^{(i)}[p+2k,\xi_i]}$ of $(\frakK_{q,\xi})_{\le \xi}$ with $\xi_i=s+2k$ $(k\in \N_0)$ via a certain sequence of mutations. 
\end{theorem}

\begin{remark}
By Theorem~\ref{thm: KR-poly}~\eqref{it: range of KR}, $(F_q(m^{(i)}[p,s]))_{\le\xi} = F_q(m^{(i)}[p,s]) $ if $k$ in Theorem~\ref{thm: cluster algorithm} is big enough.   
\end{remark}

\begin{conjecture}[\cite{JLO1}]\label{conj: quatum positive}
\hfill
\bnum
\item \label{it: positive} For each KR-monomial $m^{(i)}[p,s]$, the element 
$F_q(m^{(i)}[p,s])$ is quantum positive; i.e., each coefficient of  $F_q(m^{(i)}[p,s])$ 
is contained in $\N_0[q^{\pm 1/2}]$. 
\item \label{it: L F KR} For each KR-monomial $m^{(i)}[p,s]$, we have $L_q(m^{(i)}[p,s]) = F_q(m^{(i)}[p,s])$.
\ee
\end{conjecture}

The conjecture in the above is proved by Nakajima in \cite{Nak04} for symmetric cases (see also \cite{FHOO2}). By Theorem~\ref{thm: cluster algorithm}, 
Conjecture~\ref{conj: quatum positive} is closely related to the quantum Laurent positivity conjecture suggested by Berenstein--Fomin--Zelevinsky. 

\subsection{Dynkin quivers and their heart subrings} \label{subsec: Dynkin quiver}
A \emph{Dynkin quiver} $Q$ is a pair $(\Dynkin,\xi)$ consisting of
(i) a Dynkin diagram $\Dynkin$ of $\g$ and (ii) a height function $\xi$ on $\Dynkin$. 

For a Dynkin quiver $Q=(\Dynkin,\xi)$, we call $i \in I$ a \emph{source} (resp. \emph{sink}) of $Q$, if $\xi_i > \xi_j$ (resp. $\xi_i < \xi_j$) for all $j \in I$
with $i \sim j$. For a Dynkin quiver $Q=(\Dynkin,\xi)$ and its source $i$, we denote by $s_iQ$ the Dynkin quiver $(\Dynkin,s_i\xi)$, where $s_i\xi$
is the height function defined by
$$
(s_i\xi)_j \seteq \xi_j-2 \times \delta_{i,j}. 
$$
We call the operation from $Q$ to $s_iQ$ the reflection of $Q$ at the source $i$
of $Q$.
For any Dynkin quivers $Q$ and $Q'$ with the same Dynkin diagram $\Dynkin$,
there exists a sequence of reflections, including $s_i^{-1}$, from $Q$ to $Q'$. 

For a sequence of indices $\ii=(i_1,i_2, \ldots,i_l) \in I^l$ $(l \in \N \sqcup \{ \infty \})$ and a Dynkin quiver $Q=(\Dynkin,\xi)$, we call $\ii$ \emph{$Q$-adapted} if
$$
\text{$i_k$ is a source of $s_{i_{k-1}}s_{i_{k-2}}\cdots s_{i_1}Q$ for all 
$1 \le  k \le l$.}
$$

The followings are known (see \cite{HL15,KKOP2,KO22}):  
\bnA
\item Any $I$-compatible reading of $\lxi\hDynkin_0$ is $Q$-adapted.
\item For each Dynkin quiver $Q$, there exists a unique Coxeter element
$\tau_Q \in \weyl$, all
of whose reduced sequences are adapted to $Q$. Conversely, for any Coxeter element $\tau$, there exists a unique Dynkin quiver $Q$ such that $\tau =\tau_Q$.
\item  For each Dynkin quiver $Q$, there exists a reduced sequence $\ii$ of $w_\circ$ adapted to $Q$ and the set of all reduced sequence of $w_\circ$ adapted to $Q$ forms a single commutation class of $w_\circ$, denoted by $[Q]$.
\item For a reduced sequence $\ii=(i_1,\ldots,i_\ell)$ of $w_\circ$ adapted to $Q$, define 
$\hii =(\im_1,\im_2,\ldots ) \in I^{(\infty)}$ by 
\begin{align} \label{eq: hii}
\im_k =  i_k \quad \text{ for } 1\le k \le \ell \quad \text{ and } 
\quad \im_{k} = \im^*_{k-\ell} \quad 
\text{ for all } k > \ell.
\end{align}
Then $\hii$ is an $I$-compatible reading of $\lxi\hDynkin_0$ and we have a compatible reading $\tii$ of $\lxi\hDynkin_0$ corresponding to $\hii$.
Furthermore, we have $\hii\overset{c}{\sim} \brho$ for any $I$-compatible reading $\brho$ of $\lxi\hDynkin_0$. 
\item  For each Dynkin quiver $Q=(\Dynkin,\xi)$, there exists a bijection $\phi_Q: \hDynkin_0 \to \Phi^+ \times \Z$ defined recursively as follows: Set $\upgamma_i^Q \seteq (1 - \tau_Q)\varpi_i$ and
\bnum
\item $\phi_Q(i,\xi_i)\seteq (\upgamma_i^Q,0)$ for each $i \in I$,
\item if $\phi_i(i,p)=(\al,k)$, we have 
\begin{align} \label{eq: phi_Q}
\phi_Q(i,p\pm 2) \seteq \bc
(\tau_Q^{\mp1}\al,k) \text{ if } \tau_Q^{\mp1} \in \Phi^+, \\
(-\tau_Q^{\mp1}\al,k\pm 1) \text{ if } \tau_Q^{\mp1} \in \Phi^-.
\ec
\end{align}  
\ee
\ee

For a Dynkin quiver $Q=(\Dynkin,\xi)$, we denote by $\frakD^{\pm 1}Q$ the Dynkin quiver $(\Dynkin,\frakD^{\pm 1}\xi)$, where $\frakD^{\pm 1}\xi$ is the height function defined as follows:
$$
(\frakD^{\pm 1}\xi)_j \seteq \xi_{j^*} \pm \sfh \quad \text{ for } j \in I. 
$$

\begin{proposition}[\cite{HL15,FO21,KO22}]
Let $Q$ be a Dynkin quiver and set 
$\Gamma_0^Q \seteq \phi_{Q}^{-1}(\Phi^+ \times \{0\})$. Then $\Gamma_0$
is a subset of $\hDynkin_0$ characterized by 
$$
\Gamma^Q_0 = \{ (i,p) \in \hDynkin_0 \ |  \ \xi_{i^*} -\sfh < p \le \xi_i \}. 
$$
Furthermore, for any $(j,s) \in \hDynkin_0$, there exists unique $k \in \Z$
and $(i,p) \in \Gamma^Q_0$ such that
$$
\frakD^k(X_{i,p}) = X_{j,s}.
$$
\end{proposition}

We define $\calX_{q,Q}$
the quantum subtorus of $\calX_{q}$ generated by $\{\tX_{i,p} \ | \  (i,p) \in \Gamma^Q_0\}$, and set
$$\calM^Q \seteq \calM \cap \calX_{q,Q} \quad \text{ and } \quad \calM^Q_+ \seteq \calM_+ \cap \calX_{q,Q}.$$ 
We denote by
$\frakK_{q,Q}$ the subring of $\frakK_{q,\xi}$ which is generated by $\{ F_q(X_{i,p}) \ | \  (i,p) \in \Gamma^Q_0\}$. 
We call it
the \emph{heart} subring associated with the Dynkin quiver $Q=(\Dynkin,\xi)$. 

\begin{theorem}[\cite{HL15,HO19,FHOO,FHOO2,JLO1,JLO2}]
Let $\ii$ be a reduced sequence of $w_\circ$ adapted to $Q=(\Dynkin,\xi)$. 
\bnum
\item We have $\bbA$-algebra isomorphisms $\tka_\ii: \calT(\La^{\ii}) \to \calX_{q,Q}$ and 
$\kappa_\ii: \calA_\ii \to \frakK_{q,Q}$, and an $\bbA$-algebra monomorphism
$(\cdot)_{\le \xi}|_{\frakK_{q,Q}}$ which make the following diagram commutative:
\begin{align*}
\xymatrix@R=3ex@C=5ex{
&\calA_{\ii} \ar[rr]^{\kappa_{\ii}} \ar@{^{(}->}[dl] \ar@{^{(}->}[dd] && \frakK_{q,Q}\ar@{^{(}->}[dd]^{(\cdot)_{\le \xi}|_{\frakK_{q,Q}}}  \ar@{^{(}->}[dl] \\ 
\calA_{\hii} \ar[rr]^{\qquad \kappa_{\hii}} \ar@{^{(}->}[dd] && \frakK_{q,\xi}\ar@{^{(}->}[dd]^{(\cdot)_{\le \xi}}
\\
&\calT(\La^{\ii}) \ar[rr]_{\tka_\ii \qquad }  \ar@{^{(}->}[dl]  &&  \calX_{q,Q}  \ar@{^{(}->}[dl]  \\
\calT(\lhii\La) \ar[rr]_{\tka_\hii} &&  \calX_{q,\le \xi}
}
\end{align*}

\item The set $\sfF_{q,Q} \seteq \sfF_{q} \cap \frakK_{q,Q}$,
$\sfE_{q,Q} \seteq \sfE_{q} \cap \frakK_{q,Q}$ and 
$\sfL_{q,Q} \seteq \sfL_{q} \cap \frakK_{q,Q}$ are bases of $\frakK_{q,Q}$. 
\ee
\end{theorem}

The following theorem is proved in \cite[Corollary 5.4]{JLO2} which gives a partial answer for Conjecture~\ref{conj: quatum positive}~\eqref{it: L F KR}.

\begin{theorem}[{\cite[Corollary 5.4]{JLO2}}]
For a KR-monomial $m$ contained in $\calM_+^Q$ for some Dynkin quiver $Q=(\Dynkin,\xi)$, we have
$$
L_q(m) = F_q(m). 
$$
\end{theorem}

Recall that, for a reduced sequence $\ii$ of $w_\circ$ adapted to a Dynkin quiver $Q$, we have an $\bbA$-algebra isomorphism 
$$
\varphi_\ii: \calA_\ii \isoto \calA_\bbA(\n).
$$
Hence we have an $\bbA$-algebra isomorphism 
\begin{align} \label{eq: small triangle 1}
\Psi_{\ii} \seteq \kappa_{\ii} \circ \varphi_\ii^{-1} : \calA_\bbA(\n) \isoto \frakK_{q,Q}.    
\end{align}
Moreover, for another reduced sequence $\ii'$ of $w_\circ$ adapted to $Q$, 
we have $\Psi_{\ii}=\Psi_{\ii'}$ (see \cite{FHOO,JLO2} for more details about the independence of the choice of $\ii \in [Q]$ for the isomorphism $\Psi_{\ii}$).   
Thus the notation $\Psi_{Q}$ make sense.

\begin{theorem} [{\cite{HL15}, \cite[\S 10.3]{FHOO}, \cite[\S 7]{JLO2}}] 
\label{thm: auto}
\hfill 
\bna
\item For each Dynkin quiver $Q$, the isomorphism $\Psi_{Q}$ sends $\tbfB^\up$
to $\sfL_{q,Q}$. 
\item \label{it: i-box} For a reduced sequence $\ii$ of $w_\circ$ adapted to $Q$ and a KR-monomial $m^{(i)}[p,s]$ in $\calM^Q_+$, there exists
a pair of integers $(t,u)$ such that $0 \le t < u \le \ell$, 
$${\rm (i)} \ \ i_t=i_u \ \  \text{ if }\  t>0  \quad \text{ and } \quad {\rm (ii)} \ \ \Psi_{Q}(\tD(w^\ii_u\varpi_{i_u}, w^\ii_t\varpi_{i_u}))=F_q(m^{(i)}[p,s]).$$
Here we understand $w^\ii_0 = {\rm id} \in \weyl$. 
\item
Let $Q=(\Dynkin,\xi)$ and $Q'=(\Dynkin,\xi')$ be Dynkin quivers defined on the same $\Dynkin$. 
Then we have a unique $\bbA$-algebra automorphism
\begin{align} \label{eq: auto}
\Psi_{\frakQ,\frakQ'}:\frakK_q \longrightarrow \frakK_q
\end{align}

satisfying the following properties:
\bnum
\item The automorphisms $\Psi_{\frakQ,\frakQ'}$ restricted to $\frakK_{q,Q'}$ and $\frakK_{q,\xi}$ respectively induce $\bbA$-algebra isomorphisms 
\begin{align}\label{eq: Psi Q,Q'}
\Psi_{Q,Q'}  : \frakK_{q,Q'} \isoto \frakK_{q,Q} \quad \text{ and } \quad \Psi_{\xi,\xi'} : \frakK_{q,\xi'} \isoto  \frakK_{q,\xi}  
\end{align}
such that $\Psi_{Q,Q'} = \Psi_{Q} \circ  \Psi_{Q'}^{-1}$. 
\item We have $\Psi_{\frakQ,\frakQ'} \circ \frakD^{\pm 1}_q = \frakD^{\pm 1}_q \circ \Psi_{\frakQ,\frakQ'}$.
\item $\Psi_{\frakQ,\frakQ'}$ sends $\sfL_{q}$ to  $\sfL_{q}$ bijectively. 
\ee
\ee
\end{theorem}

\begin{theorem}[\cite{KKOP21A,JLO2}]\label{thm: braid}
Let $Q$ be a Dynkin quiver with a source $i \in I$ and write $\TT_i$ for the automorphism $\Psi_{s_i\frakQ,\frakQ}$ of $\frakK_q$. Then
the collection of automorphism $\{ \TT_i \}_{i \in I}$ satisfies the braid relations of $\g$. 
\end{theorem}

For Dynkin quivers $Q=(\Dynkin,\xi)$ and $Q'=(\Dynkin,\xi')$ on the same $\Dynkin$, let $\ii=(i_1,\ldots,i_\ell)$
and $\ii'=(i'_1,\ldots,i'_\ell)$ be reduced sequences of $w_\circ$ adapted to $Q$ and $Q'$, respectively. Then we have a sequence $$\bstau = (\tau_1, \ldots,\tau_r)
\text{ in } \{ \ga_1, \ldots ,\ga_{\ell-1} \} \sqcup
\{ \beta_1, \ldots ,\beta_{\ell-2} \} \sqcup
\{ \eta_1, \ldots ,\eta_{\ell-3} \} \cup \{ \zeta_1 \}$$ such that
$$
\ii' = \tau_1 \cdots \tau_r \ii.
$$
Recall we have $\hii=(i_1,i_2,\ldots),\hii'=(i'_1,i'_2,\ldots) \in I^{(\infty)}$ corresponding to $\ii$ and $\ii'$ respectively, 
which are constructed in~\eqref{eq: hii}. We extend the isomorphism 
$\bstau^*: \calA_{\ii'} \to \calA_\ii$ to the one 
from  $\calA_{\hii'}$ to $\calA_\hii$ as follows: For each $n \in \N$, consider
the $\bbA$-algebra isomorphism
$$\widehat{\bstau}^{(n)*} 
: = \htau_r^{(n)*} \circ \cdots  \circ  \htau_1^{(n)*}: \calA_{\hii'}^{n\ell} \to\calA_\hii^{n\ell},$$
where $\ell = \ell(w_\circ)$ and 
\begin{align*}
    &\hga^{(n)*}_k \seteq \ga_k^*\ga_{k+\ell}^* \cdots \ga_{k+(n-1)\ell}^* ,
    &&\hbe^{(n)*}_k \seteq \be_k^*\be_{k+\ell}^* \cdots \be_{k+(n-1)\ell}^* ,\allowdisplaybreaks\\
    &\heta^{(n)*}_k \seteq \eta_k^*\eta_{k+\ell}^* \cdots \eta_{k+(n-1)\ell}^*, 
    &&\hzeta^{(n)*}_1 \seteq \zeta_1^*\zeta_{1+\ell}^* \cdots \zeta_{1+(n-1)\ell}^*.
\end{align*}
Recalling $\tau_k\circ \tau_l  = \tau_l\circ \tau_k$ when $|k-l| \ge \ell$, we define 
$$
\widehat{\bstau}^{*} \seteq \lim_{n \to \infty} 
\widehat{\bstau}^{(n)*} : \calA_{\hii'} \to  \calA_{\hii}.
$$

The following theorem tells us that the isomorphism $\widehat{\bstau}^{*}$ is compatible with the isomorphisms $\Psi_{\xi,\xi'}$, $\kappa_{\hii'}$ and 
$\kappa_{\hii}$. The argument of a proof can be almost the same as the one in \cite[Theorem 6.3]{FHOO2},
which investigates simply-laced types only. 
For the reader's convenience, we give a proof of the theorem.

\begin{theorem} \label{thm: comm square}
We have the following commutative diagram:
\begin{align*}
\xymatrix@R=6ex@C=10ex{
\calA_{\hii'} \ar[r]^{\kappa_{\hii'}} \ar[d]^{\widehat{\bstau}^*} & \frakK_{q,\xi'} \ar[d]^{\Psi_{\xi,\xi'}} \\
\calA_{\hii} \ar[r]^{\kappa_{\hii}} & \frakK_{q,\xi} \\
}
\end{align*}
\end{theorem}

\begin{proof}
Setting $\Uppsi \seteq \kappa_{\hii} \circ \widehat{\bstau}^* \circ \kappa_{\hii}^{-1}$, it is enough to prove that
$$
\Uppsi = \Psi_{\xi,\xi'}.
$$
By Theorem~\ref{thm: auto}, we have to show the following:
\ben
\item \label{it: Q-commutativity} $\Uppsi|_{\frakK_{q,Q}}  = \Psi_{Q,Q'}$,
\item \label{it: dual-commutativity} $\Uppsi \circ \frakD_q^{-1} = \frakD_q^{-1} \circ \Uppsi$. 
\ee
Noting that $\calA_{\hii'}^\ell \simeq \calA_{\ii'}$, $\calA_{\hii}^\ell \simeq \calA_{\ii}$,
$\widehat{\bstau}^*|_{\calA_{\ii'}} = \bstau^*$, $\kappa_{\hii'}|_{\calA_{\ii'}}=\kappa_{\ii'}$,
$\kappa_{\hii}|_{\calA_{\ii}}=\kappa_{\ii}$
and $\Psi_{\xi,\xi'}|_{\frakK_{q,Q'}}=\Psi_{Q,Q'}$, 
we have
\begin{align}
\raisebox{3em}{\xymatrix@R=0.4ex@C=1ex{
\calA_{\ii'} \ar[rrrr]^{\kappa_{\ii'}} \ar[dddd]_{\bstau^*} \ar[ddrr]^{\varphi_{\ii'}}&&&&  \frakK_{q,Q'}  \ar[dddd]^{\Psi_{Q,Q'}}\\
&&  \circled{A}   \\
& \circled{B} \quad & \calA_{\bbA}(\n) \ar[uurr]^{\Psi_{Q'}} \ar[ddrr]_{\Psi_{Q}} & \quad \circled{C}  \\
&&  \circled{A} \\
\calA_{\ii} \ar[rrrr]_{\kappa_{\ii}} \ar[uurr]_{\varphi_{\ii}} &&&&  \frakK_{q,Q} 
}}
\end{align}
Here the commutativity for $\circled{A}$-triangles follows from~\eqref{eq: small triangle 1},
 the commutativity for $\circled{B}$-triangle from Lemma~\ref{Cor:bftau} and 
the commutativity for $\circled{C}$-triangle from the definition in~\eqref{eq: Psi Q,Q'}.
Thus we have shown~\eqref{it: Q-commutativity}.

Now let us consider~\eqref{it: dual-commutativity}. Setting $\hii^*\seteq (i^*_k)_{k \in \N}$, we have $\hii = \partial_+^\ell \hii^*$,
$\tB^{\ii}= \tB^{\ii^*}$ and $\La^{\ii}=\La^{\ii^*}$. Thus we have an isomorphism $\nu: \calA_{\hii^*} \simeq \calA_{\hii}$ sending $Z_k$ to itself.  
By Proposition~\ref{prop: fshift} and Theorem~\ref{thm: HLO iso truncation}, the cluster variables 
$$
\kappa_{\hii}^{-1} \frakD_q^{-1}F_q(X_{i,p}) \quad \text{ and } \quad \nu (\partial_+^*)^\ell \kappa_{\hii}^{-1}  F_q(X_{i,p}) 
$$
have the same degrees for all $(i,p) \in \lxi\hDynkin_0$, and hence they are the same by Theorem~\ref{thm: same degree same element}.  Since the set 
$\{ F_q(X_{i,p}) \}_{(i,p) \in \lxi\hDynkin_0}$
generates $\frakK_{q,\xi}$, we conclude that
\begin{align} \label{eq: step1}
 \nu \circ (\partial_+^*)^\ell \circ \kappa_{\hii}^{-1} = \kappa_{\hii}^{-1} \circ \frakD_q^{-1}.
\end{align}

Note that the same assertion for $\hii'$ holds by the same argument. Furthermore, for a cluster monomial $x \in \calA_{\hii'}$
whose degree belongs to $C_{\hii'}$, one can see that the degrees of the cluster variables
\begin{align} \label{eq: step2}
\widehat{\bstau}^* \nu (\partial_+^*)^\ell x\quad  \text{ and } \quad   \nu (\partial_+^*)^\ell \widehat{\bstau}^* x \quad
\text{ are the same}
\end{align}
by Proposition~\ref{prop: fshift} and by the result in Section~\ref{sec: Braid and degree} that $\widehat{\bstau}^*$ sends $C_{\hii'}$
to $C_{\hii}$. 
Now we have
\begin{align*}
\Uppsi \frakD_q^{-1} F_{q}(X_{i,p}) &= \kappa_{\hii} \widehat{\bstau}^*  \kappa_{\hii'}^{-1}\frakD_q^{-1} F_{q}(X_{i,p}) \allowdisplaybreaks\\
& = \kappa_{\hii} \widehat{\bstau}^* \nu (\partial_+^*)^\ell  \kappa_{\hii'}^{-1}  F_{q}(X_{i,p}) && \text{ by~\eqref{eq: step1} for $\hii'$ } \allowdisplaybreaks\\
& = \kappa_{\hii} \nu (\partial_+^*)^\ell  \widehat{\bstau}^*  \kappa_{\hii'}^{-1}  F_{q}(X_{i,p}) && \text{ by~\eqref{eq: step2} } \allowdisplaybreaks\\
& = \frakD_q^{-1} \kappa_{\hii} \widehat{\bstau}^*  \kappa_{\hii'}^{-1} F_{q}(X_{i,p}) && \text{ by~\eqref{eq: step1} for $\hii$ } \allowdisplaybreaks\\
& = \frakD_q^{-1} \Uppsi F_q(X_{i,p})
\end{align*}
for any $(i,p) \in \lxip\hDynkin_0$. Since the set 
$\{ F_q(X_{i,p}) \}_{(i,p) \in \lxip\hDynkin_0}$ generates $\frakK_{q,\xi'}$, the assertion~\eqref{it: dual-commutativity}  follows.
\end{proof}

\begin{remark} \label{rmk: our purpose}
In the above theorem, we have applied a special sequence $(i_1,i_2,\ldots) \in I^{(\infty)}$ such that $(i_k,\ldots, i_{k+\ell-1})$ is
a reduced sequence of $w_\circ$ for all $k \in \N$.
\end{remark}

\begin{corollary}
Let $Q$ and $Q'$ be Dynkin quivers on the same Dynkin diagram $\Dynkin$.
For reduced sequences $\ii$ and $\ii'$ adapted to $Q$ and $Q'$ respectively, the $\bbA$-algebra isomorphism $\kappa_{\hii}^{-1} \circ \Psi_{\frakQ,\frakQ'}
\circ \kappa_{\hii'}$ induces a bijection between the set of quantum cluster monomials of $\calA_{\hii'}$ and that of   $\calA_{\hii}$.
\end{corollary}

\subsection{Substitution formulas}
One of the celebrated results of Berenstein--Fomin--Zelevinsky is to show that quantum cluster algebra 
is still contained in $\Z[q^{\pm 1/2}][Z^{\pm 1}_{k}]_{k \in \sfK}$ even though mutations involves non-trivial fractions. 
In this subsection, we shall show that the automorphism $\Psi_{\frakQ,\frakQ'}$ can be interpreted as an restriction of birational 
auto-transformation of $\bbF(\calX_q)$, which is referred to as \emph{substitution formula} in~\cite[Section 7]{FHOO2}. Since $\Psi_{\frakQ,\frakQ'}$
can be understood as a composition of $\{\TT^{\pm 1}_i \}_{i \in I}$, it gives
an interesting relation among elements in $\sfL_q$. The result in this section for simply-laced type is given in~\cite[Section 7]{FHOO2} and
we apply the same framework and arguments of proofs in this paper for non simply-laced type, based on the results that have been obtained for $\frakK_q$, 
$4$-moves and $6$-moves. 
Thus we just give statements, several notations and maps, so that readers may recover the substitution formula on $\bbF(\calX_q)$ and $\frakK_q$
by themselves. 
\smallskip

Throughout this subsection, we (i) fix Dynkin quivers $Q=(\Dynkin,\xi)$ and $Q'=(\Dynkin,\xi')$ defined on the same Dynkin diagram $\Dynkin$
of non simply-laced type, (ii) denote by
$\ii=(i_1,\ldots,i_\ell)$ and $\jj=(j_1,\ldots,j_\ell)$ reduced sequences of $w_\circ$ adapted to $Q$ and $Q'$ respectively. To distinguish generators of quantum tori $\calT(\La^\hii)$ and $\calT(\La^\hjj)$, we use $Z_k$ for initial cluster variables $\calT(\La^\hii)$ while $Z'_k$ for ones of $\calT(\La^\hjj)$. 

\smallskip
For $k \in \Z$,
we set $\xi[k] \seteq \frakD^k\xi$ and $Q[k] \seteq \frakD^kQ$ for notational simplicity. We recall the following:
\bna
\item We have injective homomorphisms 
$$
\frakK_{q,\xi[-1]} \hookrightarrow (\frakK_{q,\xi})_{\le \xi} \hookrightarrow \calX_{q,\le \xi}
$$
induced from $(\cdot)_{\le \xi}$, whose composition is an identity indeed, by Theorem~\ref{thm: KR-poly} and the fact that $\frakK_{q,\xi}$ is generated by $\{ F_q(X_{i,p}) \ | \ (i,p) \in \lxi\hDynkin_0 \}$. 
\item We have a commutative diagram
\begin{align*}
\xymatrix@R=4ex@C=12ex{
\frakK_{q,\xi} \ar[dr]^{\simeq} \ar[r]^{(\cdot)_{\le \xi}} & \calX_{q,\le \xi} \\
\frakK_{q,\xi[-1]} \ar@{^{(}->}[u] \ar@{^{(}->}[r]   &  (\frakK_{q,\xi})_{\le \xi} \ar@{^{(}->}[u]  
}
\end{align*}
\ee
Then, for $I$-compatible readings  $\hii$ (resp. $\hii'$) of $\lxi\hDynkin_0$ (resp. $\lxip\hDynkin_0$), 
we have unique morphisms $\tka_{\hii,\bbF}$, $\tka_{\hii',\bbF}$, $\widehat{\bstau}^*_\bbF$ and $\tPsi_{\xi,\xi}$
which make the diagram below commutative:
\begin{align}
\raisebox{3.5em}{\xymatrix@R=4ex@C=7ex{
&\bbF(\calX_{q,\xi'}) \ar@/^2.5pc/[rrr]^{\tPsi_{\xi,\xi'}} & \bbF(\calT(\La^{\hjj}))\ar[r]_{\sim}^{\widehat{\bstau}^*_\bbF} \ar[l]_{\tka_{\hjj,\bbF}}^\sim  
&  \bbF(\calT(\La^{\hii})) \ar[r]^{\tka_{\hii,\bbF}}_\sim & \bbF(\calX_{q,\xi})  \\
&\calX_{q,\le \xi'} \ar@{^{(}->}[u] &  \calT(\La^{\hjj}) \ar[l]^{\sim}_{\tka_{\hjj}} \ar@{^{(}->}[u] & 
\calT(\La^{\hii})\ar@{^{(}->}[u] \ar[r]_{\sim}^{\tka_{\hii}} & \calX_{q,\xi} \ar@{^{(}->}[u]\\
(\frakK_{q,\xi'})_{\le \xi'}  \ar@{^{(}->}[ur] & \frakK_{q,\xi'} \ar@/_1.8pc/[rrr]^{\Psi_{\xi,\xi'}}_\sim  \ar[l]_{\sim} \ar[u]_{(\cdot)_{\le \xi'}}& \calA_{\hjj}\ar@{^{(}->}[u] \ar[l]_{\sim}^{\kappa_{\hjj}} \ar[r]_{\sim}^{\widehat{\bstau}^*}
& \calA_{\hii} \ar@{^{(}->}[u] \ar[r]^{\sim}_{\kappa_{\hii}} & \frakK_{q,\xi} 
\ar[r]^{\sim} \ar[u]^{(\cdot)_{\le\xi}} & (\frakK_{q,\xi})_{\le \xi} \ar@{^{(}->}[ul] \\
& \frakK_{q,\xi'[-1]} \ar@{^{(}->}[u] \ar@{^{(}->}[ul] \ar@/_1pc/[rrr]_{\Psi_{\xi'[-1],\xi[-1] }}^\sim &&&  \frakK_{q,\xi[-1]} \ar@{^{(}->}[u] \ar@{^{(}->}[ur] 
}}
\end{align}

Let us set 
\begin{align*}
&\tSigma \seteq \tka_{\hii,\bbF}^{-1} \circ \frakD^{-1}_q \circ  \tka_{\hii,\bbF} : \bbF(\calT(\La^{\hii})) \to \bbF(\calT(\La^{\hii})), \\
&\tSigma'  \seteq \tka_{\hjj,\bbF}^{-1} \circ \frakD^{-1}_q  \circ  \tka_{\hjj,\bbF} : \bbF(\calT(\La^{\hjj})) \to \bbF(\calT(\La^{\hjj})),
\end{align*}
which are automorphisms of skew field of fractions. We also define $\Z$-module homomorphisms $\Upsigma,\tUpsigma: \Z^{\oplus \N} \to \Z^{\oplus \N}$
given by
$$
\Upsigma(\bse_u)=\bse_{u+\ell} \qquad \tUpsigma(\bse_u) = \bse_{u+\ell}- \bse_{(\ell+1)_{\hii}^-(i^*_u)} \quad \text{ for all } u \in \N.
$$ 
Note that
\begin{align} \label{eq: tupsigma upsigma}
\tUpsigma(\bsa) = \Upsigma(\bsa) - \sum_{i \in I} \sfp_\hii(\bsa;i)\bse_{(\ell+1)_{\hii}^-(i^*)} \quad \text{ for $\bsa \in \Z^\N$}. 
\end{align}

The proof of the following lemma utilizes Proposition~\ref{prop: bold b}. 

\begin{lemma} \label{lem: tSigma}
For $\bsa \in \Z^{\oplus \N}$, we have 
$$
\tSigma(Z^{\bsa}) = Z^{\tUpsigma(\bsa)} \quad \text{ and } \quad \tSigma'(Z^{\prime\bsa}) = Z^{\prime\tUpsigma(\bsa)}. 
$$
In particular, 
\begin{align} \label{eq: b ell}
\tUpsigma(\bsb^{\hii}_u) = \bsb^{\hii}_{ u+\ell } \quad \text{ and } \quad \tUpsigma(\bsb^{\hjj}_u) = \bsb^{\hjj}_{ u+\ell }.
\end{align}
\end{lemma}

For $u \in \N$ and an initial cluster variable $Z'_u$ of $\calA_{\hjj}$, we set $\calZ_u \seteq \widehat{\bstau}^*(Z'_u)$, which is a cluster variable of
$\calA_{\hii}$. Then there exists $\bsg_u \in \Z^{\oplus \N}$ and $c_{u,\bsn}$ satisfying
$$
\calZ_u = Z^{\bsg_u}\left(
1+ \sum_{\bfn \in \N_0^{\oplus \N} \setminus \{0\}} c_{u,\bfn} Z^{\sum_{k \in \N} n_k \bsb^\hii_k}
\right) 
$$
by Theorem~\ref{thm: same degree same element}.  

For $k \in \N_0$, we set
$$
\Gamma^{Q[-k]}_{0,\fr} \seteq \{ (i, \xi[-k-1]_i+2) \ | \ i\in I \}, \quad \Gamma^{Q[-k]}_{0,\ex} = \Gamma^{Q[-k]}_0 \setminus \Gamma^{Q[-k]}_{0,\fr}. 
$$
Note that, for each $u \in \N$, there exists a unique $k \in \N_0$ such that $\calZ_u$ can be obtained from the initial cluster $\{Z_u \ | \ u \in \N \}$ by applying the mutations 
only at the vertices labeled by $\Gamma^{Q[-k]}_{0,\ex}$. By $\ell$-periodicity of $\hii$, $\hjj$ and $\widehat{\bstau}^*$, for $u \in \N$, we have
\begin{align} \label{eq: g ell}
\bsg_{u+\ell}  &\bc
= \Upsigma(\bsg_u) & \text{ if } u > \ell \text{ or } u = (\ell+1)_{\hjj}^-(i) \text{ for } i \in I, \\
\in \Upsigma(\bsg_u) + \sum_{i\in I} \Z \bse_{ (\ell+1)^-_{\hii}(i) } & \text{ otherwise},
\ec \end{align}
and   
\begin{align} \label{eq: c ell}  
c_{u+\ell,\bsn} = &\bc
c_{u,\Upsigma^{-1}(\bsn)}  & \text{ if } \bsn \in \Upsigma(\N_0^{\oplus \N}), \\
0 & \text{ otherwise}.
\ec 
\end{align}

\begin{remark} \label{rmk: center Z}
Any braid move between $\hii$ and $\hjj$ does not occur at $(\ell+1)_{\hii}^-(i)+k\ell$ for any $k \in \N_0$ and $i \in I$.
Hence $\calZ_{(\ell+1)_{\hii}^-(i)+k\ell} = Z_{(\ell+1)_{\hii}^-(i)+k\ell}= Z'_{(\ell+1)_{\hii}^-(i)+k\ell}$ in $\calA_{\hii}$ and they
$q$-commute with all $Z_a$ $(a \in \N)$. 
\end{remark}

When $\g$ is of non-simply-laced type, 
we need to employ Lemma~\ref{lem: 4g Pgi}, Lemma~\ref{lem: i-degree 6g}
and Lemma~\ref{lem: 6g Pgi} to prove the following lemma.    

\begin{lemma}
For $u \in \N$, we have
\begin{align} \label{eq: tSigma Zu}  
\tSigma(\calZ_u) & =  Z^{\tUpsigma(\bsg_u)}\left(
1+ \sum_{\bfn \in \N_0^{\oplus \N} \setminus \{0\}} c_{u+\ell,\bfn} Z^{\sum_{k \in \N} n_k \bsb^{\hii}_{k}} \right) \simeq Z^{\tUpsigma(\bsg_u)-\bsg_{u+\ell}} \calZ_{u+\ell}.   
\end{align} 
\end{lemma}

\begin{remark}
The following proposition is a non-symmetric analogue of \cite[Lemma 7.4 (1)]{FHOO2}, and it is \emph{not at all} a consequence of  Theorem~\ref{thm: comm square} as emphasized in the introduction of \cite[Section 7]{FHOO2}.
\end{remark}

\begin{proposition} \label{prop: compatible with dual}
We have 
$$\frakD_{q}^{-1}|_{\bbF(\calX_{q,\le\xi})} \circ \tPsi_{\xi,\xi'}= \tPsi_{\xi,\xi'} \circ \frakD_{q}^{-1}|_{\bbF(\calX_{q,\le\xi'})}.$$ 
\end{proposition}

\begin{corollary}
For $k \in \N_0$ and $(i,p) \in \Gamma^{ Q'[-k]}_0$, we have 
$$
\tPsi_{\xi,\xi'}(\tX_{i,p}) \in \bbF(\calX_{q,Q[-k]}). 
$$
\end{corollary}

\begin{proof}
Since  $\tPsi_{\xi,\xi'}(\bbF(\calX_{q,Q'})) \subset \calX_{q,Q}$, it follows from Proposition~\ref{prop: compatible with dual}.   
\end{proof}

\begin{theorem} \label{thm: substituion}
There exists an automorphism of the skew field $\bbF(\calX_q)$ 
$$
\tPsi_{\frakQ,\frakQ'}: \bbF(\calX_q) \isoto \bbF(\calX_q) 
$$
satisfying the following properties:
\bna
\item $\frakD_q \circ \tPsi_{\frakQ,\frakQ'} = \tPsi_{\frakQ,\frakQ'} \circ  \frakD_q$.
\item $\tPsi_{\xi,\xi'}(\tX_{i,p}) \in \bbF(\calX_{q,Q[k]})$ for any $k \in \Z$ and $(i,p) \in \Gamma^{ Q'[k]}_0$.
\item We have the following commutative diagram:
$$
\xymatrix@R=4ex@C=12ex{
\bbF(\calX_{q}) \ar[r]^{\tPsi_{\frakQ,\frakQ'}}_\sim  &  \bbF(\calX_{q}) \\
\frakK_q  \ar[r]^{\Psi_{\frakQ,\frakQ'}}_\sim \ar@{^{(}->}[u] &  \frakK_q \ar@{^{(}->}[u]
}
$$
\ee
\end{theorem}

\begin{proof}
We first define $\tPsi_{\frakQ,\frakQ'}$ from $\tPsi_{\xi,\xi'}$ by using $\frakD_q$ as follows: For each $y \in \calF(\calX_q)$, there exists
$k \in \N_0$ such that $\frakD_q^{-k}(y) \in \bbF(\calX_{q,\le\xi})$. Then we define
$$\tPsi_{\frakQ,\frakQ'}(y) = \frakD_q^{k} (\tPsi_{\xi,\xi'}(\frakD_q^{-k}(y))).$$
Now, as in~\cite[Theorem 7.1]{FHOO2}, one can prove that it is a well-defined automorphism of $\bbF(\calX_q)$ satisfying the properties in the statement. 
\end{proof}

In Appendix~\ref{sec: substitution formula}, we provide examples of substitution formulas.

\section{Categorification via quiver Hecke algebras and Laurent family} \label{sec: quiver Hecke and Laurent}

In this section, we first review the results on the categorification theory via the quiver Hecke algebra $R$ related to our purpose. Then we recall the definitions and the results about the Laurent family of simple $R$-modules,
which were introduced and investigated in~\cite{KKOP23} very recently. 

\subsection{Quiver Hecke algebra}
Let $\bfk$
be a base field. For $i,j \in I$, we choose polynomials $\calQ_{i,j}(u,v) \in \bfk[u,v]$ such that 
\bna
\item $\calQ_{i,j}(u,v)=\calQ_{j,i}(v,u)$,
\item $\calQ_{i,j}(u,v) =\displaystyle\delta(i \ne j) \hspace{-1ex}\sum_{2\sfd_i p+2\sfd_j q=-2(\al_i,\al_j)} \hspace{-1ex}
t_{i,j;p,q}u^pv^q$ \quad where $t_{i,j;-a_{i,j},0} \in \bfk^\times$. 
\ee
For $\be \in \rl^+$ with $|\be|=n$, we set 
$$
I^\be \seteq \{ \nu= (\nu_1,\ldots,\nu_n) \in I^n \ | \ 
\sum_{k=1}^n \al_{\nu_k} =\be 
\}.
$$

For $\g$ and $\be \in \rl^+$ with $|\be|=r$, the quiver Hecke algebra $R(\be)$
associated with $(\calQ_{i,j})_{i,j \in I}$ is the $\Z$-graded algebra
over $\bfk$ generated by  
$$ x_k \ \ (1 \le k \le r), \quad e(\nu) \ \ (\nu \in I^{\be}), \quad \tau_s \ (1\le s < r), $$
subject to certain defining relations (see \cite{KKKO18} for the relations). Note that
the $\Z$-grading of $R(\be)$ is determined by the degrees of following elements:
$$
\deg(e(\nu)) = 0,  \ \  \deg(x_ke(\nu)) =2\sfd_{\nu_k} \ (1 \le k \le n), \ \text{ and }
 \ \deg(\tau_me(\nu)) = -(\al_{\nu_m},\al_{\nu_{m+1}})
 \ (1 \le m <n). 
$$

We denote by $R(\be)\gmod$ the category of finite-dimensional graded $R(\be)$-modules with homogeneous homomorphisms. For $M \in R(\be)\gmod$, we set
$\wt(M)=-\be \in \rl^-$. Denote by $q$ the degree shift functor
acting on $R(\be)\gmod$ as follows:
$$
q(M)_n = M_{n-1} \qquad \text{ for } M = \soplus_{k \in \Z} M_k \in R(\be)\gmod. 
$$

For an $R(\be)$-module $M$ and an $R(\ga)$-module $N$, we obtain an $R(\be+\ga)$-module by
$$
M \conv N \seteq R(\be+\ga) e(\be,\ga) \otimes_{R(\be)\otimes R(\ga)} 
(M \conv N),
$$
where $e(\be,\ga) \seteq \sum_{\nu \in I^\be, \nu' \in I^\ga} e(\nu*\nu') \in R(\be+\ga)$. Here $\nu *\mu$ denotes the concatenation of $\nu$ and $\mu$, and 
$\conv$ is called the \emph{convolution product}. We say that two simple $R$-modules $M$ and $N$ \emph{strongly commute} if $M \conv N$ is simple. If a simple module $M$ strongly commutes with itself, then $M$ is called \emph{real}. 

Set $R\gmod \seteq \oplus_{\be \in \rl^+}R(\be)\gmod$. Then $R\gmod$
has the monoidal category structure with the unit object $\mathbf{1}=\bfk \in R(0)\gmod$ and the tensor product $\conv$. Thus the Grothendieck group $K(R\gmod)$
has the $\Z[q^{\pm 1}]$-algebra structure derived from $\conv$ and the degree shift functors $q^{\pm1}$. We set $$\calK_\bbA(R\gmod) \seteq \bbA \otimes_{\Z[q^{\pm 1}]}K(R\gmod).$$

\begin{theorem}[\cite{KL1,KL2,R08}]There exists an $\bbA$-algebra isomorphism
$$\Upomega : \calK_\bbA(R\gmod) \isoto \calA_\bbA(\n).$$
\end{theorem}

\begin{proposition}[\cite{KKOP18}]\label{prop: det}
For $\la \in \wl^+$  and $\mu,\zeta \in \weyl\la$ with $\mu \preceq \zeta$, there 
exists a self-dual real simple $R(\zeta-\mu)$-module $L(\mu,\zeta)$ such that
$$
\Upomega([L(\mu,\zeta)]) = D(\mu,\zeta). 
$$
Here $[L(\mu,\zeta)]$ denotes the isomorphism class of $L(\mu,\zeta)$.  
\end{proposition}

We call $L(\mu,\zeta)$ the \emph{determinantal module} associated with $D(\mu,\zeta)$.

\subsection{$R$-matrices and their $\Z$-invariants} For $\be \in \rl^+$ and $i \in I$, let
\begin{align*} 
\mathfrak{p}_{i,\be} = \sum_{ \nu\,  \in I^\be}  \big( \prod_{a \in [1,|\be|]; \  \nu_a  = i} x_a \big) e(\nu) \in \frakX(R(\be)),
\end{align*}
where $\frakX(R(\be))$ denotes the center of $R(\be)$.

\begin{definition}[{\cite[Definition 2.2]{KP18}}]
For an $R(\be)$-module $L$, we say that $L$ \emph{admits an affinization} if  there exists an $R(\be)$-module $\widehat{L}$
satisfying the condition: there exists an endomorphism $z_{\widehat{L}}$ of degree $t  \in \Z_{>0}$ such that $\widehat{L}/z_{\widehat{L}}\widehat{L} \simeq L$ and
\bnum
\item $\widehat{L}$ is a finitely generated free modules over the polynomial ring $\bfk[z_{\widehat{L}}]$,
\item $\mathfrak{p}_{i,\be}  \widehat{L} \ne 0$ for all $i \in I$.
\ee
We say that a simple $R(\be)$-module $L$ is \emph{affreal} if $L$ is real and admits an affinization.
\end{definition}

It is known that any $L \in R(\be)\gmod$ admits an affinization if $R(\be)$ is symmetric. However,  when $R(\be)$ is not symmetric, it is widely open whether an $R(\be)$-module  $L$ admits an affinization or not.

\begin{theorem} [{\cite[Theorem 3.26]{KKOP21}}]
For $\varpi \in\wl^+$ and $\mu,\zeta \in \weyl\varpi$ such that $\mu\preceq \zeta$, the determinantal module $L(\mu,\zeta)$ is affreal.
\end{theorem} 
 Let $\ii$ be a reduced sequence
of $w_\circ$. For $\be \in \Phi^+$ such that $\be = \be^\ii_k$ in~\eqref{eq: param positive roots}, we write
$S^\ii_k$ for the determinantal $R(\be)$-module such that
$$
\Upomega([S^\ii_k]) = D(w^\ii_k\varpi_{i_k},w^\ii_{k^-}\varpi_{i_k}). 
$$
We sometimes write $S^{\ii}(\be)$ for $S^\ii_k$ to emphasize $\be$. Note that $$\wt(S^\ii_k)= \wt(D(w^\ii_k\varpi_{i_k},w^\ii_{k^-}\varpi_{i_k}))=-\be.$$  
The module $S^\ii_k$ is also called the \emph{cuspidal module} associated with $\ii$ and $\be$. For another reduced sequence $\ii'$ of $w_\circ$ such that $\ii \overset{c}{\sim} \ii'$, we have  $S^\ii(\be) \simeq S^{\ii'}(\be)$. Thus
the notation $S^{[\ii]}(\be)$ make sense.

\begin{proposition}%
[{\cite{KKKO18,KKOP21}}]\label{prop: l=r} 
Let $M$ and $N$ be simple modules such that one of them is affreal.
Then there exists a unique $R$-module homomorphism $\Rr_{M,N} \in \HOM_R(M,N)$ satisfying
$$\HOM_R(M \conv N,N \conv M)=\bfk\, \Rr_{M,N}.$$
\end{proposition}
We call the homomorphism $\Rr_{M,N}$ the \emph{$R$-matrix}.

\begin{definition} \label{def: inv}
For simple $R$-modules $M$ and $N$ such that one of them is affreal, we define
\begin{align*}
\La(M,N) &\seteq  \deg (\Rr_{M,N}) , \\
\tLa(M,N) &\seteq   \frac{1}{2} \Bigl( \La(M,N) + \bl\wt(M), \wt(N)\br \Bigr) , \\
\de(M,N) &\seteq  \frac{1}{2} \big( \La(M,N) + \La(N,M)\big).
\end{align*}
\end{definition}
It is proved in \cite{KKKO18,KKOP21} that the invariants $\tLa(M,N)$ and $\de(M,N)$ in Definition~\ref{def: inv} belong to $\Z_{\ge 0}$.

\smallskip

For simple modules $M$ and $N$, let $M \hconv N$ and  $M \sconv N$ denote the head and the socle of $M \conv N$, respectively.

\begin{proposition}  [{\cite[Lemma 3.1.4]{KKKO15}, \cite[Lemma 4.1.2]{KKKO18}}]
\label{prop: simple head and socle}
Let $M$ and $N$ be self-dual simple $R$-modules such that one of them is affreal. 
\bna
\item  $M \hconv N$ and  $M \sconv N$ are simple modules and they appear exactly once in the composition series of $M \conv N$.
\item  $  q^{\tLa(M,N)} M \hconv N$ is a self-dual simple module.
\item \label{it: de 0 simple} If $M \conv N \simeq N \conv M$ up to a grading shift, then $M \conv N$ is simple, which is equivalent to $\de(M,N)=0$ and $M \hconv N \simeq M \sconv N$ up to a grading shift.
\ee
 \end{proposition}

For a reduced sequence $\ii$ of $w_\circ$, we say an interval $[a,b] \subset [1,\ell]$ is an \emph{$i$-box} if  $i_a=i_b$. 
For each
$i$-box $[a,b]$ and $\ii$, we define a self-dual simple modules $L^\ii[a,b]$ as follows:  $ L^\ii[a,a] \seteq  S^\ii_a$ and  
$$ L^\ii[a,b] \seteq q^{\tLa\big(S^\ii_b,L^\ii[a,b^-]\big)} S^\ii_b  \hconv L^\ii[a,b^-].$$

\begin{proposition}[{\cite[Proposition 4.6]{KKOP18}}] \label{prop: Upomega L}
 For a reduced sequence $\ii=(i_1,\ldots,i_\ell)$ of $w_\circ$ and an $i$-box $[a,b] \subset [1,\ell]$, we have 
$$
\Upomega\left(\big[L^\ii[a,b]\big]\right) = D(w_{b}^\ii \varpi_{i_b},w_{a^-}^\ii \varpi_{i_a}).
$$
Hence $L^\ii[a,b]$ is an affreal simple module. 
\end{proposition}

\subsection{Laurent family}
Let $J$ be an index set and $\scrC$ be a full monoidal subcategory of $R\gmod$ stable by taking subquotients, extensions and grading shifts.
We say that a family of affreal modules $\calL=\{ L_j \}$ in $\scrC$ is a commuting family in $\scrC$ if 
$$
L_i \conv L_j \simeq L_j \conv L_i \quad \text{ up to a grading shift for any $i,j \in J$.} 
$$
Then there exists a family $\{ \calL(\bfa) \ | \ \bsa \in \N_0^{\oplus J} \}$ of simple modules in $\scrC$ such that
\begin{align*}
&\calL(0) = \mathbf{1}, \quad \calL(\bse_j) = L_j \quad \text{ for any } j \in J, \\
&\calL(\bsa) \conv \calL(\bsb) \simeq q^{-\tLa(\calL(\bsa),\calL(\bsb) )}\calL(\bsa+\bsb) \quad \text{ for any } \bsa,\bsb \in \N_0^{\oplus J}.
\end{align*}

We say a commuting family $\calL=\{ L_j \}_{j \in J}$ is \emph{independent} if, when  $\bsa,\bsb \in \N_0^{\oplus J}$ satisfies $\calL(\bsa) \simeq  q^s\calL(\bsb)$
for some $s \in \Z$, we have $\bsa=\bsb$. 

\begin{definition} \label{def: quasi-Laurent}
A commuting family $\calL= \{ L_j \mid j \in J \}$
of affreal simple objects of $\scrC$ is said to be a \emph{quasi-Laurent family} in $\scrC$  if
$\calL$ satisfies the following conditions:
\bna
\item\label{it:dsc}
$\calL$ is independent, and
\item if a simple module $X$ commutes with all $L_j$ $(j \in J)$,
then there exist $\bsa,\bsb \in \Z_{\ge 0}^{J}$ such that
$$
X \conv \calL(\bsa) \simeq \calL(\bsb)\quad \text{up to a grading shift}.
$$
\setcounter{myc}{\value{enumi}}
\ee
If $\calM$ satisfies \eqref{it:dsc} and \eqref{it:Laur} below, then we say that $\calL$
is a {\em Laurent family}:
\bna
\setcounter{enumi}{\value{myc}}
\item \label{it:Laur}  if a simple module $X$ commutes with all $L_j$ $(j \in J)$,
then there exists $\bsb\in \Z_{\ge 0}^{J}$ such that $X\simeq\calL(\bsb)$ up to a grading shift.
\ee
\end{definition}

For an index set $J = J_\ex \sqcup J_\fr$,
assume that $\calK_\bbA(\scrC) \seteq \bbA \otimes_{\Z[q^{\pm 1}]} K(\scrC)$ is isomorphic to a quantum cluster algebra $\scrA$ such that 
$\Upomega(\calK_\bbA(\scrC)) \simeq \scrA$.
We say that a commuting family $\calL=\{  L_i \}_{i\in J}$ in $\scrC$ is a \emph{monoidal cluster} in $\scrC$ if there exists
a seed  $(\{ Z_i \}_{i \in J}, \La=(\La_{i,j})_{i,j \in J}, \tB=(b_{i,j})_{i \in J,j\in J_\ex} )$ of $\scrA$ such that
$$   Z_i \cong \Upomega( [L_i]) \ \ \text{ in $\scrA$ } \quad \text{ and } \quad \La_{i,j} = -\La(L_i,L_j).$$

\begin{theorem}[{\cite[Proposition 3.6]{KKOP23}}]  \label{thm: Laurent}
Assume that $\calK_\bbA(\scrC)$ has a quantum cluster algebra structure, and let
$\calL=\{  L_j \mid  j\in J\}$ be a monoidal cluster in $\scrC$. Then the commuting family $\calL$
is a quasi-Laurent family.
In particular,
the isomorphism class $[X]$ of a module $X$ in $\calK_\bbA(\scrC)$
can be expressed as an element in the Laurent polynomial $\Z[q^{\pm 1}][\;[ L_j] \mid j \in J]$ with positive coefficients. Moreover, if every $[L_k]$ is prime in $\calK_\bbA(\scrC)\vert_{q=1}$,
then $\calL$ is a Laurent family.
\end{theorem}

\begin{remark}
Note that if $K(\scrC)\vert_{q=1}$ is factorial, then every $[L_k]$ is prime in $K(\scrC)\vert_{q=1}$ (see \cite{GLS13A}).    
\end{remark}

For a reduced sequence $\ii$ of $w_\circ$ and $1 \le k \le \ell$, we set
$$L^\ii_k \seteq L(w^\ii_k\varpi_{i_k},\varpi_{i_k}) \simeq L^{\ii}[k^{\min},k].$$
 
\begin{proposition}[{\cite[Theorem 4.12]{KKOP18}}]  \label{prop: La det}
Let $\ii$ be a reduced sequence of $w_\circ$. For $1 \le s < t \le \ell$, we have
$$
-\La(L^\ii_s, L^\ii_t) = (\varpi_{i_s} - w^\ii_s\varpi_{i_s},\varpi_{i_t} + w^\ii_t\varpi_{i_t} ). 
$$
Hence $\calL^{\ii} \seteq \{ L^\ii_{k} \}$ forms a commuting family.
\end{proposition}

\section{Quantum positivity} \label{sec: positivity}
In this section, we give a partial proof of the quantum positivity conjecture (Conjecture~\ref{conj: quatum positive}~\eqref{it: positive}) for non-simply-laced type. Thus throughout this section, we assume that $\g$ is of non-simply-laced type.

\medskip

Recall the following:
\begin{eqnarray} &&
\parbox{95ex}{
\bna
\item For any reduced sequence $\ii$ of $w_\circ$, there exists an $\bbA$-algebra isomorphism 
$$
\varphi_\ii : \calA_\ii \longrightarrow \calA_\bbA(\n)
$$
sending $Z_k$ to $\tD(w^\ii_k\varpi_{i_k},\varpi_{i_k})$. Moreover, this isomorphism 
sends the seed $(\{ Z_k \}_{1 \le  k\le \ell}, \La^\ii,\tB^\ii )$ of $\calA_\ii$ to 
the seed $(\{ \tD(w^\ii_k\varpi_{i_k},\varpi_{i_k}) \}_{1 \le  k\le \ell}, \La^\ii,\tB^\ii )$ of $\calA_{\bbA}(\n)$.  
\item  There exists an $\bbA$-algebra monomorphism $(\cdot)_{\le \xi}: \frakK_{q,Q} \to \calX_{q,\le \xi}$
for each Dynkin quiver $Q=(\Dynkin,\xi)$. Thus the image $(\frakK_{q,Q})_{\le \xi}$ is isomorphic to $\frakK_{q,Q}$
as $\bbA$-algebras. Moreover  $(F_q(m^{(i)}[p,\xi_{i}])_{\le \xi}= \underline{m^{(i)}[p,\xi_{i}]}$ for $(i,p) \in \Gamma^Q_0$.  
\item For a reduced sequence $\ii$ of $w_\circ$ adapted to a Dynkin quiver $Q=(\Dynkin,\xi)$, there exists an $\bbA$-algebra isomorphism 
$$
\kappa_\ii : \calA_\ii \longrightarrow \frakK_{q,Q}
$$
sending 
the seed $(\{ Z_k \}_{1 \le  k\le \ell}, \La^\ii,\tB^\ii )$ of $\calA_\ii$ to 
the one $(\{ F_q([p_k,\xi_{i_k}] \}_{1 \le  k\le \ell}, \La^\ii,\tB^\ii )$ of $\frakK_{q,Q}$, where $\tii(k)=(i_k,p_k)$.  
\item For a reduced sequence $\ii$ of $w_\circ$ adapted to a Dynkin quiver $Q=(\Dynkin,\xi)$, Proposition~\ref{prop: Upomega L} tells us that 
$$
\Upomega(L^\ii_k) \cong \tD(w^\ii_k\varpi_{i_k},\varpi_{i_k}) \quad \text{ for all } 1 \le k \le \ell. 
$$
\ee
}\label{eq: recall}
\end{eqnarray}

Thus, for  a reduced sequence $\ii$ of $w_\circ$ adapted to a Dynkin quiver $Q=(\Dynkin,\xi)$, we have the following commutative diagram
of isomorphisms
\begin{align} \label{eq: isos}
 \raisebox{2em}{\xymatrix@R=5ex@C=12ex{   &\calK_\bbA(R\gmod) \ar[dl]_{\Xi_{Q,\le \xi}} \ar[d]_{\Xi_Q} \ar[r]^{\Upomega} & \calA_\bbA(\n) \ar[dl]_{\Psi_Q}  \\
 (\frakK_{q,Q})_{\le \xi} &  \frakK_{q,Q} \ar[l]^{(\cdot)_{\le \xi}} & \calA_{ii} \ar[l]^{\kappa_\ii} \ar[u]^{\varphi_\ii}
}  } 
\end{align}
such that
$$
\Xi_Q \seteq \Psi_Q \circ \Upomega: \calK_\bbA(R\gmod) \isoto \frakK_{q,Q} \quad \text{ and }  \quad 
\Xi_{Q,\le \xi} \seteq (\cdot)_{\le \xi} \circ \Xi_Q:  \calK_\bbA(R\gmod) \isoto (\frakK_{q,Q})_{\le \xi}.  
$$
Moreover, the isomorphism $\Xi_{Q,\le \xi}$ sends 
$ L(w^\ii_k\varpi_k,\varpi_k)$ for $1\le k \le \ell$ to $q^{a/2}\underline{m^{(i)}[p,\xi_i]}$ for some $a \in \Z$, where  $\tii(k)=(i,p) \in \Gamma^Q_0$.
\medskip

\begin{theorem} \label{thm: quantum positive}
Let $Q=(\Dynkin,\xi)$ be a Dynkin quiver. 
For an element $\sfX_{\le \xi} \in (\frakK_{q,Q})_{\le \xi}$, assume that there exists a module $X \in R\gmod$ and $\Xi_{Q,\le \xi}([X])\cong \sfX_{\le \xi}$. Write
$$ \sfX_{\le \xi}= \sum_{m' \in \calM^Q} a[m;m'] \um' \quad \text{ for } a[m:m'] \in \Z[q^{\pm 1/2}].
$$
Then $\sfX_{\le \xi}$ is quantum positive; that is
$$
a[m;m'] \in \N_0[q^{\pm 1/2}]. 
$$
\end{theorem}

\begin{proof} 
By Proposition~\ref{prop: La det} and~\eqref{eq: recall},
$$
\text{ the commuting family } \calL^\ii \text{ is a monoidal cluster of $R\gmod$}. 
$$
Hence Theorem~\ref{thm: Laurent} tells us that
$$
[X] = \sum_{\bsa \in \Z^{\ell}} \alpha(X;\bsa) \prod^{\to}_{r \in [1,\ell]} [L_r^{\ii}]^{\bsa_r}
\quad \text{for some $\alpha(X;\bsa) \in \N_0[q^{\pm1}]$.}
$$

Now let us take a reduced sequence $\ii$ of $w_\circ$ adapted to the Dynkin quiver $Q$. Then, for $1 \le k \le \ell$ with $\tii(k)=(i_k,p_k)$, we have 
$\Xi_{Q,\le \xi}(L^{\ii}_k) \cong  \underline{m^{(i)}[p_k,\xi_{i_k}]} \in \calM^Q_+$. 
Hence 
$$
\sfX_{\le \xi} \cong \Xi_{Q,\le \xi}([X]) = \sum_{\bsa \in \Z^{\ell}} \alpha'(X;\bsa) \prod^{\to}_{r \in [1,\ell]} \underline{m^{(i)}[p_r,\xi_{i_r}]}^{\bsa_r} \quad
\text{ where $ \alpha'(X;\bsa) \cong  \alpha'(X;\bsa)$.}$$
Since $\displaystyle\prod^{\to}_{r \in [1,\ell]} \underline{m^{(i)}[p_r,\xi_{i_r}]}^{\bsa_r} \in \calM_Q$ for any $\bsa \in \Z^\ell$, our assertion follows. 
\end{proof}

The following corollary provides a partial proof of Conjecture~\ref{conj: quatum positive}~\eqref{it: positive}. 

\begin{corollary} \label{cor: pos cor}
Let $m$ be a KR-monomial contained in $\calM^Q_+$ for some Dynkin quiver $Q=(\Dynkin,\xi)$. Then 
$(F_q(m))_{\le \xi}$ is quantum positive. In particular, $F_q(X_{i,p})$ is quantum positive for every $(i,p) \in \hDynkin_0$, and hence, so are the elements in the standard basis $\sfE_q$.
\end{corollary}

\begin{proof}
The first assertion is a direct consequence of Theorem~\ref{thm: auto}~\eqref{it: i-box}, Proposition~\ref{prop: det},~\eqref{eq: recall}
and Theorem~\ref{thm: quantum positive}.
Let us focus on the second assertion which gives
a complete proof of Conjecture~\ref{conj: quatum positive}~\eqref{it: positive} for $m=X_{i,p}$. Recall that 
$i^*=i$ for all $i \in I$ because $\g$ is of non-simply-laced type. Hence 
$$
\Gamma^Q_0 = \{ (i,p) \in \hDynkin_0 \ |  \ \xi_{i} -\sfh < p \le \xi_i \}. 
$$
By the first assertion, 
$(F_q(X_{i,\xi_{i}-\sfh+2} ))_{\le \xi}$ is quantum positive as an element in $\N_0[q^{\pm 1/2}][\tX_{i,p}^{\pm 1}]_{(i,p)\in \Gamma^Q_0}$.
On the other hand, Theorem~\ref{thm: KR-poly}~\eqref{it: range of KR} implies that
$$(F_q(X_{i,\xi_{i}-\sfh+2} ))_{\le \xi} = F_q(X_{i,\xi_{i}-\sfh+2}) - \underline{X_{i,\xi_i+2}^{-1}}.$$
Hence our assertion follows from the automorphism $\sfT_r$ in Remark~\ref{rem: Tr}. 
\end{proof}

\begin{example} \label{ex: quantum positive} 
\bna
\item Let $\g$ be of type $B_4$. Then we have
\begin{align*}
F^{B_4}_q(X_{1,0}) & =  q \tX_{1,0} + q \tX_{2,1}*\tX_{1,2}^{-1} + q \tX_{3,2}*\tX_{2,3}^{-1} + q^{2} \tX_{4,3}^{2}*\tX_{3,4}^{-1} + (q^{-1} + q) \tX_{4,3}*\tX_{4,5}^{-1} \\
           &  \hspace{25ex} + q^{2} \tX_{3,4}*\tX_{4,5}^{-2} + q \tX_{2,5}*\tX_{3,6}^{-1} + q \tX_{1,6}*\tX_{2,7}^{-1} + q^{-1} \tX_{1,8}^{-1},    
\end{align*}
 where $q \tX_{1,0} = \underline{X_{1,0}}$ and $q^{-1} \tX_{1,8}^{-1} = \underline{X_{1,8}^{-1}}$. 
\item Let $\g$ be of type $B_3$. Take $Q = (\Dynkin_{B_3},\xi)$ with $\xi_1=\xi_3=0$ and $\xi_2=-1$, and a simple module $X= L(2) \hconv L(1)$ in $R^{B_3}\gmod$, where
$L(1)$ and $L(2)$ are $1$-dimensional cuspidal module with their weights $\al_1$ and $\al_2$, respectively. Note that the image $\sfX$ of $X$ under $\Xi_{Q}$ is neither a fundamental polynomial nor a KR-polynomial,
while $\Xi_{Q}(L(1) \hconv L(2)) \simeq F^{B_3}_q(X_{1,-4})$. 
Then one can compute that
\begin{align*}
& \Xi_{Q,\le \xi}([X]) =\sfX_{\le \xi}  = q \tX_{2,-5}*\tX_{1,0} + q^{4} \tX_{1,-4}*\tX_{3,-4}^{2}*\tX_{2,-3}^{-1}*\tX_{1,0} + q^{4} \tX_{3,-4}^{2}*\tX_{1,-2}^{-1}*\tX_{1,0}  \\
& \qquad + (q + q^{3}) \tX_{1,-4}*\tX_{3,-4}*\tX_{3,-2}^{-1}*\tX_{1,0} + q^{4} \tX_{1,-4}*\tX_{2,-3}*\tX_{3,-2}^{-2}*\tX_{1,0}  \\
& \qquad + (q^{3} + q^{5}) \tX_{3,-4}*\tX_{2,-3}*\tX_{1,-2}^{-1}*\tX_{3,-2}^{-1}*\tX_{1,0} + q^{8} \tX_{2,-3}^{2}*\tX_{1,-2}^{-1}*\tX_{3,-2}^{-2}*\tX_{1,0}  \\ 
& \qquad + q^{3}\tX_{1,-4}*\tX_{1,-2}*\tX_{2,-1}^{-1}*\tX_{1,0} + (q^{2} + q^{4}) \tX_{3,-4}*\tX_{3,-2}*\tX_{2,-1}^{-1}*\tX_{1,0}  \\
& \qquad + (q + q^{5}) \tX_{2,-3}*\tX_{2,-1}^{-1}*\tX_{1,0} + q^{8} \tX_{1,-2}*\tX_{3,-2}^{2}*\tX_{2,-1}^{-2}*\tX_{1,0} + (q + q^{3}) \tX_{3,-4}*\tX_{1,0}*\tX_{3,0}^{-1}  \\
& \qquad + (q^{2} + q^{4}) \tX_{2,-3}*\tX_{3,-2}^{-1}*\tX_{1,0}*\tX_{3,0}^{-1} + (q^{3} + q^{5}) \tX_{1,-2}*\tX_{3,-2}*\tX_{2,-1}^{-1}*\tX_{1,0}*\tX_{3,0}^{-1} \\
& \qquad + q^{4} \tX_{1,-2}*\tX_{1,0}*\tX_{3,0}^{-2},       
\end{align*}
all of whose coefficients are contained in $\Z_{\ge0}[q^{\pm 1/2}]$ and $q \tX_{2,-5}*\tX_{1,0} = \underline{X_{2,-5}X_{1,1}}$.
\ee
\end{example}

\bigskip

\appendix

\section{Examples of substitution formula} \label{sec: substitution formula}
In this appendix, we provide several examples of substitution formulas for $B_2$. 
We consider the situation when $q=1$ for simplicity. Its quantum analogue $\tPsi_{\frakQ,\frakQ'}$ can be calculated in a similar manner.

\medskip
Consider the height functions given by
\begin{align*}
\xi' &: \{ 1,2 \} \to \Z, \quad 1 \mapsto 0, \quad 2 \mapsto -1, \\
\xi &: \{ 1,2 \} \to \Z, \quad 1 \mapsto 0, \quad 2 \mapsto 1.
\end{align*}
Then the following infinite sequences $\ii' = (i'_u)_{u \in \N}, \ii = (i_u)_{u \in \N}$ are of the form in ~\eqref{eq: hii}:
\begin{align*}
    &\ii' =( \hspace{-2ex}\underset{\text{\scalebox{0.8}{reduced word for $w_{\circ}$}}}{\underbrace{  {\color{red}1, 2, 1,2,}}} \hspace{-2ex} \underbrace{{\color{red}1,2,1,2}}, \underbrace{{\color{red}1,2,1,2}},\dots),  \\
    &\ii=(\hspace{-2ex} \underset{\text{\scalebox{0.8}{reduced word for $w_{\circ}$}}}{\underbrace{{\color{red}2,1, 2, 1,}}} \hspace{-2ex} \underbrace{{\color{red}2,1,2,1}}, \underbrace{{\color{red}2,1,2,1,}},\dots).
\end{align*}
Then $\tGamma^\ii$ is the following quiver:
\[
\raisebox{3mm}{
\scalebox{0.9}{\xymatrix@!C=1mm@R=3mm{
(\im\setminus p) &\cdots&-12&-11&-10&-9& -8 & -7 & -6 &-5&-4 &-3& -2 &-1& 0 & 1\\
1&\cdots& \circled{14} \ar[dr]|{\ulcorner 2}&&\circled{12} \ar[ll]\ar[dr]|{\ulcorner 2}&&\circled{10} \ar[ll]\ar[dr]|{\ulcorner 2}&& \circled{8} \ar[ll]\ar[dr]|{\ulcorner 2} &&\circled{6} \ar[ll]\ar[dr]|{\ulcorner 2}
&&\circled{4} \ar[ll]\ar[dr]|{\ulcorner 2} && \circled{2} \ar[ll]\ar[dr]|{\ulcorner 2} &\\
2&\cdots&&\circled{13} \ar@{-}[l]\ar@{=>}[ur]&&\circled{11} \ar[ll]\ar@{=>}[ur]&&\circled{9} \ar[ll]\ar@{=>}[ur]&&\circled{7} \ar[ll]\ar@{=>}[ur] &&\circled{5} \ar[ll]\ar@{=>}[ur]
&& \ar[ll]\circled{3} \ar@{=>}[ur]&& \circled{1}\ar[ll] 
}}}
\]
Note that $\ii$ can be obtained from $\ii'$, and vise versa, by applying the braid moves in the red parts of $\ii$ and $\ii'$ above. Hence, by applying the mutations at the following vertices
\begin{equation}
(2, 1), (1,0), (2, 1), (2,-3), (1,-4), (2,-3), (2, -7),(1,-8),(2,-7), \ldots  \label{eq:B2toB2}
\end{equation}
from left to right to $\tGamma^{\ii}$, we obtain the quiver $\tGamma^{\ii'}$.
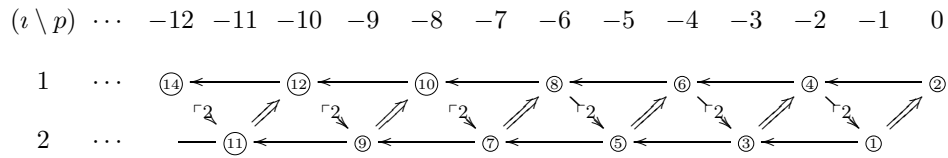
\[
\raisebox{3mm}{
\scalebox{0.9}{\xymatrix@!C=1mm@R=3mm{
(\im\setminus p) &\cdots&-12&-11&-10&-9& -8 & -7 & -6 &-5&-4 &-3& -2 &-1& 0  \\
1&\cdots&\circled{14} \ar[dr]|{\ulcorner 2}&&\circled{12} \ar[ll]\ar[dr]|{\ulcorner 2}&&\circled{10} \ar[ll]\ar[dr]|{\ulcorner 2}&& \circled{8} \ar[ll]\ar[dr]|{\ulcorner 2} &&\circled{6} \ar[ll]\ar[dr]|{\ulcorner 2}
&&\circled{4} \ar[ll]\ar[dr]|{\ulcorner 2} && \circled{2} \ar[ll]  &\\
2&\cdots&&\circled{11} \ar@{-}[l]\ar@{=>}[ur]&&\circled{9} \ar[ll]\ar@{=>}[ur]&&\circled{7} \ar[ll]\ar@{=>}[ur]&& \circled{5} \ar[ll]\ar@{=>}[ur] &&\circled{3} \ar[ll]\ar@{=>}[ur]
&& \ar[ll] \circled{1} \ar@{=>}[ur]  
}}}
\]
Here the correspondence of labelling of vertices is given by
\begin{equation} \label{eq:vertexcorrep2}
\begin{aligned}
   (1, -4m) & \mapsto   (1, -4m), \hspace{8.8ex}
    (1, -2-4m)\mapsto   (1, -2-4m),\\
    (2, 1-4m)&\mapsto  (2,-1-4m), \qquad
    (2, -1-4m)\mapsto   (2,-3-4m).
\end{aligned}
\end{equation}
Let us write the initial cluster variable for $\ii$ as $\{ Z_{i,p} \seteq Z^\xi_{i,p} \}$ and the one for $\ii'$ as $\{Z'_{i,p} \seteq Z^{\xi'}_{i,p} \}$. Similarly, we write
the generators of $\calX_{q,\le \xi}$ as $\{ X_{i,p} \}$
and the generators of $\calX_{q,\le \xi'}$ as $\{ X'_{i,p} \}$.

Using Keller's mutation applet, the isomorphism $\widehat{\bstau}^*(Z'_{(i, p)})$ is computed as follows:
\begin{align*}
Z'_{1,-4m} &\mapsto  \left( Z_{2,-4m+3}^2 Z_{1,-4m}^2 + 2 Z_{2,-4m+3} Z_{1,-4m+2} Z_{1,-4m} Z_{2,-4m-1} \right.  \\
     & \hspace{5ex} \left. + Z_{1,-4m+2}^2 Z_{2,-4m-1}^2 + Z_{1,-4m+2} Z_{2,-4m+1}^2 Z_{1,-4m-2}\right)  /  ( Z_{2,-4m+1}^2 Z_{1,-4m}),  \allowdisplaybreaks \\
Z'_{2,-4m-1} & \mapsto \dfrac{Z_{2,-4m+3} Z_{1,-4m} Z_{2,-4m-1} + Z_{1,-4m+2} Z_{2,-4m-1}^2 + Z_{2,-4m+1}^2 Z_{1,-4m-2}}{Z_{2,-4m+1} Z_{1,-4m}},   \allowdisplaybreaks\\
Z'_{1,-4m-2} & \mapsto Z_{1,-4m-2}, \qquad Z'_{2,-4m-3} \mapsto Z_{2,-4m-1}.
\end{align*}
Hence
\begin{align*}
&\tPsi_{\xi', \xi}(m^{(i)}[p, \xi'_{i}]) \\ &  = \begin{cases}
m^{(1)}[-4m-2,0] & \text{ if } (i,p) = (1, -4m-2), \\
m^{(2)}[-4m-1,1] & \text{ if } (i,p) = (2, -4m-3), \\
 \left(X_{2,-4m+1}^{-2}X_{1,-4m}  + 2 X_{2,-4m-1}X_{2,-4m+1}^{-1} \right. \\
\hspace{12ex} \left. + X_{1,-4m}^{-1}X_{2,-4m-1}^2   + X_{1,-4m-2}\right) m^{(1)}[-4m+2,0]          & \text{ if } (i,p) = (1, -4m), \\
(X_{2,-4m+1}^{-1} +    X_{1,-4m}^{-1}X_{2,-4m-1}  + X_{1,-4m-2}X_{2,-4m-1}^{-1})m^{(2)}[-4m-1,1]                                       & \text{ if } (i,p) = (2, -4m-1).
\end{cases}
\end{align*}
This calculation implies that, for $(i, p)\in {}^{\xi'[-2]}\hDynkin_0$,
\begin{align*}
&\tPsi_{\xi', \xi}(X'_{i, p})=\tPsi_{\xi', \xi}(m^{(i)}[p, \xi'_{i}])\tPsi_{\xi', \xi}(m^{(i)}[p+2, \xi'_{i}])^{-1} \\
& = \begin{cases}
\left( X_{1,-4m-2}^{-1} X_{2,-4m+1}^{-2} + 2 X_{2,-4m-1} X_{2,-4m+1}^{-1} X_{1,-4m}^{-1} X_{1,-4m-2}^{-1} \right.  \\
   \hspace{20ex} \left. + X_{1,-4m}^{-2}X_{1,-4m-2}^{-1} X_{2,-4m-1}^2 + X_{1,-4m}^{-1}\right)^{-1}  & \text{ if } (i,p) = (1,-4m-2), \\
 \left(X_{1,-4m}X_{2,-4m+1}^{-2}+2 X_{2,-4m-1}X_{2,-4m+1}^{-1} +X_{1,-4m}^{-1}X_{2,-4m-1}^2   + X_{1,-4m-2}\right)      & \text{ if } (i,p) = (1,-4m), \\
(X_{2,-4m+1}^{-1} +    X_{1,-4m}^{-1}X_{2,-4m-1}  + X_{1,-4m-2}X_{2,-4m-1}^{-1})^{-1}   & \text{ if } (i,p) = (2,-4m-3), \\
(X_{2,-4m-1}  +    X_{1,-4m}^{-1}X_{2,-4m-1}^2 X_{2,-4m+1} + X_{1,-4m-2}X_{2,-4m+1})   & \text{ if } (i,p) = (2,-4m-1),
\end{cases}
\end{align*}
which shows the substitution formula representing $T_2$ on $\mathfrak{K}(B_2)$. Then we can observe the following:
\begin{align*}
 &\tPsi_{\frakQ,\frakQ}|_{q=1}(X_{1, -4m-2}) \tPsi_{\frakQ,\frakQ}|_{q=1}(X_{1, -4m}) = X_{1, -4m}X_{1, -4m-2}, \\
 &\tPsi_{\frakQ,\frakQ}|_{q=1}(X_{2, -4m-3}) \tPsi_{\frakQ,\frakQ}|_{q=1}(X_{2, -4m-1}) = X_{2, -4m+1} X_{2, -4m-1}.
\end{align*}

\noindent 
(I) Note that we have
\begin{align} \label{eq: B2 to B2 itself}
L_{q=1}(X_{2,-7}) = X_{2,-7} + X_{1,-6} X_{2,-5}^{-1} + X_{2, -5} X_{1,-4}^{-1} +   X_{2,-3}^{-1}
\end{align}
(see \cite{JLO1} for more details). Here, we set $L_{q=1}(m) \seteq \ev_{q=1}(L_q(m))$.  We set
\begin{enumerate}
\item $\reced{1} = X_{2,-3}^{-1}+X_{1,-4}^{-1}X_{2,-5}+ X_{1,-6}X_{2,-5}^{-1} \leftrightarrow (X_{2,-7}^{-1})'$,
\item $\reced{2} = X_{2,-5}+ X_{1,-4}^{-1}X_{2,-3}X_{2,-5}^{2}+ X_{1,-6}X_{2,-3} \leftrightarrow (X_{2,-5})'$,
\item $\reced{3} = X_{1,-6}^{-1}X_{2,-3}^{-2}+2 X_{1,-4}^{-1}X_{1,-6}^{-1}X_{2,-3}^{-1}X_{2,-5} + X_{1,-4}^{-2}X_{1,-6}^{-1}X_{2,-5}^{2}+ X_{1,-4}^{-1} \leftrightarrow (X_{1,-6}^{-1})'$,
\item $\reced{4} = X_{1,-4}X_{2,-3}^{-2}+ 2X_{2,-3}^{-1}X_{2,-5} + X_{1,-4}^{-1}X_{2,-5}^{2} + X_{1,-6} \leftrightarrow (X_{1,-4})'$,
\item $\reced{5} = X_{2,1}^{-1}+X_{1,0}^{-1}X_{2,-1}+ X_{1,-2}X_{2,-1}^{-1} \leftrightarrow (X_{2,-3}^{-1})'$.
\end{enumerate}
Note that 
$$
\reced{1} \times X_{2,-3}X_{2,-5} =  \reced{2}  \quad \text{ and } \quad \reced{3} \times X_{1,-4}X_{1,-6} =  \reced{4}.
$$
Then the first three terms of ~\eqref{eq: B2 to B2 itself} become
\begin{align} \label{eq: B2 to B2 itself 2}
\dfrac{ \reced{3} \times X_{2,-3}X_{2,-5} + 1 + \reced{2}^2 \times   X_{1,-4}^{-1}X_{1,-6}^{-1}   }{\reced{2}  \times  \reced{3}}.
\end{align}

Let us compute each term in numerator: 
\begin{align*}
& {\rm (i)} \ \ \reced{3} \times X_{2,-3}X_{2,-5}  = (X_{1,-6}^{-1}X_{2,-3}^{-2}+2 X_{1,-4}^{-1}X_{1,-6}^{-1}X_{2,-3}^{-1}X_{2,-5} + X_{1,-4}^{-2}X_{1,-6}^{-1}X_{2,-5}^{2}+ X_{1,-4}^{-1} )  \times X_{2,-3}X_{2,-5}    \\
& \ \  =  X_{1,-6}^{-1}X_{2,-3}^{-1}X_{2,-5}  +2 X_{1,-4}^{-1}X_{1,-6}^{-1} X_{2,-5}^2 + X_{1,-4}^{-2}X_{1,-6}^{-1}X_{2,-3}X_{2,-5}^{3}+ X_{1,-4}^{-1}X_{2,-3}X_{2,-5}, \\
& {\rm (ii)} \ \  1, \\
& {\rm (iii)} \ \  \reced{2}^2 \times   X_{1,-4}^{-1}X_{1,-6}^{-1}  =   (X_{2,-5}+ X_{1,-4}^{-1}X_{2,-3}X_{2,-5}^{2}+ X_{1,-6}X_{2,-3} )^2 \times  X_{1,-4}^{-1}X_{1,-6}^{-1}  \\
&  \ \ = \left(  X_{2,-5}^2+ X_{1,-4}^{-2}X_{2,-3}^2X_{2,-5}^{4}+ X_{1,-6}^2X_{2,-3}^2 \right. \\
& \hspace{10ex} \left. + 2X_{1,-4}^{-1}X_{2,-3}X_{2,-5}^{3} +   2X_{1,-6}X_{2,-3} X_{2,-5}  +  2 X_{1,-4}^{-1}X_{1,-6}X_{2,-3}^2X_{2,-5}^{2} \right)   \times  X_{1,-4}^{-1}X_{1,-6}^{-1}  \\
&  \ \  =  X_{1,-4}^{-1}X_{1,-6}^{-1} X_{2,-5}^2+ X_{1,-4}^{-3}X_{1,-6}^{-1}X_{2,-3}^2X_{2,-5}^{4}+ X_{1,-4}^{-1}X_{1,-6}^{1} X_{2,-3}^2 \\
& \hspace{20ex}  + 2X_{1,-4}^{-2}X_{1,-6}^{-1}X_{2,-3}X_{2,-5}^{3} +   2 X_{1,-4}^{-1}X_{2,-3} X_{2,-5}  +  2 X_{1,-4}^{-2}X_{2,-3}^2X_{2,-5}^{2}.
\end{align*}
On the other hand, the denominator is
\begin{align*}
& \reced{2}  \times  \reced{3} = (X_{2,-5}+ X_{1,-4}^{-1}X_{2,-3}X_{2,-5}^{2}+ X_{1,-6}X_{2,-3} ) \\
& \hspace{15ex} \times (X_{1,-6}^{-1}X_{2,-3}^{-2}+2 X_{1,-4}^{-1}X_{1,-6}^{-1}X_{2,-3}^{-1}X_{2,-5} + X_{1,-4}^{-2}X_{1,-6}^{-1}X_{2,-5}^{2}+ X_{1,-4}^{-1}) \\
& = X_{1,-6}^{-1}X_{2,-3}^{-2}X_{2,-5}+2 X_{1,-4}^{-1}X_{1,-6}^{-1}X_{2,-3}^{-1}X_{2,-5}^2 + X_{1,-4}^{-2}X_{1,-6}^{-1}X_{2,-5}^{3}+ X_{1,-4}^{-1}X_{2,-5} \\ 
&  \hspace{3ex} + X_{1,-4}^{-1}X_{1,-6}^{-1}X_{2,-3}^{-1}X_{2,-5}^{2} +2 X_{1,-4}^{-2}X_{1,-6}^{-1}X_{2,-5}^{3}+ X_{1,-4}^{-3}X_{1,-6}^{-1}X_{2,-3} X_{2,-5}^{4}+  X_{1,-4}^{-2}X_{2,-3}X_{2,-5}^{2}\\ 
&  \hspace{6ex} +    X_{2,-3}^{-1}+2 X_{1,-4}^{-1}X_{2,-5} + X_{1,-4}^{-2}X_{2,-3}X_{2,-5}^{2}+ X_{1,-4}^{-1}X_{1,-6}X_{2,-3}. 
\end{align*}
Thus~\eqref{eq: B2 to B2 itself 2} becomes
$$\dfrac{ \reced{3} \times X_{2,-3}X_{2,-5} + 1 + \reced{2}^2 \times   X_{1,-4}^{-1}X_{1,-6}^{-1}   }{\reced{2}  \times  \reced{3}} =X_{2,-3}.$$  Consequently, \eqref{eq: B2 to B2 itself} is the same as
\begin{align} \label{eq: amazing}
X_{2,1}^{-1}+X_{1,0}^{-1}X_{2,-1}+ X_{1,-2}X_{2,-1}^{-1}  + X_{2,-3} = L_{q=1}(X_{2,-3}) =\frakD_q|_{q=1}(L_{q=1}(X_{2,-7})).
\end{align}

\noindent 
(II)  Now let us observe
\begin{align} \label{eq: B2 to B2 itself 1,4}
L_{q=1}(X_{1,-4}) =  X_{1,-4} + X_{2,-3}^{2}X_{1,-2}^{-1} + 2 X_{2,-3}X_{2, -1}^{-1} +   X_{1,-2}X_{2,-1}^{-2} +   X_{1,0}^{-1}.
\end{align}

We set
\begin{enumerate}
\item $\reced{1} = X_{2,-1}+ X_{1,0}^{-1}X_{2,1}X_{2,-1}^{2}+ X_{1,-2}X_{2,1} \leftrightarrow (X_{2,-1})'$,
\item $\reced{2} = X_{1,-4}X_{2,-3}^{-2}+ 2X_{2,-3}^{-1}X_{2,-5} + X_{1,-4}^{-1}X_{2,-5}^{2} + X_{1,-6} \leftrightarrow (X_{1,-4})'$,
\item $\reced{3} = X_{1,-2}^{-1}X_{2,1}^{-2}+2 X_{1,0}^{-1}X_{1,-2}^{-1}X_{2,1}^{-1}X_{2,-1} + X_{1,0}^{-2}X_{1,-2}^{-1}X_{2,-1}^{2}+ X_{1,0}^{-1} \leftrightarrow (X_{1,-2}^{-1})'$,
\item $\reced{4} = X_{2,1}^{-1}+X_{1,0}^{-1}X_{2,-1}+ X_{1,-2}X_{2,-1}^{-1} \leftrightarrow (X_{2,-3}^{-1})'$,
\item $\reced{5} = X_{1,0}X_{2,1}^{-2}+ 2X_{2,1}^{-1}X_{2,-1} + X_{1,0}^{-1}X_{2,-1}^{2} + X_{1,-2} \leftrightarrow (X_{1,0})'$.
\end{enumerate}
Note that
\begin{align*}
\reced{3} \times X_{1,-2} X_{1,0} =\reced{5}  \quad \text{ and } \quad \reced{4} \times X_{2,-1} X_{2,1} =\reced{1}.
\end{align*}
 
Thus~\eqref{eq: B2 to B2 itself 1,4} becomes
\begin{align}
& \reced{2} + \dfrac{\reced{3} \times  X_{2,-1}^2 X_{2,1}^2}{\reced{1}^2} + 2 \times \dfrac{X_{2,-1} X_{2,1}}{ \reced{1}^2} + \dfrac{X_{1,-2} X_{1,0}}{\reced{1}^2 \times \reced{5} } + \dfrac{1}{ \reced{5}}  \nonumber \\
& = \reced{2} + \dfrac{\reced{3} \times  X_{2,-1}^2 X_{2,1}^2  \times \reced{5} + 2  \times \reced{5} \times X_{2,-1} X_{2,1}+ X_{1,-2} X_{1,0} +\reced{1}^2  }{\reced{1}^2 \times \reced{5}}.  \label{eq: the last four terms}
\end{align}
After a miraculous reduction of the fraction, the second term in~\eqref{eq: the last four terms} become $X_{1,-2}^{-1}$  as follows: Here are the four terms in
the numerator  
\begin{align*}
& {\rm (i)} \ \  2  \times \reced{5} \times X_{2,-1} X_{2,1}   = 2X_{1,0}X_{2,-1}X_{2,1}^{-1} + 4X_{2,-1}^2   + 2X_{1,0}^{-1}X_{2,-1}^{3}  X_{2,1} + 2X_{1,-2} X_{2,-1} X_{2,1},   \\
& {\rm (ii)}  \ \ \reced{1}^2 = X_{2,-1}^2+ X_{1,0}^{-2}X_{2,1}^2X_{2,-1}^{4}+ X_{1,-2}^2X_{2,1}^2   \\ 
& \hspace{10ex} + 2  X_{1,0}^{-1}X_{2,1}X_{2,-1}^{3} + 2  X_{1,-2} X_{2,-1}X_{2,1}  + 2 X_{1,0}^{-1}X_{1,-2} X_{2,1}^2X_{2,-1}^{2},  \\
& {\rm (iii)}  \ \ \reced{3} \times  X_{2,-1}^2 X_{2,1}^2  \times \reced{5}  = ( X_{1,-2}^{-1}X_{2,-1}^2  +2 X_{1,0}^{-1}X_{1,-2}^{-1}X_{2,1}^3X_{2,-1}     + X_{1,0}^{-2}X_{1,-2}^{-1}X_{2,1}^2 X_{2,-1}^{4} + X_{1,0}^{-1}X_{2,-1}^2 X_{2,1}^2 ) \\
& \hspace{25ex} \times  (X_{1,0}X_{2,1}^{-2}+ 2X_{2,1}^{-1}X_{2,-1} + X_{1,0}^{-1}X_{2,-1}^{2} + X_{1,-2}) \\
& =  X_{1,-2}^{-1}X_{1,0}X_{2,1}^{-2}X_{2,-1}^2  +2  X_{1,-2}^{-1}X_{2,1}X_{2,-1}   + X_{1,0}^{-1} X_{1,-2}^{-1}  X_{2,-1}^{4} +   X_{2,-1}^2   \\
& \hspace{3ex}  + 2X_{2,1}^{-1} X_{1,-2}^{-1}X_{2,-1}^3  + 4   X_{1,0}^{-1}X_{1,-2}^{-1}X_{2,1}^2X_{2,-1}^2      + 2  X_{1,0}^{-2}X_{1,-2}^{-1}X_{2,1}  X_{2,-1}^{5} + 2  X_{1,0}^{-1}X_{2,-1}^3 X_{2,1}  \\
& \hspace{6ex}  + X_{1,-2}^{-1} X_{1,0}^{-1}X_{2,-1}^{4}  +2 X_{1,0}^{-2}X_{2,-1}^{3}X_{1,-2}^{-1}X_{2,1}^3      + X_{1,0}^{-3}X_{2,-1}^{6}X_{1,-2}^{-1}X_{2,1}^2   + X_{1,0}^{-2}X_{2,-1}^{4} X_{2,1}^2 \\
& \hspace{9ex}  +  X_{2,-1}^2  +2  X_{1,0}^{-1} X_{2,1}^3X_{2,-1}     + X_{1,0}^{-2} X_{2,1}^2 X_{2,-1}^{4} + X_{1,0}^{-1}X_{1,-2}X_{2,-1}^2 X_{2,1}^2, \\
& {\rm (iv)}  \ \  X_{1,-2} X_{1,0},
\end{align*}
while the denominator is 
\begin{align*} 
& \reced{1}^2  \times \reced{5} =   \\
& \hspace{2ex}  X_{1,0}X_{2,1}^{-2}X_{2,-1}^2+ X_{1,0}^{-1}  X_{2,-1}^{4}+ X_{1,-2}^2 X_{1,0}    + 2   X_{2,1}^{-1} X_{2,-1}^{3} + 2  X_{1,-2} X_{1,0}X_{2,1}^{-1} X_{2,-1}
+ 2   X_{1,-2}  X_{2,-1}^{2} \\ 
& \hspace{4ex} + 2X_{2,1}^{-1} X_{2,-1}^3+ 2  X_{1,0}^{-2}X_{2,1}^3 X_{2,-1}^{3}+ 2 X_{2,-1}X_{1,-2}^2X_{2,1}     +  4   X_{1,0}^{-1} X_{2,1}^{2}X_{2,-1}^{2} + 4    X_{1,-2} X_{2,-1}^2 
\\
& \hspace{6ex} + 4  X_{1,0}^{-1}X_{1,-2} X_{2,1} X_{2,-1}^{3}  +   X_{1,0}^{-1}X_{2,-1}^{4}+ X_{1,0}^{-3} X_{2,1}^2X_{2,-1}^{6}+ X_{1,-2}^2 X_{1,0}^{-1}X_{2,-1}^{2}X_{2,1}^2  \\
&  \hspace{8ex}  + 2  X_{1,0}^{-2} X_{2,1}X_{2,-1}^{5} + 2  X_{1,-2}  X_{1,0}^{-1}X_{2,-1}^{3}X_{2,1}  + 2 X_{1,0}^{-2} X_{1,-2} X_{2,1}^2X_{2,-1}^{4} + X_{1,-2}X_{2,-1}^2 \\
& \hspace{10ex} + X_{1,-2}X_{1,0}^{-2}X_{2,1}^2X_{2,-1}^{4}+ X_{1,-2}^3X_{2,1}^2    + 2  X_{1,-2}X_{1,0}^{-1}X_{2,1}X_{2,-1}^{3} +    2  X_{1,-2}^2 X_{2,-1}X_{2,1} \\
& \hspace{12ex} + 2 X_{1,-2}^2X_{1,0}^{-1}  X_{2,1}^2X_{2,-1}^{2}.
\end{align*} 
Thus~\eqref{eq: B2 to B2 itself 1,4} is transformed into
$$
X_{1,-4}X_{2,-3}^{-2}+ 2X_{2,-3}^{-1}X_{2,-5} + X_{1,-4}^{-1}X_{2,-5}^{2} + X_{1,-6} + X_{1,-2}^{-1} = L_{q=1}(X_{1,-6}).
$$

\section{$G_2$-equivalence} \label{sec: G2 equi}
The sequences of mutations that give the equivalence in~\eqref{eq: the G2 mutation} are listed below:
\begin{align*}
& \mu_{k}^* \mu_{k+1}^* \mu_{k+2}^* \mu_{k}^* \mu_{k+3}^* \mu_{k+1}^* \mu_{k}^* \mu_{k+2}^* \mu_{k+3}^* \mu_{k}^*,
& \mu_{k}^* \mu_{k+1}^* \mu_{k+2}^* \mu_{k}^* \mu_{k+3}^* \mu_{k+1}^* \mu_{k}^* \mu_{k+3}^* \mu_{k+2}^* \mu_{k}^*, \allowdisplaybreaks\\
& \mu_{k}^* \mu_{k+1}^* \mu_{k+2}^* \mu_{k+3}^* \mu_{k}^* \mu_{k+1}^* \mu_{k}^* \mu_{k+2}^* \mu_{k+3}^* \mu_{k}^*,
& \mu_{k}^* \mu_{k+1}^* \mu_{k+2}^* \mu_{k+3}^* \mu_{k}^* \mu_{k+1}^* \mu_{k}^* \mu_{k+3}^* \mu_{k+2}^* \mu_{k}^*,\allowdisplaybreaks\\
& \mu_{k}^* \mu_{k+2}^* \mu_{k+1}^* \mu_{k}^* \mu_{k+3}^* \mu_{k+1}^* \mu_{k}^* \mu_{k+2}^* \mu_{k+3}^* \mu_{k}^*,
& \mu_{k}^* \mu_{k+2}^* \mu_{k+1}^* \mu_{k}^* \mu_{k+3}^* \mu_{k+1}^* \mu_{k}^* \mu_{k+3}^* \mu_{k+2}^* \mu_{k}^*,\allowdisplaybreaks\\
& \mu_{k}^* \mu_{k+2}^* \mu_{k+1}^* \mu_{k+3}^* \mu_{k}^* \mu_{k+1}^* \mu_{k}^* \mu_{k+2}^* \mu_{k+3}^* \mu_{k}^*,
&\mu_{k}^* \mu_{k+2}^* \mu_{k+1}^* \mu_{k+3}^* \mu_{k}^* \mu_{k+1}^* \mu_{k}^* \mu_{k+3}^* \mu_{k+2}^* \mu_{k}^*,\allowdisplaybreaks\\
& \mu_{k+1}^* \mu_{k+2}^* \mu_{k+3}^* \mu_{k+1}^* \mu_{k}^* \mu_{k+1}^* \mu_{k+2}^* \mu_{k}^* \mu_{k+3}^* \mu_{k+1}^*,
& \mu_{k+1}^* \mu_{k+2}^* \mu_{k+3}^* \mu_{k+1}^* \mu_{k}^* \mu_{k+1}^* \mu_{k+2}^* \mu_{k+3}^* \mu_{k}^* \mu_{k+1}^*,\allowdisplaybreaks\\
& \mu_{k+1}^* \mu_{k+2}^* \mu_{k+3}^* \mu_{k+1}^* \mu_{k}^* \mu_{k+2}^* \mu_{k+1}^* \mu_{k}^* \mu_{k+3}^* \mu_{k+1}^*,
& \mu_{k+1}^* \mu_{k+2}^* \mu_{k+3}^* \mu_{k+1}^* \mu_{k}^* \mu_{k+2}^* \mu_{k+1}^* \mu_{k+3}^* \mu_{k}^* \mu_{k+1}^*,\allowdisplaybreaks\\
& \mu_{k+1}^* \mu_{k+3}^* \mu_{k+2}^* \mu_{k+1}^* \mu_{k}^* \mu_{k+1}^* \mu_{k+2}^* \mu_{k}^* \mu_{k+3}^* \mu_{k+1}^*,
& \mu_{k+1}^* \mu_{k+3}^* \mu_{k+2}^* \mu_{k+1}^* \mu_{k}^* \mu_{k+1}^* \mu_{k+2}^* \mu_{k+3}^* \mu_{k}^* \mu_{k+1}^*,\allowdisplaybreaks\\
& \mu_{k+1}^* \mu_{k+3}^* \mu_{k+2}^* \mu_{k+1}^* \mu_{k}^* \mu_{k+2}^* \mu_{k+1}^* \mu_{k}^* \mu_{k+3}^* \mu_{k+1}^*,
& \mu_{k+1}^* \mu_{k+3}^* \mu_{k+2}^* \mu_{k+1}^* \mu_{k}^* \mu_{k+2}^* \mu_{k+1}^* \mu_{k+3}^* \mu_{k}^* \mu_{k+1}^*,\allowdisplaybreaks\\
& \mu_{k+2}^* \mu_{k}^* \mu_{k+1}^* \mu_{k+2}^* \mu_{k+3}^* \mu_{k+1}^* \mu_{k+2}^* \mu_{k}^* \mu_{k+3}^* \mu_{k+2}^*,
& \mu_{k+2}^* \mu_{k}^* \mu_{k+1}^* \mu_{k+2}^* \mu_{k+3}^* \mu_{k+1}^* \mu_{k+2}^* \mu_{k+3}^* \mu_{k}^* \mu_{k+2}^*,\allowdisplaybreaks\\
& \mu_{k+2}^* \mu_{k}^* \mu_{k+1}^* \mu_{k+2}^* \mu_{k+3}^* \mu_{k+2}^* \mu_{k+1}^* \mu_{k}^* \mu_{k+3}^* \mu_{k+2}^*,
& \mu_{k+2}^* \mu_{k}^* \mu_{k+1}^* \mu_{k+2}^* \mu_{k+3}^* \mu_{k+2}^* \mu_{k+1}^* \mu_{k+3}^* \mu_{k}^* \mu_{k+2}^*,\allowdisplaybreaks\\
& \mu_{k+2}^* \mu_{k+1}^* \mu_{k}^* \mu_{k+2}^* \mu_{k+3}^* \mu_{k+1}^* \mu_{k+2}^* \mu_{k}^* \mu_{k+3}^* \mu_{k+2}^*
& \mu_{k+2}^* \mu_{k+1}^* \mu_{k}^* \mu_{k+2}^* \mu_{k+3}^* \mu_{k+1}^* \mu_{k+2}^* \mu_{k+3}^* \mu_{k}^* \mu_{k+2}^*,\allowdisplaybreaks\\
& \mu_{k+2}^* \mu_{k+1}^* \mu_{k}^* \mu_{k+2}^* \mu_{k+3}^* \mu_{k+2}^* \mu_{k+1}^* \mu_{k}^* \mu_{k+3}^* \mu_{k+2}^*,
& \mu_{k+2}^* \mu_{k+1}^* \mu_{k}^* \mu_{k+2}^* \mu_{k+3}^* \mu_{k+2}^* \mu_{k+1}^* \mu_{k+3}^* \mu_{k}^* \mu_{k+2}^*,\allowdisplaybreaks\\
& \mu_{k+3}^* \mu_{k+1}^* \mu_{k+2}^* \mu_{k}^* \mu_{k+3}^* \mu_{k+2}^* \mu_{k+3}^* \mu_{k}^* \mu_{k+1}^* \mu_{k+3}^*,
& \mu_{k+3}^* \mu_{k+1}^* \mu_{k+2}^* \mu_{k}^* \mu_{k+3}^* \mu_{k+2}^* \mu_{k+3}^* \mu_{k+1}^* \mu_{k}^* \mu_{k+3}^*,\allowdisplaybreaks\\
& \mu_{k+3}^* \mu_{k+1}^* \mu_{k+2}^* \mu_{k+3}^* \mu_{k}^* \mu_{k+2}^* \mu_{k+3}^* \mu_{k}^* \mu_{k+1}^* \mu_{k+3}^*,
& \mu_{k+3}^* \mu_{k+1}^* \mu_{k+2}^* \mu_{k+3}^* \mu_{k}^* \mu_{k+2}^* \mu_{k+3}^* \mu_{k+1}^* \mu_{k}^* \mu_{k+3}^*,\allowdisplaybreaks\\
& \mu_{k+3}^* \mu_{k+2}^* \mu_{k+1}^* \mu_{k}^* \mu_{k+3}^* \mu_{k+2}^* \mu_{k+3}^* \mu_{k}^* \mu_{k+1}^* \mu_{k+3}^*,
& \mu_{k+3}^* \mu_{k+2}^* \mu_{k+1}^* \mu_{k}^* \mu_{k+3}^* \mu_{k+2}^* \mu_{k+3}^* \mu_{k+1}^* \mu_{k}^* \mu_{k+3}^*,\allowdisplaybreaks\\
& \mu_{k+3}^* \mu_{k+2}^* \mu_{k+1}^* \mu_{k+3}^* \mu_{k}^* \mu_{k+2}^* \mu_{k+3}^* \mu_{k}^* \mu_{k+1}^* \mu_{k+3}^*,
& \mu_{k+3}^* \mu_{k+2}^* \mu_{k+1}^* \mu_{k+3}^* \mu_{k}^* \mu_{k+2}^* \mu_{k+3}^* \mu_{k+1}^* \mu_{k}^* \mu_{k+3}^*.
\end{align*}

\providecommand{\bysame}{\leavevmode\hbox to3em{\hrulefill}\thinspace}
\providecommand{\MR}{\relax\ifhmode\unskip\space\fi MR }
\providecommand{\MRhref}[2]{%
  \href{http://www.ams.org/mathscinet-getitem?mr=#1}{#2}
}
\providecommand{\href}[2]{#2}


\begin{thebibliography}{10}

\bibitem{BZ97}
Arkady Berenstein and Andrei Zelevinsky, \emph{Total positivity in schubert
  varieties}, Commentarii Mathematici Helvetici \textbf{72} (1997), no.~1,
  128--166.

\bibitem{BZ05}
\bysame, \emph{Quantum cluster algebras}, Advances in Mathematics \textbf{195}
  (2005), no.~2, 405--455. \MR{2146350}

\bibitem{B21}
L\'{e}a Bittmann, \emph{A quantum cluster algebra approach to representations
  of simply laced quantum affine algebras}, Mathematische Zeitschrift
  \textbf{298} (2021), no.~3-4, 1449--1485.

\bibitem{Dav18}
Ben Davison, \emph{Positivity for quantum cluster algebras}, Annals of
  Mathematics \textbf{187} (2018), no.~1, 157--219.

\bibitem{DWZ10}
Harm Derksen, Jerzy Weyman, and Andrei Zelevinsky, \emph{Quivers with
  potentials and their representations {II}: applications to cluster algebras},
  Journal of the American Mathematical Society \textbf{23} (2010), no.~3,
  749--790. \MR{2629987}

\bibitem{FZ02}
Sergey Fomin and Andrei Zelevinsky, \emph{Cluster algebras. {I}.
  {F}oundations}, Journal of the American Mathematical Society \textbf{15}
  (2002), no.~2, 497--529. \MR{1887642}

\bibitem{FZ07}
\bysame, \emph{Cluster algebras. {IV}. {C}oefficients}, Compositio Mathematica
  \textbf{143} (2007), no.~1, 112--164. \MR{2295199}

\bibitem{FM01}
Edward Frenkel and Evgeny Mukhin, \emph{Combinatorics of {$q$}-characters of
  finite-dimensional representations of quantum affine algebras},
  Communications in Mathematical Physics \textbf{216} (2001), no.~1, 23--57.
  \MR{1810773}

\bibitem{FR99}
Edward Frenkel and Nicolai Reshetikhin, \emph{The {$q$}-characters of
  representations of quantum affine algebras and deformations of
  {$\mathscr{W}$}-algebras}, Recent developments in quantum affine algebras and
  related topics ({R}aleigh, {NC}, 1998), Contemp. Math., vol. 248, Amer. Math.
  Soc., Providence, RI, 1999, pp.~163--205. \MR{1745260}

\bibitem{FHOO}
Ryo Fujita, David Hernandez, Se-jin Oh, and Hironori Oya, \emph{Isomorphisms
  among quantum {G}rothendieck rings and propagation of positivity}, Journal
  f\"{u}r die Reine und Angewandte Mathematik. (Crelle's Journal) \textbf{785}
  (2022), 117--185. \MR{4402493}

\bibitem{FHOO2}
\bysame, \emph{Isomorphisms among quantum {G}rothendieck rings and cluster
  algebras}, Preprint, \arxiv{2304.02562}, 2023.

\bibitem{FO21}
Ryo Fujita and Se-jin Oh, \emph{Q-data and representation theory of untwisted
  quantum affine algebras}, Communications in Mathematical Physics \textbf{384}
  (2021), no.~2, 1351--1407. \MR{4259388}

\bibitem{GLS13}
Christof Gei\ss, Bernard Leclerc, and Jan Schr\"{o}er, \emph{Cluster structures
  on quantum coordinate rings}, Selecta Mathematica. New Series \textbf{19}
  (2013), no.~2, 337--397. \MR{3090232}

\bibitem{GLS13A}
\bysame, \emph{Factorial cluster algebras}, Documenta Mathematica \textbf{18}
  (2013), 249--274.

\bibitem{GY17}
Ken Goodearl and Millen Yakimov, \emph{Quantum cluster algebra structures on
  quantum nilpotent algebras}, Memoirs of the American Mathematical Society
  \textbf{247} (2017), no.~1169, vii+119. \MR{3633289}

\bibitem{GY20}
\bysame, \emph{Integral quantum cluster structures}, Duke Mathematical Journal
  \textbf{170} (2021), no.~6, 1137--1200.

\bibitem{GHKK}
Mark Gross, Paul Hacking, Sean Keel, and Maxim Kontsevich, \emph{Canonical
  bases for cluster algebras}, J. Amer. Math. Soc. \textbf{31} (2018), no.~2,
  497--608. \MR{3758151}

\bibitem{Her04}
David Hernandez, \emph{Algebraic approach to {$q,t$}-characters}, Advances in
  Mathematics \textbf{187} (2004), no.~1, 1--52. \MR{2074171}

\bibitem{Her05}
\bysame, \emph{Monomials of {$q$} and {$q, t$}-characters for non simply-laced
  quantum affinizations}, Mathematische Zeitschrift \textbf{250} (2005), no.~2,
  443--473. \MR{2178794}

\bibitem{Her06}
\bysame, \emph{The {K}irillov-{R}eshetikhin conjecture and solutions of
  {T}-systems}, Journal f\"ur die Reine und Angewandte Mathematik. (Crelle's
  Journal) \textbf{596} (2006), 63--87. \MR{2254805 (2007j:17020)}

\bibitem{HL10}
David Hernandez and Bernard Leclerc, \emph{Cluster algebras and quantum affine
  algebras}, Duke Mathematical Journal \textbf{154} (2010), no.~2, 265--341.
  \MR{2682185}

\bibitem{HL15}
\bysame, \emph{Quantum {G}rothendieck rings and derived {H}all algebras},
  Journal f\"{u}r die Reine und Angewandte Mathematik. (Crelle's Journal)
  \textbf{701} (2015), 77--126. \MR{3331727}

\bibitem{HL16}
\bysame, \emph{A cluster algebra approach to {$q$}-characters of
  {K}irillov--{R}eshetikhin modules}, Journal of the European Mathematical
  Society \textbf{18} (2016), no.~5, 1113--1159. \MR{3500832}

\bibitem{HO19}
David Hernandez and Hironori Oya, \emph{Quantum {G}rothendieck ring
  isomorphisms, cluster algebras and {K}azhdan--{L}usztig algorithm}, Advances
  in Mathematics \textbf{347} (2019), 192--272. \MR{3916871}

\bibitem{JLO2}
Il-Seung Jang, Kyu-Hwan Lee, and Se-jin Oh, \emph{{B}raid group action on
  quantum virtual {G}rothendieck ring through constructing presentations},
  (2023), Preprint, \arxiv{2305.19471}v2.

\bibitem{JLO1}
\bysame, \emph{Quantization of virtual grothendieck rings and their structure
  including quantum cluster algebras}, Communications in Mathematical Physics
  \textbf{405} (2024), no.~7, 173.

\bibitem{KKK2}
Seok-Jin Kang, Masaki Kashiwara, and Myungho Kim, \emph{Symmetric quiver
  {H}ecke algebras and {$R$}-matrices of quantum affine algebras, {II}}, Duke
  Mathematical Journal \textbf{164} (2015), no.~8, 1549--1602.

\bibitem{KKKO15}
Seok-Jin Kang, Masaki Kashiwara, Myungho Kim, and Se-jin Oh, \emph{Simplicity
  of heads and socles of tensor products}, Compositio Mathematica \textbf{151}
  (2015), no.~2, 377--396.

\bibitem{KKKO18}
\bysame, \emph{Monoidal categorification of cluster algebras}, Journal of the
  American Mathematical Society \textbf{31} (2018), no.~2, 349--426.
  \MR{3758148}

\bibitem{K91}
Masaki Kashiwara, \emph{On crystal bases of the q-analogue of universal
  enveloping algebras}, Duke Mathematical Journal \textbf{63} (1991), no.~2,
  465--516.

\bibitem{KK19}
Masaki Kashiwara and Myungho Kim, \emph{{L}aurent phenomenon and simple modules
  of quiver {H}ecke algebras}, Compositio Mathematica \textbf{155} (2019),
  no.~12, 2263--2295. \MR{4016058}

\bibitem{KKOP18}
Masaki Kashiwara, Myungho Kim, Se-jin Oh, and Euiyong Park, \emph{Monoidal
  categories associated with strata of flag manifolds}, Advances in Mathematics
  \textbf{328} (2018), 959--1009.

\bibitem{KKOP20}
\bysame, \emph{Monoidal categorification and quantum affine algebras},
  Compositio Mathematica \textbf{156} (2020), no.~5, 1039--1077.

\bibitem{KKOP21A}
\bysame, \emph{{B}raid group action on the module category of quantum affine
  algebras}, Proceedings of the Japan Academy, Series A, Mathematical Sciences
  \textbf{97} (2021), no.~3, 13--18.

\bibitem{KKOP21}
\bysame, \emph{{L}ocalizations for quiver {H}ecke algebras}, Pure and Applied
  Mathematics Quarterly \textbf{17} (2021), no.~4, 1465--1548.

\bibitem{KKOP23}
\bysame, \emph{Laurent family of simple modules over quiver {H}ecke algebras},
  Compositio Mathematica \textbf{160} (2024), no.~8, 1916--1940.

\bibitem{KKOP2}
\bysame, \emph{Monoidal categorification and quantum affine algebras {II}},
  Inventiones mathematicae \textbf{236} (2024), no.~2, 837--924.

\bibitem{KO19}
Masaki Kashiwara and Se-jin Oh, \emph{Categorical relations between {L}anglands
  dual quantum affine algebras: doubly laced types}, Journal of Algebraic
  Combinatorics \textbf{49} (2019), 401--435.

\bibitem{KO22}
\bysame, \emph{The {$(q, t) $}-{C}artan matrix specialized at {$ q= 1$} and its
  applications}, Mathematische Zeitschrift \textbf{303} (2023), no.~42.

\bibitem{KP18}
Masaki Kashiwara and Euiyong Park, \emph{Affinizations and ${R}$-matrices for
  quiver {H}ecke algebras}, Journal of the European Mathematical Society
  \textbf{20} (2018), no.~5, 1161--1193.

\bibitem{KL1}
Mikhail Khovanov and Aaron~D. Lauda, \emph{A diagrammatic approach to
  categorification of quantum groups {$I$}}, Representation Theory of the
  American Mathematical Society \textbf{13} (2009), no.~14, 309--347.

\bibitem{KL2}
\bysame, \emph{A diagrammatic approach to categorification of quantum groups
  {$II$}}, Transactions of the American Mathematical Society (2011),
  2685--2700.

\bibitem{Kimura12}
Yoshiyuki Kimura, \emph{Quantum unipotent subgroup and dual canonical basis},
  Kyoto Journal of Mathematics \textbf{52} (2012), no.~2, 277--331.
  \MR{2914878}

\bibitem{L90}
George Lusztig, \emph{Canonical bases arising from quantized enveloping
  algebras}, {J}ournal of the {A}merican {M}athematical {S}ociety \textbf{3}
  (1990), no.~2, 447--498.

\bibitem{L91}
\bysame, \emph{Quivers, perverse sheaves, and quantized enveloping algebras},
  {J}ournal of the {A}merican {M}athematical {S}ociety \textbf{4} (1991),
  no.~2, 365--421.

\bibitem{LusztigBook}
\bysame, \emph{Introduction to quantum groups}, Progress in Mathematics, vol.
  110, Birkh\"{a}user Boston, Inc., Boston, MA, 1993. \MR{1227098}

\bibitem{Nak03}
Hiraku Nakajima, \emph{{$t$}-analogs of {$q$}-characters of
  {K}irillov-{R}eshetikhin modules of quantum affine algebras}, Representation
  Theory. An Electronic Journal of the American Mathematical Society \textbf{7}
  (2003), 259--274 (electronic). \MR{1993360}

\bibitem{Nak04}
\bysame, \emph{Quiver varieties and {$t$}-analogs of {$q$}-characters of
  quantum affine algebras}, Annals of Mathematics. Second Series \textbf{160}
  (2004), no.~3, 1057--1097. \MR{2144973}

\bibitem{OS19}
Se-jin Oh and Travis Scrimshaw, \emph{Categorical relations between {L}anglands
  dual quantum affine algebras: exceptional cases}, Communications in
  Mathematical Physics \textbf{368} (2019), 295--367.

\bibitem{Pla13}
Pierre-Guy Plamondon, \emph{Generic bases for cluster algebras from the cluster
  category}, International Mathematics Research Notices \textbf{2013} (2013),
  no.~10, 2368--2420.

\bibitem{Qin21}
Fan Qin, \emph{Dual canonical bases and quantum cluster algebras}, Preprint
  \arxiv{2003.13674} (2020).

\bibitem{R08}
Rapha{\"e}l Rouquier, \emph{2-{K}ac-{M}oody algebras},  (2008), Preprint,
  \arxiv{0812.5023}.

\bibitem{SST2018}
Ben Salisbury, Adam Schultze, and Peter Tingley, \emph{Combinatorial
  descriptions of the crystal structure on certain pbw bases}, Transformation
  Groups \textbf{23} (2018), 501--525.

\bibitem{Tran}
Thao Tran, \emph{{$F$}-polynomials in quantum cluster algebras}, Algebr.
  Represent. Theory \textbf{14} (2011), no.~6, 1025--1061. \MR{2844755}

\bibitem{VV02}
Michela Varagnolo and Eric Vasserot, \emph{Standard modules of quantum affine
  algebras}, Duke Mathematical Journal \textbf{111} (2002), no.~3, 509--533.

\bibitem{VV03}
\bysame, \emph{Perverse sheaves and quantum {G}rothendieck rings}, Studies in
  memory of {I}ssai {S}chur ({C}hevaleret/{R}ehovot, 2000), Progress in
  Mathematics, vol. 210, Birkh\"{a}user Boston, Boston, MA, 2003, pp.~345--365.
  \MR{1985732}

\end{thebibliography}
\end{document}